\documentclass[hidelinks]{article}
\usepackage[utf8]{inputenc}
\usepackage{needspace}
\usepackage{amsrefs, amsmath, amsfonts, amssymb, amsthm, latexsym, hyperref, comment,csquotes}
\usepackage[margin=1in]{geometry}
\usepackage{enumerate,yfonts,eucal,mathtools}
\usepackage{xcolor}
\usepackage{csquotes}
\usepackage[symbol]{footmisc}
\usepackage{bbm}
\usepackage{soul}
\usepackage{cancel}
\usepackage{enumitem}
\usepackage{wrapfig}
\usepackage{scalerel}[2016/12/29]
\usepackage{caption}
\usepackage{subcaption}
\usepackage{floatrow}
\usepackage{calc}
\usepackage[normalem]{ulem}

\newcommand{\bigcupdot}{\ensuremath{\mathop{\makebox[-2pt]{\hspace{11pt}{\(\cdot\)}}\bigcup}}}
\DeclareFloatVCode{largevskip}%
{\vskip 10pt}
\providecommand{\keywordss}[1]
{
  \small	
  \textbf{\textit{Keywords---}} #1
}

\newtheorem{thmmain}{Theorem}
\newtheorem{thm}{Theorem}[section]

\newtheorem{lem}[thm]{Lemma}
\newtheorem{rmk}{Remark}
\newtheorem*{rmk*}{Remark}
\newtheorem{prop}[thm]{Proposition}
\newtheorem{clm}[thm]{Claim}

\newtheorem{obs}[thm]{Observation}
\newtheorem*{obs*}{Observation}

\newtheorem{cor}[thm]{Corollary}
\newtheorem*{cor*}{Corollary}

\newtheorem{problem}{Problem}

\makeatletter
\newcommand{\vast}{\bBigg@{4}}
\newcommand{\Vast}{\bBigg@{5}}
\makeatother

\newcommand{\anc}{A}

\newcommand{\FF}{{\mathcal F}}
\newcommand{\TT}{{\mathcal T}}

\newcommand{\II}{{\mathcal I}}

\newcommand{\RR}{\mathcal R}
\newcommand{\R}{\mathbb R}

\newcommand{\N}{\mathbb N}
\newcommand{\Z}{\mathbb Z}

\newcommand{\p}{\varphi}

\newcommand{\ct}{\tau}

\renewcommand{\H}{{H}}
\newcommand{\h}{{h}}

\newcommand{\sca}{0.6}

\newcommand{\smallsca}{0.2}

\newcommand{\xmargin}{0.2}

\newcommand{\ZeiSign}{\scaleobj{0.7}{\uparrow\mkern-6mu\nearrow}}
\newcommand{\ZhpSign}{\scaleobj{0.7}{\nwarrow\mkern-6mu\uparrow\mkern-6mu\nearrow}}

\newcommand{\ZZ}{{\Z \boxtimes \Z}}
\newcommand{\Zhp}{\Z^2_{\ZhpSign}}

\newcommand{\Zei}{\Z^2_{\ZeiSign}}
\newcommand{\Ztr}{\Z^2_{\triangle}}

\newcommand{\tms}{\N}

\newcommand{\origin}{\beta}

\newcommand{\rng}{r}

\newcommand{\spr}{\chi}

\DeclareMathOperator{\dis}{\Omega}

\newcommand{\tempd}{\tilde{d}}

\newcommand{\ppo}{\Phi}

\newcommand{\lowphiseg}{\Bar{\seg}}

\newcommand{\GG}{\mathcal{G}}
\newcommand{\area}{\mathcal{A}}
\newcommand{\trap}{\TT}

\newcommand{\str}{q}

\newcommand{\cint}{2} 
\newcommand{\cR}{3}

\newcommand{\dispath}[2]{b_{#1}^{#2}}

\newcommand{\rad}{\rho}

\newcommand{\gray}[1]{\textcolor{gray}{#1}}

\newcommand{\seg}{{\mathcal S}}

\DeclarePairedDelimiter\floor{\lfloor}{\rfloor}

\newcommand{\ind}{\mathbbm 1}

\newcommand{\f}{{\tilde{f}}}

\newcommand{\horint}{\mathcal{L}}

\usepackage[colorinlistoftodos]{todonotes}

\makeatletter
\DeclareRobustCommand{\cev}[1]{%
  \mathpalette\do@cev{#1}%
}
\newcommand{\do@cev}[2]{%
  \fix@cev{#1}{+}%
  \reflectbox{$\m@th#1\vec{\reflectbox{$\fix@cev{#1}{-}\m@th#1#2\fix@cev{#1}{+}$}}$}%
  \fix@cev{#1}{-}%
}
\newcommand{\fix@cev}[2]{%
  \ifx#1\displaystyle
    \mkern#23mu
  \else
    \ifx#1\textstyle
      \mkern#23mu
    \else
      \ifx#1\scriptstyle
        \mkern#22mu
      \else
        \mkern#22mu
      \fi
    \fi
  \fi
}

\makeatother
\usepackage{setspace}
\makeatletter
\newcommand{\subjclass}[2][1991]{%
  \let\@oldtitle\@title%
  \gdef\@title{\@oldtitle\footnotetext{#1 \emph{Mathematics subject classification.} #2}}%
}
\newcommand{\keywords}[1]{%
  \let\@@oldtitle\@title%
  \gdef\@title{\@@oldtitle\footnotetext{\emph{Key words and phrases.} #1.}}%
}
\makeatother

\usepackage{graphicx}
\usepackage{tikz}
\usetikzlibrary{decorations.markings}
\usepackage{rotating} \usepackage{fontawesome} \usepackage[misc]{ifsym} \usepackage{ifthen}
\usepackage{graphicx}
\usepackage{tikz-3dplot}
\usepackage{xstring}
\usepackage{xcolor}
\usetikzlibrary{calc}

\usetikzlibrary{patterns}
\tikzset{
        hatch distance/.store in=\hatchdistance,
        hatch distance=10pt,
        hatch thickness/.store in=\hatchthickness,
        hatch thickness=1pt
}
\makeatletter
\pgfdeclarepatternformonly[\hatchthickness]{my horizontal lines}
{\pgfpointorigin}{\pgfqpoint{100pt}{1pt}}{\pgfqpoint{100pt}{3pt}}
{
    \pgfsetcolor{\tikz@pattern@color}
    \pgfsetlinewidth{\hatchthickness}
    \pgfpathmoveto{\pgfqpoint{0pt}{0.5pt}}
    \pgfpathlineto{\pgfqpoint{100pt}{0.5pt}}
    \pgfusepath{stroke}
}

\pgfdeclarepatternformonly[\hatchthickness]{my vertical lines}
{\pgfpointorigin}{\pgfqpoint{1pt}{100pt}}{\pgfqpoint{3pt}{100pt}}
{
    \pgfsetcolor{\tikz@pattern@color}
    \pgfsetlinewidth{\hatchthickness}
    \pgfpathmoveto{\pgfqpoint{0.5pt}{0pt}}
    \pgfpathlineto{\pgfqpoint{0.5pt}{100pt}}
    \pgfusepath{stroke}
}

\pgfdeclarepatternformonly[\hatchthickness]{my north east lines}
{\pgfqpoint{-1pt}{-1pt}}{\pgfqpoint{4pt}{4pt}}{\pgfqpoint{3pt}{3pt}}
{
    \pgfsetcolor{\tikz@pattern@color}
    \pgfsetlinewidth{1pt}
    \pgfpathmoveto{\pgfqpoint{0pt}{0pt}}
    \pgfpathlineto{\pgfqpoint{3.1pt}{3.1pt}}
    \pgfusepath{stroke}
}

\pgfdeclarepatternformonly[\hatchthickness]{my north west lines}
{\pgfqpoint{-1pt}{-1pt}}{\pgfqpoint{4pt}{4pt}}{\pgfqpoint{3pt}{3pt}}
{
    \pgfsetcolor{\tikz@pattern@color}
    \pgfsetlinewidth{1.2pt}
    \pgfpathmoveto{\pgfqpoint{0pt}{3pt}}
    \pgfpathlineto{\pgfqpoint{3.1pt}{-0.1pt}}
    \pgfusepath{stroke}
}

\tikzset{plane/.style n args={3}{insert path={#1 -- ++ #2 -- ++ #3 -- ++ ($-1*#2$) -- cycle}},
unit xy plane/.style={plane={#1}{(\scaling,0,0)}{(0,\scaling,0)}},
unit xz plane/.style={plane={#1}{(\scaling,0,0)}{(0,0,\scaling)}},
unit yz plane/.style={plane={#1}{(0,\scaling,0)}{(0,0,\scaling)}},
get projections/.style={insert path={let \p1=(\scaling,0,0),\p2=(0,\scaling,0)  in 
[/utils/exec={\pgfmathtruncatemacro{\xproj}{sign(\x1)}\xdef\xproj{\xproj}
\pgfmathtruncatemacro{\yproj}{sign(\x2)}\xdef\yproj{\yproj}
\pgfmathtruncatemacro{\zproj}{sign(cos(\tdplotmaintheta))}\xdef\zproj{\zproj}}]}},
pics/unit cube/.style={code={
\path[get projections];
\draw (0,0,0) -- (\scaling,\scaling,\scaling);
\ifnum\zproj=-1
 \path[3d cube/every face,3d cube/xy face,unit xy plane={(0,0,0)}]; 
\fi
\ifnum\yproj=1
 \path[3d cube/every face,3d cube/yz face,unit yz plane={(\scaling,0,0)}]; 
\else
 \path[3d cube/every face,3d cube/yz face,unit yz plane={(0,0,0)}]; 
\fi
\ifnum\xproj=1
 \path[3d cube/every face,3d cube/xz face,unit xz plane={(0,0,0)}]; 
\else
 \path[3d cube/every face,3d cube/xz face,unit xz plane={(0,\scaling,0)}]; 
\fi
\ifnum\zproj>-1
 \path[3d cube/every face,3d cube/xy face,unit xy plane={(0,0,\scaling)}]; 
\fi
}},
3d cube/.cd,
xy face/.style={fill=lightgray},
xz face/.style={fill=darkgray},
yz face/.style={fill=gray},
every face/.style={draw,very thick}
}

\colorlet{lightgrayX}{black!17!white}
\colorlet{grayX}{black!26!white}
\colorlet{darkgrayX}{black!47!white}
\colorlet{lightblackX}{black!90!white}
\colorlet{verylightgrayX}{black!2!white}

\usetikzlibrary{decorations.pathreplacing,calligraphy}

\usepackage{setspace}
\title{The Containment Game: between the Firefighter Problem and Conway's Angel Problem}
\author{Ohad Noy Feldheim\thanks{Hebrew University of Jerusalem Israel, \texttt{ohad.feldheim@mail.huji.ac.il}.}
\and Itamar Israeli\thanks{Hebrew University of Jerusalem Israel, \texttt{itamar.israeli@mail.huji.ac.il}.}}

\newcommand{\NE}{{\text{\resizebox{7pt}{!}{$\uparrow\kern-0.52em{\nearrow}$}}}}
\newcommand{\WNE}{{\text{\resizebox{9pt}{!}{$\nwarrow\kern-0.52em{\uparrow}\kern-0.32em{\nearrow}$}}}}
\begin{document}
\maketitle
\begin{abstract}
The Containment Game is a two-player perfect-information game, initialised with a finite set of occupied vertices in an infinite connected graph $G$. On the $t$-th turn, the first player, called \emph{Spreader}, replaces the occupied set with a collection of $g(t)$ vertices adjacent to it; the second player, called \emph{Container}, then removes $q$ unoccupied vertices from the graph. If the spreading process continues indefinitely, Spreader wins; otherwise, Container wins. 
For $g\equiv 1$ it is equivalent to the $k$-king variant of Conway's angel problem, where the angel movements must trail along unblocked paths, while for $g=\infty$ this game reduces to a solitaire game for Container, known as the \emph{Firefighter Problem}. 

We introduce the game and study it on the strongly connected two dimensional integer lattice $\Z\boxtimes\Z$. Writing $q(G,g)$ for the set of $q$ values for which Container wins against a given $g(t)$, we study the minimal asymptotics of $g(t)$ for which $q(G,g)=q(G,\infty)$, that is, for which defeating Spreader is as hard as winning the Firefighter Problem solitaire.
We show, by providing explicit winning strategies for Container, a sub-linear upper bound $g(t)=O(t^{6/7})$ and a lower bound of $g(t)=\Omega(t^{1/2})$.

\end{abstract}

\keywordss{Conway's angel problem, Firefighter problem, combinatorial games, pursuit-evasion games, k-King problem}

\section{Introduction}\label{section: introduction}

\emph{Conway's Angel Problem} is a classical two player game played between an \emph{Angel} of speed $k$ and a \emph{Devil} of strength $q$ (first appearing in~\cite{berlekampwinningways}*{Section 19}). In this game, an angel is located at the origin of the graph $G=\Z\boxtimes\Z$. In every turn the Angel player may reposition the angel to any unblocked vertex at distance $k$ or less from its current location. The Devil player then deletes from the graph $q$ vertices, unoccupied by the angel. The Devil player wins if the angel is contained in a finite set; while the Angel player wins if the angel may continue to occupy new vertices indefinitely.  If the graph distance is taken with respect to the remainder graph, the problem is known as the $k$-King problem \cite{kutz2004angel}. In both cases, a Devil of strength $q=1$ is able to win against an Angel of speed $1$ (see e.g. Kutz \cite{kutz2004angel}) while losing to an Angel of speed $2$ (see Kloster \cite{kloster2007solution}).

The $(G,\str)$-\emph{Firefighter Problem}, introduced by Hartnell~\cite{hartnell1995firefighter}, is a single real parameter solitaire combinatorial game played on a graph $G$, obtained by considering the $k$-king problem played by the Devil against a non-deterministic angel. Namely, the Devil is tasked with deleting a deterministic sequence of vertices, in order to win against all possible angel strategies at once. 
Formally, the game could be described as follows. Given a finite starting set of \emph{burning}  vertices $B_0\subset V$ (representing initial positions of the Angel), at every turn $t\in\{1,2,\dots\}$, the so called \emph{Firefighter} player chooses an arbitrary collection of at most $\floor{t\str} - \floor{(t-1)\str}$ non-burning vertices and deletes them from the graph. Then, neighbours of burning vertices in the remainder graph become burning. If no new burning vertices are generated, we say that the \emph{fire} is \emph{contained}.
The Firefighter player \emph{wins} $(G,\str)$ if  able to contain any finite initial burning set.
We denote 
\[\str_G:=\{\str\ :\ \text{$(G,\str)$ is a Firefighter win for every finite $B_0$}\}.\] 
Since winning is monotone in $\str$, this set always forms an infinite ray. The problem has been extensively studied on the integer lattice with nearest neighbour adjacency, either with respect to $L^\infty$ or to $L^1$, graphs which are denoted by $\Z\boxtimes \Z$ and $\Z\,\square\,\Z$, respectively. Following the work of Fogarty \cite{fogarty2003catching}, Hod and the first author \cite{feldheim20133} showed that $\str_{\Z\boxtimes \Z}=(3,\infty)$ and $\str_{\Z\square \Z}=(3/2,\infty)$. 

Naturally, the Firefighter player needs greater strength to contain the fire than the Devil player requires for winning against an angel of speed $1$. This gives rise to the following informal question, which motivates our study:

\begin{displayquote}
``How much non-determinism is required for the angel to be as hard to contain as the fire?''
\end{displayquote}

To pose this formally, we introduce the following generalisation of both games, which we call the \emph{Containment Game}, played between two players: \emph{Spreader} and \emph{Container}.
This is also a generalisation of the $k$-firefighter game suggested by Devlin and Hartke~\cite{develin2007fire} and studied on finite graphs by Bonato, Messinger and Pra\l{}at~\cite{bonato2012fighting}.

Given a \emph{spreading} function $g:\N \to \N$
we define the $(G=(V,E),\str,g)$ \emph{Containment Game} as the following two-player combinatorial game. We initialise the game with a finite initial set $B_0$ of \emph{occupied} vertices, and an empty set $\FF_0=\emptyset$ of deleted vertices.
At every time-step $t \ge 1$, Container deletes a collection of at most $\floor{t\str} - \floor{(t-1)\str}$ non-occupied vertices from the graph, adding them to $\FF_{t-1}$ to form $\FF_t$. Then, Spreader selects $g(t)$ vertices at distance at most one from $B_{t-1}$ in the remainder graph $G_t:=V\setminus \FF_t$, which form $B_t$.
If eventually no new vertices are occupied -- Container wins; otherwise -- Spreader wins. We call $(G,q,g)$ a \emph{Container win} if Container can win the game against any Spreader strategy, and any finite initial set of occupied vertices $B_0$.

Given $g$ we define 
\[q(G,g):=\{q\ :\ \text{$(G,q,g)$ is a Container win}\},\]
observing that this is an infinite (open or closed) ray, and that $q(G,g')\subset q(G,g)$ for all $g'>g$. We write $(G,q,\infty)$ for the Containment Game without any restriction on the number of vertices selected by Spreader each turn.
It is not difficult to observe that in this case the optimal Spreader strategy is to always extend $B_t$ to all of its available neighbours, and the game played by Container is non other than the Firefighter Problem solitaire. 

Our question of interest could be now phrased as:
\begin{displayquote}
``Which is the minimal $g(t)$ asymptotics for which $q(G,g)=q(G,\infty)$?''
\end{displayquote}
Observe that this is always satisfied for $g=4t+C$ for sufficiently large $C=C(B_0)$, as full spreading never requires more than this number of occupied vertices. The fundamental question, therefore, is whether one can achieve this for $g=ct+C$ for $c<4$ or even for sub-linear $g=o(t)$.

For brevity and conciseness we show our results only for $G=\ZZ$, although analogues for $G=\Z\,\square\,\Z$ seem to follow from the same arguments. 
We also study the game on the following simpler graph $G=\Zei$, which is the sub-graph of $\ZZ$ restricted to $\{(x,y)\ :\  0\le x\le y\}$ equipped with the edge set $\left\{\big((x,y), (x+i,y+1)\big) : 0\le x\le y,\, i\in \{0,1\}\right\}$. This graph captures the essence of our methods and serves as a stepping stone for the study of $\ZZ$.

Our main result is that a sub-linear $g$ can achieve $q(\ZZ,g)=q(\ZZ,\infty)$.

\begin{thmmain}\label{thm: sub-linear win}
Let $G\in\{\Zei,\ZZ\}$.
For all sufficiently large $C > 0$, we have
$q(G,Ct^{6/7})=q(G,\infty)$.
\end{thmmain}

We complement this theorem with the following lower-bound.

\begin{thmmain}\label{thm: root-fire loses}
Let $G\in\{\Zei,\ZZ\}$.
For all $c < \frac{1}{6}$ we have
$q(G,ct^{1/2})\subsetneq q(G,\infty)$.
\end{thmmain}

Note that $q(\Zei,\infty) = q_{\Zei}= (1,\infty)$ and $q(\ZZ,\infty) = q_{\ZZ}=(3,\infty)$, by \cite{fogarty2003catching} and \cite{feldheim20133}, respectively, so that  Theorems~\ref{thm: sub-linear win} and \ref{thm: root-fire loses} are equivalent to the fact that $g(t)=\Theta(t^{6/7})$ is sufficient for Spreader to win against Container of strength $q=3$ ($q=1$) on $\ZZ$ ($\Zei$), while $g(t)=O(t^{1/2})$ is insufficient.



\begin{rmk}\label{rmk: spreading instead of replacing}
    Theorem~\ref{thm: root-fire loses} is actually proven under a stronger-Spreader variation of the game, where Spreader is allowed to retain any previously occupied vertices, i.e. $B_t$ is added to $B_{t-1}$, instead of replacing it.
\end{rmk}

\begin{rmk}\label{rmk: accumulating ff}
Note that rather than asking for the number of deleted vertices up to time-step $t$ to be at most $\str t$ as per \cite{feldheim20133}, we pose the stronger turn-specific restriction that the number of vertices deleted at turn $t$ is at most $\floor{t\str} - \floor{(t-1)\str}$. In the Firefighter Problem, where the fire is deterministic, the two models are equivalent. In the Containment Game, however, a Spreader with a sub-linear spread function cannot win against an accumulating Container of any strength $\str>0$ even on $\Zei$. This can be proved similarly to our proof of Theorem~\ref{thm: root-fire loses}.
\end{rmk}



\subsection{Related work}

\textbf{Conway's Angel problem.}
The origin of Conway's Angel problem is somewhat obscure. Variants are mentioned by Martin Gardner as early as '74~\cite{gardner1974mathematical}, where credit is given to D. Silverman and R. Epstein.
In its current transformation, the problem first appeared in the classical monograph by Berlkamp, Conway and Guy~\cite{berlekampwinningways}*{Section 19}.
Conway~\cite{conway1996angel} later showed that an angel of speed $1$ loses. Several years later, Kutz~\cite{kutz2004angel} showed that an angel of speed $2-\epsilon$ loses for every $\epsilon>0$. In the meantime, the problem of showing that an angel of high enough speed wins acquired some notoriety, until finally in '07 four independent papers by Bowditch~\cites{bowditch2007angel}, Gács~\cites{gacs2007angel},
Kloster~\cites{kloster2007solution} and
Máthé~\cites{mathe2007angel} we able to show it, thus resolving the original problem. In fact, the latter two showed that speed $2$ is sufficient.

\textbf{The Firefighter Problem.} Ever since it's introduction in 1995~\cite{hartnell1995firefighter}, Hartnell's Firefighter Problem has been the subject of study in diverse contexts, as a model for expanding phenomena. Beyond $\ZZ$ and $Z\square Z$ where the problem has been resolved in \cite{feldheim20133}, the case of planar infinite $G$ considered here (where constant $\str$ makes sense) has been studied also for the triangular grid $\Ztr$, where Dean et al.~\cite{dean2021firefighting} improved upon a result of Gavenčiak, Kratochvíl, and Pra\l{}at~\cite{gavenvciak2014firefighting}, showing that a single firefighter per turn plus one additional firefighter at some turn are sufficient to contain a single fire source. It has been conjectured that this bound is tight and $q_{\Ztr}=(1,\infty)$. 
Indeed, by now the problem of finding tight bounds for winning on the triangular and hexagonal lattices gained some notoriety and remains widely open. 
In addition, the problem on a slab $G=\Z\boxtimes\Z\boxtimes[k]$ has been studied by Deutsch, Hod and the first author, who showed $q_G=(3k,\infty)$.

On $\Z^d$, where growing $\str=\str(t)$ is required to win against large starting occupied sets, rough bounds were obtained by Develin and Hartke~\cite{develin2007fire}. The problem has also been studied on Cayley graphs. Dyer, Martinez-Pedroza and Thorne~\cite{dyer2017coarse} showed that the critical growth rate of the critical $\str$ is roughly invariant under quasi-isometries. The same authors have also obtained bounds on the critical growth for the number of firefighters required to contain any finite starting fire on Cayley graphs of polynomial growth in \cite{dean2021firefighting}. This was later improved by 
Amir, Baldasso, and Kozma~\cite{AMIR2020112077}, to obtain bounds that are tight up to a constant. The problem has also been studied for Cayley graphs of exponentially growing groups \cite{MR3997950} and on groups of intermediate growth \cite{amir2022branching}.
Other questions relating the game to group theory have been studied \cites{amir2023fire,MR4549803}. 

The Firefighter Problem has been widely considered also on finite graphs where one wishes either to reduce the number of burning vertices when the process terminates, or to minimise the time it takes to contain the fire. This has been studied mainly from an algorithmic point of view, showing that the problems are NP-hard
\cites{MacGillivray2003Ontheff,finbow2007firefighter,king2010firefighter}, but could be approximated up to a constant factor in polynomial time \cites{adjiashvili2018firefighting,Approx2,hartke2004attempting,HL00,iwaikawa2011improved}.
These results are surveyed in~\cite{finbow2009firefighter}.

\textbf{Precursors of current work.}
The idea of studying a restricted fire model arose in the context of the original Firefighter Problem on finite graphs. After being suggested by Devlin and Hartke~\cite{develin2007fire}, the constant $g$,  $\str=1$ case of the Containment Game suggested here, played on finite $G$, was studied by Bonato, Messinger and Pra\l{}at~\cite{bonato2012fighting}. There the quantity of interest was the expected percentage of surviving vertices under optimal play, when the initial occupied set consists of a single uniformly chosen vertex.

\textbf{Other pursuit games on graphs.} Many other pursuit games on graphs have been studied, both for applicative reasons and as a method of getting insight into graph connectivity properties. We provide references to several prominent examples, somewhat related to the game studied here.
The Burning Number of a Graph~\cite{bonato2016burn} is a solitaire game initialised with an empty burning set, where at every time-step, neighbours of burning vertices are added to the burning set along with a single additional vertex chosen by the player. The goal here is to reduce the number of rounds required to burn all vertices in the graph. Cops and Robbers is another two-player pursuit game, where Container is moving a set of blocked vertices (cops) along the graph's edges trying to catch a finite number of Spreader (robbers) vertices. See~\cite{bonato2011game} for a monograph on this game. 
Finally, Invisible Rabbits~\cite{abramovskaya2016hunt} is an oblivious variant of the cops and robbers game, where the cops' movements are not restricted by the edges of the graph, but Container is oblivious to the location of the robbers and must nevertheless catch them.


\subsection{Open problems}\label{subsection: discussion}

We pose several open problems concerning the Containment Game and its variants. 

Firstly, we would be most interested in the following problem.
\begin{problem}
Close the gap between Theorem~\ref{thm: sub-linear win} and Theorem~\ref{thm: root-fire loses}. 
\end{problem}
\noindent We conjecture that the lower bound in Theorem~\ref{thm: root-fire loses} is tight, so that $q(G,t^{1/2+\epsilon})=q(G,\infty)$ for all $\epsilon>0$. 

In addition, the proof of Theorem~\ref{thm: root-fire loses} merely show  $1\in q(\Zei,\infty)\setminus q(\Zei,Ct^{1/2})$  and   $3\in q(\ZZ,\infty)\setminus q(\ZZ,Ct^{1/2})$.
This leaves us with the following.
\begin{problem}
Is $\inf q(G,t^{1/2-\epsilon}) < \inf q(G,\infty)$ for $G\in \{\Zei,\ZZ\}$ and every $\epsilon>0$? 
\end{problem}

\textbf{Constant expansion.}
Here we have analysed the parameter range for which the Containment Game on a graph $G$ satisfies $q(G, g) = q(G,\infty)$. It would be interesting to consider the more general question of recovering the dependence of $q(G, g)$ on $g$. 
The case of constant expansion, i.e. $g\equiv k$, is of particular interest, both for its tighter relations with Conway's angel problem and for its elegance, pitting a constant power Spreader against a constant power Container.
It can be shown, by means similar to the proof of Theorem~\ref{thm: root-fire loses}, that Container wins the game $(\ZZ,k,3-\tfrac{c}{k})$ for some $c$ independent of $k$. As this bound tends to $3$ as $k$ tends to infinity, the following question stands.
\begin{problem}
What is $\liminf_{k\to\infty} q(\ZZ,k)$? Is it strictly less than $3$?
\end{problem}

\textbf{High dimensions.} 
By~\cite{AMIR2020112077} we know that the critical number of firefighters needed to contain the fire in the strongly connected $\Z^d$ is $\Theta(t^{d-2})$, and recovering the exact constant seems within reach (using techniques from \cite{feldheim20133}). One can ask what is the critical spreading function that would be equivalent to unrestricted spreading in this case and is it asymptotically smaller, namely:
\begin{problem}
Is there $g(t)=o(t^{d-1})$ which is equivalent to unrestricted spreading in the strongly connected $\Z^d$ graph?
\end{problem}

In general, we suspect that for all groups of sub-exponential growth, some $g(t)$ which grows asymptotically slower than the isoperimetric profile should be able to imitate unrestricted fire spreading.

\textbf{Probabilistic variant.} It should also be mentioned that the setting of probabilistic spread where the fire spreads to a neighbour with constant probability at every turn, appears not to have been studied so far. The realistic application of this setting, along with natural limiting shape questions and relations to classical models in particle systems, make this a rather appealing variant to study.

\subsection{Overview of the Spreader strategy}

We begin by outlining Spreader's strategy on $\Zei$. The goal here is to simulate the perimeter of the fire in the Firefighter Problem's setting, retaining the lower bound of~\cite{fogarty2003catching}, using a sparse occupied set. To do so, Spreader will always occupy vertices on a single row $\Z\times\{t\}$ progressing northwards at every time-step. 
In the bulk of this occupied set, unhindered by Container's deletions, these vertices are evenly spaced, being $h_t$ apart, forming a sparse array.
Every such vertex is viewed as simulating the fire in a horizontal segment of length $h_t$ in $\Z\times\{t+h_t\}$.
To allow the occupied set to grow sub-linearly growing, $h_t$ will double, once in a while, by discontinuing the evolution of every second vertex.

The key observation is that this simulation is faithful to the evolution of the fire as analysed in \cite{feldheim20133}, as long as Container deletes in time-step $t$ vertices that are north of the line $\Z\times\{t+h_t\}$. In particular, blocking a single occupied vertex would require Container to use at least $h_t$ deletions. 

Vertices deleted in proximity to the occupied front, however, can block occupied vertices much more efficiently; these are called \emph{disruptions}. Here another observation, simple, but never used before, comes into play. It is highly inefficient for the player in the Firefighter Problem to create local cavities in the fire, as these close after a short while, without any impact on the fire-front's shape. Spreader takes advantage of this in the following way. Whenever a disruption occurs, nearby vertices start spreading to all of their neighbours in the front, acting like the fire for a predetermined time interval. This is sufficient to compensate for the efficient blocking of occupied vertices. In order to avoid trouble near the eastern and western boundaries of the fire-front, vertices close to these boundaries always spread to all of their neighbours in the front. As each of the fully spreading regions returns to the simulative normality after a prescribed number of steps, their number is controlled and the occupied set remains rather sparse.  

We also present a variant of this strategy on the \emph{directed half-plane} $\Zhp = (V,E)$, the sub-graph in which vertices are connected to their immediate neighbours in the north, north-east and north-west. This is then used as a building block in the strategy on $\ZZ$, which closely imitates that used in \cite{feldheim20133}.

For technical reasons, it is more convenient to present this strategy in two steps. Defining at first ``fully spreading regions'' which are a bit large, and then avoiding the occupation of any vertex that is disconnected from the infinite component of the graph.

\subsection{Overview of the proof}
The proof of Theorem~\ref{thm: sub-linear win} in $\Zei$ consists of two parts -- proving that Spreader's strategy is sparse, namely bounding the number of vertices occupied by Spreader at every time-step, and proving that the strategy is winning, namely showing that the occupied set is never empty.

To prove the first part, we bound the total number of excess occupied vertices caused by disruptions. This we do by showing that only a bounded number of disruptions can occur at each time-step, and that every such disruption results in a bounded increase in the number of occupied vertices for a limited number of time-steps.

To prove that the strategy is winning, we use a potential-type argument. We assign each segment in Spreader's array with a potential, summing the amount of fire it currently simulates and the number of deleted vertices played above it. We show that, uninterrupted, this potential never decreases for all segments, and moreover grows by at least $1$ globally.

Disruptions in Spreader's array may, however, decrease this potential by reducing the amount of simulated fire, using only a small number of deleted vertices. We show that this loss, however, must always be matched by a later increase in potential in another segment. We hence introduce an additional \emph{debt} function, which is used for tracking these computations. We show that these debts, once obtained, are always reduced to zero within a bounded number of time-steps and before the occupied set is eliminated.

\subsection{Notations}
We use the convention $\N := \{0,1,2,\dots\}$.
For any $x\in \R$ we denote $x_+ := \max\{x,0\}$. 
$\Delta g_t$ denotes the backwards difference of $g$, defined by $\Delta g_t:= g_t - g_{t-1}$ or $\Delta g_t:= g_t \setminus g_{t-1}$.
Throughout, we naturally extend functions from elements to sets by following the convention that if $g$ has numerical output, we set $g(A)=\sum_{x\in A} g(x)$, while if it produces sets, we set $g(A)=\bigcup_{x\in A}g(x)$.
Finally, we follow the standard convention that $\max \emptyset = -\infty$ and $\min \emptyset = +\infty$.

\subsection{Outline of the paper} 
Section~\ref{section: eighth plane} is dedicated to the proof of Theorem~\ref{thm: sub-linear win} on $\Zei$, including the definition of Spreader's strategy and its analysis. In Section~\ref{section: directed half plane} we introduce the game on the directed half-plane graph $\Zhp$ and generalise the results of Section~\ref{section: eighth plane} to this setting. In Section~\ref{section: plane} we compose a winning strategy on the entire $\ZZ$ graph from the strategy of Section~\ref{section: directed half plane}, completing the proof of Theorem~\ref{thm: sub-linear win}.
Section~\ref{section: upper bounds}, which is independent of the other sections, is dedicated to the proof of Theorem~\ref{thm: root-fire loses}.
The paper is accompanied by a notation table, provided as an appendix.

\section{The Eighth Plane}\label{section: eighth plane}

This section is dedicated to proving Theorem~\ref{thm: sub-linear win} for $G=\Zei$. We begin by describing a winning Spreader strategy on the graph, then reduce the theorem to several key propositions.



\textbf{Outline of this section.} 
In Section~\ref{subsection: Zei strategy} we present the winning Spreader strategy, applying one of two local strategies in each segment,  postponing the precise definition of $\spr_t$, the function which decides between them.
In Section~\ref{subsection: fire simulation} we establish several claims, used to control the evolution of the simulated fire.
In Section~\ref{subsection: chi}
we provide the complete definition of $\spr_t$  and derive basic properties of our strategy.
In Section~\ref{subsection: sparsity} we slightly modify our strategy to reduce the number of occupied vertices used, and establish upper bounds on their number.
In Section~\ref{subsecion: strategy analysis} we associate deleted vertices with segments, defining for each segment three key functions: simulated fire count, number of deleted vertices, and debt function used for amortised bookkeeping. Using these we reduce Theorem~\ref{thm: sub-linear win} to several propositions.
Sections~\ref{subsection: analysis} to~\ref{subsection: phi < r} are dedicated to the formal definition of each of these functions, and to the analysis of their evolution, establishing the proposition of Section~\ref{subsecion: strategy analysis}.

\subsection{Winning Spreader strategy}\label{subsection: Zei strategy}

Recalling that the set of (Spreader) occupied vertices at time-step $t$ is denoted by $B_t$, the Spreader strategy will always satisfy $B_t\subset \Z\times\{t\}$. We therefore refer to the row  $\Z\times\{t\}$ as the \emph{front} at time-step $t$. 
Knowing that the firefighter player of strength $\str=1$ loses the Firefighter Problem game on $\Zei$ (by~\cites{fogarty2003catching}), the Spreader strategy strives to imitate the fire evolution using a smaller occupied set, so that each occupied vertex represents a fire segment of size $h_t$ in $\Z\times\{t+h_t\}$, where $h:\N\to\N$ is a monotonically non-decreasing function. To introduce the strategy formally we require several definitions. For convenience we assume throughout $q\ge 1$.


\textbf{Segments.} We partition $\Z$ into intervals of size $h_{t}$ of the form $\{0,\dots,h_{t}-1\}+h_{t}\Z$. The set of segments at time-step $t$ is denoted by $\seg_t$. Given an interval $I\subset \Z$ we denote $\seg_t(I) := \{S\in \seg_t : S\cap I \ne \emptyset\}$.

Segments naturally inherit the order of $\Z$, so that $S_1,S_2\in \seg_t$ satisfy $S_1<S_2$ if for all $s_1\in S_1,s_2\in S_2$ we have $s_1<s_2$.
We denote contiguous closed (similarly, open and half-open) intervals of segments by $[S_1,S_2] := \{S\in\seg_t : S_1 \le S \le S_2\}$, when $S_1,S_2 \in \seg_t$.
Given $x\in\Z$, denote by $S_t(x)$ the unique segment in $\seg_t$ containing $x$, and extend this to $\Z^2$ by setting $S_t((x,y)) := S_t(x)$. 

\textbf{Doubling segments.} For the construction and analysis of the sub-linear Spreader strategy used to establish Theorem~\ref{thm: sub-linear win}, we will allow the size of the segments $h$ to depend on $t$. To make the analysis simpler, we double the value of $h$ once in a while, rather than change it gradually, so that $h : \N\to\N$ satisfies $h_{t}/h_{t-1}\in \{1,2\}$. Given such $h$ we denote $\tms_{i}  = \tms_i(h) := \{t\in \N : h_{t}/h_{t-1} = i\}$ for $i=1,2$, such that $\{0\}\cup \tms_1 \cup \tms_2 = \N$. We refer to $\tms_2$ as \emph{doubling times}.
Throughout, we make the assumption that $h$ is sufficiently large, namely that $h_0>100$.
Denoting $H_t := 11\str h_t^2$,
we make the assumption that doubling times of $h$ do not occur too often, namely,
\begin{equation}\label{eq: doubling condition}
\text{If $t\in \tms_2$, then for all $s\in [t-2H_t,t+2H_{t}]$ we have $s\notin \tms_2$.}
\end{equation}
This will be satisfied by our final choice of $h$ in the proof of Theorem~\ref{thm: sub-linear win}.

Over the course of time smaller dyadic segments in $\seg_s$ joint to form the segments in $\seg_t$ for $t>s$. Thinking of these smaller segments as ancestors of their larger counterparts, given $S\in \seg_t$ we write $\anc_s(S):=\{S'\in \seg_s\ :\ S'\subset S\}$.
For functions $f'$ on $\seg_t$ defined below ($\phi_t(S),d_t(S),F_t(S)$, etc.), for $s<t$ we define,  
$f'_s(S):=\sum_{S'\in \anc_s(S)}f'_s(S')$, if $f'$ yields numerical value, and 
$f'_s(S):=\cup_{S'\in \anc_s(S)}f'_s(S')$ if it produces sets.


\subsubsection{Paths and the evolution of $B_t$}

Our strategy divides $\seg_t$ into two categories. The first category is \emph{spreading segments}, indicated by $\spr_t(S)=1$, the occupied vertices of which spread to all of their neighbours. The second category is \emph{simulative segments}, indicated by $\spr_t(S) = 0$, in which a single occupied vertex simulates the fire on a segment of $\Z\times\{t+h_t\}$. This vertex will be chosen to be the leftmost possible vertex which contains an unblocked path forwards. 
Let us define this notion of a path formally.

A \emph{$t$-path of length $\ell$} is a path $(x_0,y_0),(x_1,y_0+1)\dots,(x_\ell,y_0+\ell) \subset G_t$. We say that such a path is \emph{emanating from the column $x_0$.} Given a set of columns $I$, we denote the set of $t$-paths of length $\ell$ whose vertices are in $I\times \Z$ by $P_t^\ell(I)$.

The fire simulated in $S\times \{t+h_t\}$ corresponds to the end-points of $t$-paths starting from occupied vertices in $B_t(S)$. To show that such fire could indeed be effectively simulated, we must keep track of the number of such paths available, accounting for deleted vertices and for the evolution of the fire. 

Let $t,\ell, s\in\N$ and $A\subset I \times \{s\}$.
Denote $r_t^\ell(A;I):=|\{x\in G_t\ :\ \exists (x_0,\dots,x)\in P_t^\ell(I), x_0\in A|$.
The evolution of $B_t$ is hence as follows.
\begin{equation}\label{eq: Bt def eighth plane}
    B_{t}(S) := (S\times \{t\})\cap
    \begin{cases}
        \left(B_{t-1}+ \{(0,1),(1,1)\}\right) \setminus \FF_{t} & \text{if $\spr_t(S)=1,$}\\
        \min \{x\in B_{t-1} + \{(0,1),(1,1)\}\ :\ \rng_t^{h_t}(\{x\};S) > 0\}
        & \text{if $\spr_t(S)=0$},\\
    \end{cases}
\end{equation}
where the exact definition of $\spr_t(S)$ is provided in the next section.
Given $B_t(S)=B\times \{t\}$, we write $\min B_t(S)=(\min B,t)$ (similarly $\max B_t(S)=(\max B,t)$).

With applications in Section~\ref{section: directed half plane} in mind, we also present the Spreader strategy on the \emph{directed half-plane} $\Zhp = (V,E)$, the sub-graph of $\ZZ$ restricted to
\[V = \{(x,y)\in \Z^2 : y \ge 0\},\,\,\, E = \{\left((x,y), (x+i,y+1)\right) : |x| \le y, i\in\{-1,0,1\}\}.\]
The strategy is almost identical to~\eqref{eq: Bt def eighth plane}, except for the fact that in a spreading segment, the leftmost occupied vertex will spread also northwestwards. Namely,
\begin{equation}\label{eq: Bt def directed half plane}
    B_{t}(S) := (S\times \{t\})\cap
    \begin{cases}
        \big(\left(B_{t-1}+ \{(0,1),(1,1)\}\right) \cup \{\min B_{t-1}+(-1,1)\}\big) \setminus \FF_{t} & \text{if $\spr_t(S)=1,$}\\
        \min \{x\in B_{t-1} + \{(0,1),(1,1)\}\ :\ \rng_t^{h_t}(\{x\};S) > 0\}
        & \text{if $\spr_t(S)=0.$}\\
    \end{cases}
\end{equation}

In order to unify the treatment of $\Zei$ and $\Zhp$, all of our analysis, except for Section~\ref{subsection: sparsity}, will rely only on the following property common to both models. 
\begin{equation}\label{eq: Bt def contains}
\begin{aligned}
    B_t(S) &\supseteq (S\times \{t\}) \cap \left(B_{t-1}+ \{(0,1),(1,1)\}\right) \setminus \FF_{t} \quad &\text{if $\spr_t(S)=1$},\\
    B_t(S) &= (S\times \{t\}) \cap \min \{x\in B_{t-1} + \{(0,1),(1,1)\}\ :\ \rng_t^{h_t}(\{x\};S) > 0\} \quad  &\text{if $\spr_t(S)=0$}.
\end{aligned}
\end{equation}

\subsection{Fire simulation}\label{subsection: fire simulation}
Next, we establish several claims, used to control the evolution of the simulated fire.

Firstly, we observe that at least $r$ vertices must be deleted in order to block out
a set $A$ satisfying $\rng_t^\ell(A;I) \ge r$.

\begin{clm}\label{clm: faithful upwards reduction}
    Let $A\subset [a,b]\times\{s\}$ and $0\le\ell_1<\ell_2$. We have $\rng_t^{\ell_2}(A;[a,b]) \ge \rng_t^{\ell_1}(A;[a,b]) -\big|\FF_t([a,b] \times [s+\ell_1+1,s+\ell_2]) \big|.$
\end{clm}

\begin{proof}
    Extend each of the $\rng_t^{\ell_1}(A;[a,b])$ many $t$-paths starting from vertices in $A$ and leading to the vertices in $[a,b]\times\{s+\ell_1\}$ upwards by vertical columns up to $[a,b]\times\{s+\ell_2\}$. Every such extended path is either a new $t$-path, or
    it must contains a vertex in $\FF_t([a,b] \times [s+\ell_1+1,s+\ell_2])$. As these paths are disjoint in $\FF_t([a,b] \times [s+\ell_1+1,s+\ell_2])$, the claim follows.
\end{proof}

Next, we show that $\rng_t^\ell(A;I) > 0$ is sufficient to guarantee either that $\rng_t^\ell(A;I)$ is large, or the deletion of many vertices in 
$[a,b]\times[s,t+\ell]$.

\begin{clm}\label{clm: accecibility implies faithful}
    Let $[a,b]\times\{t\}\subset \ZZ$. 
    If $\rng_t^\ell(\{(a,t)\}; [a,b]) > 0$ then \[\rng_t^\ell(\{(a,t)\}; [a,b]) \ge \min(\ell,|[a,b]|) - |\FF_t([a,b]\times[t,t+\ell])|.\]
\end{clm}
\begin{proof}
    Denote a $t$-path between $(a,t)$ and $S\times \{t+\ell\}$, by 
    \[P = (x = (p_0,t), (p_1,t+1),\dots,(p_{\ell},t+\ell),\]
    observing that such a path exists, as $\rng_t^\ell(\{(a,t)\}; [a,b]\times\{t\}) > 0$. 
    Writing $w = \min(\ell,|[a,b]|)$, we construct a family of $w$ many paths in $[a,b]\times[t,t+\ell]$, each of which begins with a prefix of $P$, concatenated to a path in $[a,b]\times[t,t+\ell]\setminus P$, disjoint from all others (see the accompanying Figure~\ref{fig:path comb}). This we do as follows.    
    Define a vector $(a_0,\dots,a_{\ell-1}) \in \{0,1\}^{\ell}$ via $a_j := p_{j+1} - p_j$ and $\ell$-tuples $a^1,\dots,a^{w}$ by
    \[a^{k}_j := \begin{cases}
    a_j & \text{if $j < \gamma_k$},\\
    1- a_k & \text{if $j \ge \gamma_k$},
\end{cases} \quad \text{where} \quad \gamma_k := \min \{m\ :\ p_m\text{ is not connected in $\Zei$ to $(p_0+k-1,t+\ell)$} \}.\]
Consider the paths $((p_0,t), (p_0 + a^{k}_0, t+1),\dots, (p_0 + a^{k}_{\ell-1}, t+\ell))$ for $k = 1,\dots,w$.
As these paths are disjoint outside $P$, every deleted vertex in $[a,b]\times[t,t+\ell]$ can block at most one of these paths. Hence \[\rng_t^\ell(\{(a,t)\}; [a,b]\times\{t\}) \ge w - |\FF_t([a,b]\times[t,t+\ell])|. \qedhere\]
\end{proof}

{
\begin{figure}[b!]
    \centering
    \subfloat[\label{fig:leftmost pivot}]{
    \begin{tikzpicture}[scale=\sca*0.8]
        \draw[step=1cm,grayX,thin] (0,0) grid (7,7);
        
        \draw[black, very thick,fill=none] (0,0) rectangle (7,1); 
        \node at (0.5,0.5) {\Fire};

        \draw [line width=0.3mm] (0.5,0.5) -- (3.5,3.5);
        \draw [-stealth, line width=0.3mm] (3.5,3.5) -- (3.5,6.5);
        
        \draw [-stealth, dashed, line width=0.3mm] (0.5,0.5) -- (0.5,6.5);
        \draw [-stealth, dashed, line width=0.3mm] (1.5,1.5) -- (1.5,6.5);
        \draw [-stealth, dashed, line width=0.3mm] (2.5,2.5) -- (2.5,6.5);
        \draw [-stealth, dashed, line width=0.3mm] (3.5,3.5) -- (6.5,6.5);
        \draw [-stealth, dashed, line width=0.3mm] (3.5,4.5) -- (5.5,6.5);
        \draw [-stealth, dashed, line width=0.3mm] (3.5,5.5) -- (4.5,6.5);


    \end{tikzpicture}}
    \hspace{8em}
    \subfloat[\label{fig:middle pivot}]{
    \begin{tikzpicture}[scale=\sca*0.8]
        \draw[step=1cm,grayX,thin] (0,0) grid (7,7);
        
        \draw[black, very thick,fill=none] (0,0) rectangle (7,1); 
        \node at (2.5,0.5) {\Fire};
        
        \draw [line width=0.3mm] (2.5,0.5) -- (2.5,1.5);
        \draw [line width=0.3mm] (2.5,1.5) -- (2.5,2.5);
        \draw [line width=0.3mm] (2.5,2.5) -- (3.5,3.5);
        \draw [line width=0.3mm] (3.5,3.5) -- (3.5,4.5);
        \draw [line width=0.3mm] (3.5,4.5) -- (4.5,5.5);
        \draw [-stealth, line width=0.3mm] (4.5,5.5) -- (5.5,6.5);
        
        \draw [-stealth, dashed, line width=0.3mm] (2.5,2.5) -- (2.5,6.5);
        \draw [-stealth, dashed, line width=0.3mm] (3.5,3.5) -- (6.5,6.5);
        \draw [-stealth, dashed, line width=0.3mm] (3.5,4.5) -- (3.5,6.5);
        \draw [-stealth, dashed, line width=0.3mm] (4.5,5.5) -- (4.5,6.5);


    \end{tikzpicture}}
    
    \caption{Two possible cases for the location of $(a,t)$ within $S$, and the corresponding paths. In each sub-figure, the central path $P$ is indicated by an unbroken line, while the other paths in are indicated by dashed lines.
    }
    \label{fig:path comb}
\end{figure}
}

\textbf{Counting disjoint $t$-paths.} 
While $\rng_t^\ell(A;I)$ is concerned with how many $t$-paths emanate from $A$, it is useful to control the number of deleted vertices needed to block all of them. To this end we now count the maximum number of disjoint paths emanating from a set.

Let $A\subset\Z$. Denote the size of the largest collection of pairwise disjoint $t$-paths of length $\ell$ emanating from $B_t(A)$ by $\dispath{t}{\ell}(A)$,  extending this definition to $A\times \{t\}$ via $\dispath{t}{\ell}(A\times\{t\}) = \dispath{t}{\ell}(A)$, for convenience.

We make three straightforward observations. The first observes two general monotonicity properties of $\dispath{t}{\ell}(A)$.
\begin{obs}\label{obs: pathcol mono}
Let $t, \ell, \ell'\in\N$, and $A,B\subset \Z$. The following hold.
\begin{enumerate}[label=\textnormal{(\alph*)}]
    \item\label{obs: pathcol mono, space} If $A\subseteq B$ then $\dispath{t}{\ell}(A) \le \dispath{t}{\ell}(B) \le \dispath{t}{\ell}(A) + |B\setminus A|$.
    \item\label{obs: pathcol mono, length} If $\ell \le \ell'$ then $\dispath{t}{\ell}(A) \ge \dispath{t}{\ell'}(A)$.
\end{enumerate}
\end{obs}
{
\begin{figure}[t]
     \def\aN{55}
    \def\width{8}
    \colorlet{fireColor}{black!17!white}
    \colorlet{ffColor}{black!65!white}
    \definecolor{ffDisColor}{RGB}{30, 30, 30}
    \centering
    \begin{tikzpicture}[scale=\sca/4]

    \foreach \y in {-4,...,3}
    {
        \fill[fireColor] (0,\y) rectangle (\y+1 + \width,\y+1);
    }


    \def \height{4}
    \foreach \x in {0,...,4,8,9,...,12}{
         \fill[fireColor] (\x,\height) rectangle (\x+1,\height+1);
        }
    \def \height{5}
    \foreach \x in {0,...,4,8,9,...,13}{
         \fill[fireColor] (\x,\height) rectangle (\x+1,\height+1);
        }
    \def \height{6}
    \foreach \x in {0,...,4,8,9,...,14}{
         \fill[fireColor] (\x,\height) rectangle (\x+1,\height+1);
        }
    \def \height{7}
    \foreach \x in {0,...,4,8,9,...,15}{
         \fill[fireColor] (\x,\height) rectangle (\x+1,\height+1);
        }
    
    \def \height{8}

    \foreach \x in {0,1,2,3,4,8,12,13,14,15,16}{
         \fill[fireColor] (\x,\height) rectangle (\x+1,\height+1);
        }
    \def \height{9}
    \foreach \x in {0,1,2,3,4,8,12,13,14,15,16,17}{
         \fill[fireColor] (\x,\height) rectangle (\x+1,\height+1);
        }
    \def \height{10}
    \foreach \x in {0,1,2,3,4,8,12,13,14,15,16,17,18}{
         \fill[fireColor] (\x,\height) rectangle (\x+1,\height+1);
        }
    \def \height{11}
    \foreach \x in {0,1,2,3,4,8,12,13,14,15,16,17,18,19}{
         \fill[fireColor] (\x,\height) rectangle (\x+1,\height+1);
        }
         
    \def \height{12}
    \foreach \x in {0,1,2,3,4,8,12,16,17,18,19,20}{
         \fill[fireColor] (\x,\height) rectangle (\x+1,\height+1);
        }
    \def \height{13}
    \foreach \x in {0,1,2,3,4,8,12,16,17,18,19,20,21}{
         \fill[fireColor] (\x,\height) rectangle (\x+1,\height+1);
        }
    \def \height{14}
    \foreach \x in {0,1,2,3,4,8,12,16,17,18,19,20,21,22}{
         \fill[fireColor] (\x,\height) rectangle (\x+1,\height+1);
        }
    \def \height{15}
    \foreach \x in {0,1,2,3,4,8,12,16,17,18,19,20,21,22,23}{
         \fill[fireColor] (\x,\height) rectangle (\x+1,\height+1);
        }

    \def \height{16}
    \foreach \x in {0,1,2,3,4,8,12,16,20,21,22,23,24}{
         \fill[fireColor] (\x,\height) rectangle (\x+1,\height+1);
        }
    \def \height{17}
    \foreach \x in {0,1,2,3,4,8,12,16,20,21,22,23,24,25}{
         \fill[fireColor] (\x,\height) rectangle (\x+1,\height+1);
        }
    \def \height{18}
    \foreach \x in {4}
    {
        \draw[black] (\x+\xmargin,\height+\xmargin) -- (\x+1-\xmargin,\height+1-\xmargin);
        \draw[black] (\x+\xmargin,\height+1-\xmargin) -- (\x+1-\xmargin,\height+\xmargin);
    }
    \foreach \x in {0,1,2,3,5,9,12,16,20,21,22,23,24,25,26}{
         \fill[fireColor] (\x,\height) rectangle (\x+1,\height+1);
        }
    \def \height{19}
    \foreach \x in {0,1,2,3,5,10,12,16,20,21,22,23,24,27}{
         \fill[fireColor] (\x,\height) rectangle (\x+1,\height+1);
        }
    \foreach \x in {8,9,25,26}
    {
        \draw[black] (\x+\xmargin,\height+\xmargin) -- (\x+1-\xmargin,\height+1-\xmargin);
        \draw[black] (\x+\xmargin,\height+1-\xmargin) -- (\x+1-\xmargin,\height+\xmargin);
    }
    \def \height{20}
    \foreach \x in {0,1,2,3,5,10,16,20,24,25,27,28}{
         \fill[fireColor] (\x,\height) rectangle (\x+1,\height+1);
        }
    \def \height{21}
    \foreach \x in {0,1,2,3,5,10,16,20,24,25,26,27,28,29}{
         \fill[fireColor] (\x,\height) rectangle (\x+1,\height+1);
        }
    \def \height{22}
    \foreach \x in {0,1,2,3,5,10,16,20,24,25,26,27,28,29,30}{
         \fill[fireColor] (\x,\height) rectangle (\x+1,\height+1);
        }
    \def \height{23}
    \foreach \x in {0,1,2,3,5,10,16,20,24,25,26,27,28,29,30,31}{
         \fill[fireColor] (\x,\height) rectangle (\x+1,\height+1);
        }
    \foreach \x in {12,13,14,15}
    {
        \draw[black] (\x+\xmargin,\height+\xmargin) -- (\x+1-\xmargin,\height+1-\xmargin);
        \draw[black] (\x+\xmargin,\height+1-\xmargin) -- (\x+1-\xmargin,\height+\xmargin);
    }
         
    \def \height{24}
    \foreach \x in {0,1,2,3,4,5,6,10,11,20,21,24,25,26,27,28,29,30,31,32}{
         \fill[fireColor] (\x,\height) rectangle (\x+1,\height+1);}
    \foreach \x in {16,17}{ 
        \draw[black, ultra thick] (\x+\xmargin,\height+\xmargin) -- (\x+1-\xmargin,\height+1-\xmargin);
        \draw[black, ultra thick] (\x+\xmargin,\height+1-\xmargin) -- (\x+1-\xmargin,\height+\xmargin);}
  \def \height{25}
    \foreach \x in {0,1,2,3,4,5,6,7,10,11,12,24,25,26,27,28,29,30,31,32,33}{
         \fill[fireColor] (\x,\height) rectangle (\x+1,\height+1);}
    \foreach \x in {20,21,22}{ 
        \draw[black, ultra thick] (\x+\xmargin,\height+\xmargin) -- (\x+1-\xmargin,\height+1-\xmargin);
        \draw[black, ultra thick] (\x+\xmargin,\height+1-\xmargin) -- (\x+1-\xmargin,\height+\xmargin);}
   \def \height{26}
    \foreach \x in {0,1,2,3,4,5,6,7,8,10,11,12,13,24,25,26,27,28,29,30,31,32,33,34}{
         \fill[fireColor] (\x,\height) rectangle (\x+1,\height+1);}
    \def \height{27}
    \foreach \x in {0,1,2,3,4,5,6,7,8,9,10,11,12,13,14,24,25,26,27,28,29,30,31,32,33,34,35}{
         \fill[fireColor] (\x,\height) rectangle (\x+1,\height+1);}
    \def \height{28}
    \foreach \x in {0,1,...,15,24,25,...,36}{
         \fill[fireColor] (\x,\height) rectangle (\x+1,\height+1);}
    \def \height{29}
    \foreach \x in {0,1,...,16,24,25,...,37}{
         \fill[fireColor] (\x,\height) rectangle (\x+1,\height+1);}
    \def \height{30}
    \foreach \x in {0,1,...,17,24,25,...,38}{
         \fill[fireColor] (\x,\height) rectangle (\x+1,\height+1);}
    \def \height{31}
    \foreach \x in {0,1,...,18,24,25,...,39}{
         \fill[fireColor] (\x,\height) rectangle (\x+1,\height+1);}
    \def \height{32}
    \foreach \x in {0,1,...,19,24,25,...,40}{
         \fill[fireColor] (\x,\height) rectangle (\x+1,\height+1);}
    \def \height{33}
    \foreach \x in {0,1,...,20,24,25,...,41}{
         \fill[fireColor] (\x,\height) rectangle (\x+1,\height+1);}
    \def \height{34}
    \foreach \x in {0,1,...,3,4,8,9,...,21,24,25,...,42}{
         \fill[fireColor] (\x,\height) rectangle (\x+1,\height+1);}
    \def \height{35}
    \foreach \x in {0,1,2,3,4,8,12,...,40,41,42,43}{
         \fill[fireColor] (\x,\height) rectangle (\x+1,\height+1);}
    
     \def \height{36}
    \foreach \x in {0,1,2,3,4,8,12,...,28,29,32,33,36,37,40,41,42,43,44}{
         \fill[fireColor] (\x,\height) rectangle (\x+1,\height+1);}
    \def \height{37}
    \foreach \x in {0,1,2,3,4,8,12,...,28,29,30,32,33,34,36,37,38,40,41,42,43,44,45}{
         \fill[fireColor] (\x,\height) rectangle (\x+1,\height+1);}

    \def \height{38}
    \foreach \x in {0,1,2,3,4,8,12,...,28,29,30,...,46}{
         \fill[fireColor] (\x,\height) rectangle (\x+1,\height+1);}

    \def \height{39}
    \foreach \x in {0,1,2,3,4,8,12,...,28,29,30,...,47}{
         \fill[fireColor] (\x,\height) rectangle (\x+1,\height+1);}

    \def \height{40}
    \foreach \x in {0,1,2,3,4,8,12,...,28,29,30,...,48}{
         \fill[fireColor] (\x,\height) rectangle (\x+1,\height+1);}

    \def \height{41}
    \foreach \x in {0,1,2,3,4,8,12,...,28,29,30,...,49}{
         \fill[fireColor] (\x,\height) rectangle (\x+1,\height+1);}

    \def \height{42}
    \foreach \x in {0,1,2,3,4,8,12,...,28,29,30,...,50}{
         \fill[fireColor] (\x,\height) rectangle (\x+1,\height+1);}

    \def \height{43}
    \foreach \x in {0,1,2,3,4,8,12,...,28,29,30,...,31}{
         \fill[fireColor] (\x,\height) rectangle (\x+1,\height+1);}
     \foreach \x in {32,...,51}{ 
         \draw[black] (\x+\xmargin,\height+\xmargin) -- (\x+1-\xmargin,\height+1-\xmargin);
         \draw[black] (\x+\xmargin,\height+1-\xmargin) -- (\x+1-\xmargin,\height+\xmargin);}
    
    
    \def \height{44}
     \foreach \x in {0,1,2,3,4,8,16,24,25,28,29,30,...,32}{
         \fill[fireColor] (\x,\height) rectangle (\x+1,\height+1);}
    \def \height{45}
     \foreach \x in {0,...,5,8,16,24,25,28,29,30,...,33}{
         \fill[fireColor] (\x,\height) rectangle (\x+1,\height+1);}
    \def \height{46}
     \foreach \x in {0,...,6,8,16,24,25,26,28,29,30,...,34}{
         \fill[fireColor] (\x,\height) rectangle (\x+1,\height+1);}
    \def \height{47}
     \foreach \x in {0,...,8,16,24,25,26,27,28,29,30,...,35}{
         \fill[fireColor] (\x,\height) rectangle (\x+1,\height+1);}
    \def \height{48}
     \foreach \x in {0,...,8,16,24,25,...,36}{
         \fill[fireColor] (\x,\height) rectangle (\x+1,\height+1);}
    \def \height{49}
     \foreach \x in {0,...,8,16,24,25,...,37}{
         \fill[fireColor] (\x,\height) rectangle (\x+1,\height+1);}
     \def \height{50}
     \foreach \x in {0,...,8,16,24,25,...,38}{
         \fill[fireColor] (\x,\height) rectangle (\x+1,\height+1);}
    \def \height{51}
     \foreach \x in {0,...,8,16,24,25,...,39}{
         \fill[fireColor] (\x,\height) rectangle (\x+1,\height+1);}

    \def \height{52}
     \foreach \x in {0,...,8,16,24,32,33,...,40}{
         \fill[fireColor] (\x,\height) rectangle (\x+1,\height+1);}
    \def \height{53}
     \foreach \x in {0,...,8,16,24,32,33,...,41}{
         \fill[fireColor] (\x,\height) rectangle (\x+1,\height+1);}
    \def \height{54}
     \foreach \x in {0,...,8,16,24,32,33,...,42}{
         \fill[fireColor] (\x,\height) rectangle (\x+1,\height+1);}
    \def \height{55}
     \foreach \x in {0,...,8,16,24,32,33,...,43}{
         \fill[fireColor] (\x,\height) rectangle (\x+1,\height+1);}

    \node at (3+\width,-1) {I};
    \node at (11+\width,7) {II};
    \node at (23+\width,19) {III};
    \node at (33+\width,29) {IV};
    \node at (43+\width,39) {V};
    \node at (53+\width,49) {VI};
    
    \foreach \y in {-4,...,\aN}
    {
        \draw[step=1cm,gray,ultra thin] (0,\y) grid (\y+1+\width,\y+1);
    }
    

    \foreach \x in {0,4,...,\width}
    {
        \draw[black, thick] (\x, 0) -- (\x, \aN-12+1);
    }
    \draw[black, thick] (0, -4) -- (0, \aN-12+1);
    
    \foreach \y in {-4,0,4,...,36}
    {
        \draw[black, thick] (\y+\width, \y) -- (\y+\width, \aN-12+1);
    }

    \foreach \x in {0,8,...,\width}
    {
        \draw[black, thick] (\x, \aN-12+1) -- (\x, \aN+1);
    }
    
    \foreach \y in {0,8,...,\aN}
    {
        \draw[black, thick] (\y+\width, \y) -- (\y+\width, \aN+1);
    }
    
    \foreach \y in {-4,4,16,24,36,44}
    {
        \draw[black, dashed,thick] (0, \y) -- (\y+\width+1, \y);
    }

    \end{tikzpicture}
    \caption{
    Illustration of different phenomena of the main Spreader strategy. The illustrated phenomena are presented in different phases, marked by horizontal dashed lines, while the partition into segments is marked by vertical lines. 
    In \textbf{Phase I}, all segments, being close to the boundary, are spreading.
    In \textbf{Phase II}, the inner segments become simulative, and their occupied vertices consolidate to single occupied vertices. 
    In \textbf{Phase III}, the occupied vertices move in response to deleted vertices, those possessing an unblocked path move along it, while those who do not cease playing. All deleted vertices of this phase were removed at least $h$ vertices away from the front, hence simulative segments in their vicinity remain simulative.
    In \textbf{Phase IV}, two disruptions occur, one after another. This causes the segments in their vicinity to become spreading for a while, until they consolidate once again at the end of the step.
    In \textbf{Phase V}, the right boundary is blocked by non-disruptive deleted vertices. This causes the segments at the boundary to become spreading, such that at the end of the step there is still a bulk of occupied vertices at the boundary. 
    \textbf{Phase VI} opens with a doubling. In the boundary this causes some segments to become spreading, while the inner simulative segments have half of their vertices discontinued.}
    \label{fig:fire general evolution}
\end{figure}
}
The second is concerned with monotonicity with respect to time in regions consisting of spreading segments. It is obtained straightforwardly from~\eqref{eq: Bt def contains} together with the fact that at most $\str$ vertices are deleted at every time-step.
\begin{obs}\label{obs: pathcol mono time}
Let $t,t',\ell\in\N$, and $A\subset \Z$, and assume that $\spr_{\tau}(S')=1$ for all $t'\le \tau\le t$ and all $S'\in \seg_{\tau}(A)$. Then $\dispath{t}{\ell-(t-t')}(A) \ge \dispath{t'}{\ell}(A)-\lfloor\str t\rfloor+\lfloor\str t'\rfloor$.
\end{obs}

The third notes that paths emanating from distant columns must be disjoint.
\begin{obs}\label{obs: far paths no intersection}
Let $P$ and $Q$ be $t$-paths of length $\ell$, emanating from columns $x$ and $y$ respectively.
If $|x - y| \ge \ell$ then $P$ and $Q$ are disjoint. 
\end{obs}

\subsection{Evolution of \texorpdfstring{$\spr(S)$}{chi(S)}}\label{subsection: chi}

In this section we describe the conditions under which a segment $S$ transitions between simulative and spreading behaviour. This description is accompanied by Figure~\ref{fig:fire general evolution}, useful throughout the paper for getting a more intuitive grasp of the strategy's mechanics.

By default segments are simulative and they become spreading under two circumstances. When container deletes a vertex within vertical distance $h_t$ from a nearby occupied vertex (which we call a \emph{disruption}), or when the segment is close to the edges of the occupied region. In both cases proximity is measured in terms of $\dispath{t}{\ell}$.
Once a segment becomes spreading, it resets a countdown, after which it returns to be simulative by retaining a single occupied vertex, an operation called \emph{consolidation}. 

The set of segments disrupted at time-step $t$ is denote by
\[\dis_t := \Big\{S\in\seg_t : (\FF_t\setminus\FF_{t-1}) \cap \left(S\times [t,t+h_{t}-1]\right) \ne \emptyset\Big\}.\]

\textbf{Alert interval.} 
A disruption at time-step $t$ at $S$ causes all segments contained in an interval $I_t(S)\subset \Z$ to become spreading for a prescribed period of time.
We call $I_t(S)$ the \emph{alert interval} of $S$. Formally,  given $s\le t$ and $S\in\seg_t$ we define
\begin{equation}\label{eq: alert-interval-def}
    I_s(S) := \left[\max\{ x\in h_s\Z : \dispath{s}{\H_t}\big([x,\min S)\big) \ge H_t\},\ \min\{x\in h_s\Z : \dispath{s}{\cint\H_s}\big((\max S, x)\big) \ge 5 \H_s\}\right].
\end{equation}

\textbf{Consolidation countdown.}
We inductively define an additional auxiliary function, the \emph{consolidation-countdown} of a segment $S$, which we denote by $\ct_t(S)$. 
This function measures the number of time-steps remaining until a spreading segment returns to simulative status, unless further disruptions ensue, in which case it resets to $H_t$, or the segment becomes close to the edge of the occupied set, in which case it resets to $h_t$. The evolution of $\spr_t(S)$ for $S\in\seg_{t}$, is thus defined via
\begin{equation}\label{eq: nsim-def}
  \spr_t(S) := \ind\{\ct_t(S) > 0\}, 
\end{equation}
where
\begin{equation}\label{eq: countdown-def}
\ct_t(S) := \begin{cases}
    \H_{t} & \text{if exists $S'\in\dis_{t}$ s.t. $S \subset I_{t}(S')$},\\
    \h_{t} & \text{otherwise, if $\left|I_{t}(S)\right| = \infty$},\\
    (\ct_{t-1}(S)-1)_+ & \text{otherwise.}
    \end{cases}
\end{equation}
In order to handle doubling times we extend the definition of $\ct_{t-1}(S)$ for $t\in\tms_2$, with $\anc_{t-1}(S) = \{A,B\}$ for $A,B\in\seg_{t-1}$ via:
\[
\ct_{t-1}(S) := \begin{cases}
        0 & \text{if $\ct_{t-1}(A) = \ct_{t-1}(B) = 0$}, \\
        \max\left(h_{t-1}+1, \ct_{t-1}(A),\ct_{t-1}(B)\right) & \text{if $\ct_{t-1}(A) > 0$ or $\ct_{t-1}(B) > 0$}.
    \end{cases}
\]
Accordingly, $\spr(S)$ for $S\in \seg_t$ is extended to times $s < t$ via $\spr_s(S):=\max\{\spr_s(S')\ :\ S'\in \anc_s(S)\}$, and to times $s > t$ via $\spr_s(S):=\spr_s(S')$ where $S'\in \seg_s$ is unique segment such that $S'\supset S$.
For $S\in \seg_0$ we set $\tau_0(S)=1$.

This extension of $\ct_{t-1}(S)$ is intended to guarantee the following observation, which will prove useful in future technical calculations.
\begin{obs}\label{obs: countdown large after doubling}
Let $t\in\tms_2$ and $S\in\seg_t$. Then either $\ct_{t-1}(S) = 0$ or $\ct_{t-1}(S) > h_{t-1}$.
\end{obs}

For future use, we establish the following proposition, stating that having relatively few occupied vertices on one side of a segment $S$ implies that $S$ is simple.

\begin{prop}\label{prop: simulated path}
Let $s,w\in\N$, 
such that $0\le w-s<H_w/8$ and $S\in\seg_w$ such that
\[\max(b_{w}((-\infty,S]),b_{w}([S,\infty)) < \H_{w}/8.\]
Then $\spr_{s} = 1$.
\end{prop}

\begin{proof}

To prove the proposition, we first establish the following claim.

\begin{clm}\label{claim: simulated path}
Let $t\in\N$ and $S\in\seg_t$ such that $\spr_{t}(S) = 0$. 
For all $t \le  t' < t+\H_{t}$ we have
\[b_{t'}((-\infty,S]) \ge \H_{t} - \floor{\str t'}+\floor{\str t}\quad \text{ and } \quad b_{t'}([S,\infty)) \ge \cint\H_{t} - \floor{\str t'}+\floor{\str t}.\]
\end{clm}

We prove only that $b_{t'}((-\infty,S]) \ge \H_{t} - \floor{\str t'}+\floor{\str t}$ for all $t' \in [t,t+\H_t)$, as the bound for $b_{t'}([S,\infty))$ follows similarly.
For $t' \in [t,t+\H_t)$ denote $\kappa_{t'}:=\H_{t} - \floor{\str t'}+\floor{\str t}$ and $x_{t'}:=\inf\{x\in \Z\ :\ \dispath{t'}{\H{t}-t'+t}((-\infty, x))\ge \kappa_{t'}\}$,
so that
$S_t(x_{t'})\le S$ is equivalent to $\dispath{t'}{\H{t}-t'+t}((-\infty, S])\ge \kappa_{t'}$. Hence, By Observation~\ref{obs: pathcol mono}, to establish the proposition, it suffices to to show  $S_t(x_{t'})\le S$. 

To prove that $S_t(x_{t'})\le S$, 
we show that $S_t(x_{t})\le S$ and that $S_t(x_{t'})$ is weakly monotone decreasing in $t'\in [t,t+\H_t)$. For the former observe that
$\dispath{t}{\H_{t}}((-\infty, S)) \ge \H_{t}$ by our assumption that $\spr_{t}(S)=0$, together with \eqref{eq: alert-interval-def} and the second case of~\eqref{eq: countdown-def}. 
For the latter, 
observe that by the definition of $\spr$ we have $\spr_{t'}(S_{t'}(x_{t'}))=1$. Assume towards obtaining a contradiction,  $S_t(x_{t'+1})>S_t(x_{t'})$ so that, by the second case of~\eqref{eq: countdown-def}, 
$\spr_{t'+1}(S_t(x_{t'}))=1$. 
This would imply that
$\dispath{t'+1}{\H{t}-(t'+1-t)}(-\infty,S_t(x_{t'})] < \kappa_{t'+1},$
in contradiction with 
\[\dispath{t'+1}{\H{t}-(t'+1-t)}(-\infty,S_t(x_{t'})] \ge\dispath{t'}{\H{t}-(t'-t)}(-\infty,S_t(x_{t'})] - \floor{\str (t'+1)}+\floor{\str t'} \ge \kappa_{t'} - \floor{\str (t'+1)}+\floor{\str t'} = \kappa_{t'+1},\]
where the first equality uses Observation~\ref{obs: pathcol mono time}, applied with $t\leftarrow t'+1$, $\ell\leftarrow\H_{t}-(t'-t)$ and $I\leftarrow(-\infty,S_t(x_{t'})]$).
Hence $S_t(X_t')$ is monotone decreasing. The claim follows.

To deduce the proposition, apply the claim contra-positively with $t'\leftarrow w$,$t\leftarrow s$, observing that for each $S'\in\anc_s(S)$ we have $H_w\le 4H_s$ and $\floor{qs}-\floor{qw}< 3H_w/8$, so that
\begin{align*}b_{s}((-\infty,S'])&<H_w/8\le H_s-\floor{qw}+\floor{qs},\\
b_{s}([S',\infty))&<H_w/8\le 2H_s-\floor{qw}+\floor{qs}. \qedhere
\end{align*}
\end{proof}

\subsection{Sparsity of the strategy}\label{subsection: sparsity}

In this section we give a bound on the number of occupied vertices used by Spreader's strategy at every time-step, enabling a sub-linear $g(t)$ as per Theorem~\ref{thm: sub-linear win}. To do this, we make the following final adjustment to the strategy.

\textbf{Avoiding dead-ends.} Denote 
$B_t^\ell=\{v\in B_t\ :\ \exists (v,v_1,\dots, v_\ell)\in P_t^\ell(\Z)\}$.
Observe that for any $t$ and $\ell$, each vertex outside of $B_t^\ell$ will eventually have no descendents, regardless of how Spreader or Container play henceforth. Avoiding these ``dead-ends'', namely restricting Spreader's spread at time-step $t$ to the vertices of $B_t^\ell$ cannot alter the result of the game. Namely,
\begin{obs}\label{obs: dead ends Zei}
    For every $\ell, t\in\N$, $B_t^\ell \ne \emptyset \iff B_t \ne \emptyset$.
\end{obs}

Our winning strategy for Theorem~\ref{thm: sub-linear win} is $|B_t^{3H_t}|$ and, indeed, in this section we bound
$|B_t^{3H_t}|$ by $g(t)$. However, in Section~\ref{subsecion: strategy analysis} we validate
Spreader's win under~\eqref{eq: Bt def eighth plane} (which does not obey the constraint of Spreader occupying at most $g(t)$ vertices at every time-step), as winning under both strategies is equivalent. 

Our purpose here is to establish the following. 

\begin{prop}\label{prop: B(t) is virtually-viable}
$\left|B_t^{3H_t}\right| \le O(h_{t}^6) + \frac{t}{h_{t}}$.
\end{prop}

To this end, we require the following counting claim.
\begin{clm}\label{clm: cover}
Given $\ell>0$, $A\subset \Z$, and a collection $\{A_i\}_{i=1,\dots,k}$ of disjoint subsets of $A$, such that 
$B_t^\ell(A_i)>0$ for all $1 \le i \le k$, and $(A_i-A_j)\cap \{0,\dots, \ell\}=\emptyset$ for all $i\neq j$. Then, $\dispath{t}{\ell}(A)\ge k$.
\end{clm}
\begin{proof}
This is a nearly straightforward consequence of Observation~\ref{obs: far paths no intersection}, recalling that $\dispath{t}{\ell}(A)$ counts disjoint $t$-paths, while $B_t^\ell(A_i)$ counts $t$-paths in general.
\end{proof}

We wish to establish Proposition~\ref{prop: B(t) is virtually-viable} both for $\Zei$ and $\Zhp$. Unfortunately, this cannot be done relying only upon the common strategy \eqref{eq: Bt def contains}, and we must prove it under the explicit definitions of \eqref{eq: Bt def eighth plane} and \eqref{eq: Bt def directed half plane}. 

\begin{proof}[Proof of Proposition~\ref{prop: B(t) is virtually-viable}]
Let $t\in\N$ and partition the vertices of $B_t^{3H_t}$ according to $\spr_t$.
\[\left|B_t^{3H_t}\right| = \left|\left\{y\in B_t^{3H_t} : \spr_t(S_t(y))=0 \right\}\right| + \left|\left\{y\in B_t^{3H_t} : \spr_t(S_t(y))=1 \right\}\right|.\]
Since a simulative segment has at most one occupied vertex, and $B_t \subset [0,t]$, the first term is bounded by $\frac{t+1}{h_{t}}$.

We are thus left with showing
\begin{equation}\label{eq: sparsity}
\Sigma := \left|\left\{y\in B_t^{3H_t} : \spr_t(S_t(y))=1 \right\}\right| = O(h_{t}^6)\end{equation}

First, observe that by~\eqref{eq: nsim-def} for every $y\in \Sigma$ we have $\ct_t(S_t(y)) > 0$. 
In view of \eqref{eq: countdown-def}, let us partition $\Sigma$ into vertices which were recently affected by a disruption and those which were recently at the edge of the occupied set, denoting
\[\Sigma_1 := \bigcup_{s = t-H_t+1}^{t} \bigcup_{S\in\dis_s}
    \{y\in B_t^{3H_t} : S_{s}(y) \subset I_{s}(S)\} \quad\text{and}\quad  \Sigma_2 := \bigcup_{s = t-h_t+1}^{t} \{y\in B_t^{3H_t} : \left|I_s(S_s(y))\right| = \infty\},\]
so that $\Sigma\subset \Sigma_1\cup \Sigma_2$.

To analyse $|\Sigma_1|$ and $|\Sigma_2|$, we introduce a definition pertaining to the evolution of the process which lead to the occupation of a given vertex.
Given $y\in B_{t}$,
denote $\origin_s(y)\in B_s$ for the leftmost starting vertex of a $s$-path starting in $B_s$ and ending in $y$. Observe that the graph distance $d_G$ satisfies
\begin{equation}\label{eq:ancest}
d_G(\origin_s(y),y)=t-s.
\end{equation}

Note that by~\eqref{eq: Bt def eighth plane}, the evolution equation of the $\Zei$ strategy, that $\origin_s(y)$ indeed always exists for every $y\in B_t$ and $s\in \{t-H_t+1,\dots,t\}$.
In the case of the $\Zhp$ strategy defined via~\eqref{eq: Bt def directed half plane}, however, this may be violated, but only if there exists an $s$-path from $\min B_s + (-1,1)$ to $y$ for some $s\in \{t-H_t+1,\dots,t\}$.
The number of endpoints of $\min B_s + (-1,1)$ at time-step $t$ is bounded by $2(t-s)+1 \le 2H_t$. Thus there can be a total of at most $O(H_t^2) = O(h_t^4)$ such violating vertices in $B_t$, which does not change the asymptotics of the upper bound of $O(h_t^6) + \frac{t}{h_t}$.

We begin by obtaining an upper bound for $\left|\Sigma_1\right|$.
Fix $s\in \{t-H_t+1,\dots,t\}$ and $S\in\dis_s$, and let $y\in B_t^{3H_t}$ such that $S_s(y) \subset I_s(S)$. Consider $\origin_s(y)$ and observe that, by~\eqref{eq:ancest}, $y\in \beta_s(y)+(\{0,\dots,H_t\}\times \{t-s\})$.  
Moreover, because $y\in B_t^{3H_t}$ and there is a $s$-path between $\origin_s(y)$ and $y$, we have $\origin_s(y) \in B_s^{3H_t}$, so that $y\in B_s^{3H_t}+(\{0,\dots,H_t\}\times \{t-s\})$. 
All in all, writing $B^+=\left(B_{s}^{3H_t} + \{0,\dots,H_t\}\times\{t-s\}\right)$
and $I^+=(I_s(S) +\{-H_t,\dots,0\}) \times \{t\}$, we have
\[\left| \{y\in B_t^{3H_t} : S_{s}(y) \subset I_{s}(S)\}\right| \le \left|B^+ \cap I^+ \right|.\]

To bound the right hand side, we cover $B^+ \cap I^+$ by disjoint contiguous intervals of size $4\H_t$.
Denote \[\mathcal{R}=\{m\in4 H_t \Z\ :\ [m,m+4H_t-1]\times \{t\}\cap (B^+ \cap I^+)\neq \emptyset\},\] 
so that $|B^+ \cap I^+| \le 4H_t|\RR|$.
Observe that for all $m\in \RR$ 
the interval
$I_m:=\{m-H_t\dots,m+3H_t\}$ satisfies $I_m\times \{s\}\cap B_s^{H_t}\neq \emptyset$. Moreover, for $m\in \RR \cap 4H_t\cdot 2\Z$ (similarly  $4H_t\cdot (2\Z+1)$),  these intervals are at least $3H_t$ apart. Hence, applying Claim~\ref{clm: cover} with $\ell=3H_t$ to these intervals, we obtain
\[   \frac{1}{2}|\mathcal{R}|\le \max(|\mathcal{R}\cap 4H_t\cdot 2\Z|,|\mathcal{R}\cap 4H_t(2\Z+1)|)\le \dispath{s}{3H_t}\big(I_t(S)+\{-H_t,\dots,0\}\big),\]
so that $|\mathcal{R}|$ must contain at most $8\H_t$ intervals.


Recall that, by the definition of $I_t(S)$ (see~\eqref{eq: alert-interval-def}), we have $\dispath{s}{3H_t}(I_t(S)+\{-H_t,\dots,0\}) \le 5\H_s +H_t \le 6\H_t$.  Therefore, 
\[|B^+ \cap I^+| \le 4H_t \cdot |\RR|  \le 8H_t\cdot\dispath{s}{3H_t}(I_t(S)+\{-H_t,\dots,0\}) \le 8H_t\cdot 6H_t = O(H_t^2).\]
Observing that at most $\floor{\str}+1$ disruptions occur at every time-step, we conclude that
\[|\Sigma_1| \le H_t \cdot (\floor{\str}+1) \cdot O(H_t^2) = O(H_t^3) = O(h_t^6).\]

Next, we find an upper bound for $|\Sigma_2|$. Fix $s\in [t-h_t+1,t]$ and write, in light of~\eqref{eq: alert-interval-def},
\begin{align}
B^+_2 &:= B_s^{3H_t}+\{0,\dots,h_t\}\times\{t-s\} \notag\\
\quad I^+_2 &:= \left\{(x,t) :\dispath{s}{\H_s}((-\infty,x]) \le \H_t \text{ or } \dispath{s}{2\H_s}([x,\infty)) \le 2\H_t+h_t\right\}=(-\infty, x_0]\times \{t\}\cup [x_1, \infty)\times \{t\},\label{eq: useful I2}
\end{align}
where $x_0=\max\{x\ :\dispath{s}{\H_s}((-\infty,x]) \le \H_t\}$, $x_1=\min\{x\ :\dispath{s}{\H_s}([x,\infty)) \le 2\H_t\}$.
As in the case of $\Sigma_1$, consider $\origin_s(y)$ and
follow similar reasoning to obtain, 
\[\left|\{y\in B_t^{3H_t} : I_s(S_s(y)) = \infty\}\right| \le \left|B^+_2 \cap I^+_2\right|.\]

By the definition of $B_2^+$, we have $|B^+_2 \cap I^+_2| \le h_t \cdot |B_s^{3H_t} \times \{t\} \cap I^+_2|$. Noting that each $s$-path of length $3H_t$ emanating from a column of a vertex in $I^+_2$ intersects at most $3H_t-1$ other such paths, we apply Observation~\ref{obs: far paths no intersection} to obtain
\[|B_s^{3H_t} \times \{t\} \cap I^+_2| \le 3H_t \cdot \dispath{s}{3H_t}(I^+_2) = O(H_t^2) = O(h_t^4),\]
where the first equality uses \eqref{eq: useful I2}.
Therefore, $|B^+_2 \cap I^+_2| \le h_t\cdot |B_s^{3H_t} \times \{t\} \cap I^+_2| = O(h_t^5)$. Applying a union bound on $s\in \{t-h_t+1,\dots,t\}$, we conclude that
\[|\Sigma_2| \le h_t \cdot O(h_t^5) = O(h_t^6).\]

This concludes the proof for the $\Zei$ strategy defined in~\eqref{eq: Bt def eighth plane}. For the $\Zhp$ strategy defined in~\eqref{eq: Bt def directed half plane}, note that a vertex $y\in B_t^{3H_t}$ does not necessarily possess $\origin_s(y)$ for any $s\in \{t-H_t+1,\dots,t\}$. However, this can only happen if there exists some such $s$ such that $\min B_s$ is an ancestor of $y$. The number of offspring of $\min B_s$ at time-step $t$ is bounded by $2(t-s)+1 \le 2H_t$. Thus there can be a total of at most $O(H_t^2) = O(h_t^4)$ such offspring, which does not change the upper bound $O(h_t^6) + \frac{t}{h_t}$. 
\end{proof}




\subsection{Strategy analysis and proof of Theorem~\ref{thm: sub-linear win}}\label{subsecion: strategy analysis}

Following the description of Spreader's strategy in $\Zei$, given in Sections~\ref{subsection: Zei strategy}-\ref{subsection: chi}, and the proof of its sparsity, given in Section~\ref{subsection: sparsity}, we dedicate the remainder of Section~\ref{section: eighth plane} to show that this strategy is indeed winning.

To this end we now introduce several new notions, characterising them by their key properties and postponing their formal definition to later sections. 

Given a segment $S\in\seg_t$, 
we define an increasing  sequence of sets $F_t(S)\subset \FF_t$, which describe the deleted vertices associated with $S$ up to time-step $t$. For a formal definition see~\eqref{eq: firefighters def} below. For handling the half-plane and the entire plane graphs, it will be useful to define the game with an arbitrary initial set of deleted vertices, denoted by $\FF_0$.

The fact that we ever count deleted vertices towards more than one segment $S\in \seg_t$ is summarised in the following claim (established in Section~\ref{subsection: analysis} after \eqref{eq: firefighters def}).
\begin{clm}\label{clm: f < rho t}
$|F_t(\seg_t)|-|\FF_0| \le \str t$ for every $t\in\N$.
\end{clm}

We also introduce the function 
$\phi_t(S)$, initialised as $\phi_0(S) = |B_0(S)|$, 
which serves to represent the size of the simulated fire in $S$ at time-step $t$, together with $d_t(S)$, a non-negative quantity initialised as $d_0(S) = 0$
which tells us by how much $\phi$ deviated from its desired evolution. 

The following proposition, established in Section~\ref{subsection: debt bdd}, guarantees that $\phi_t(\seg_t) + d_t(\seg_t)=0$ is indeed equivalent to the fact that $S$ is unoccupied. 
\begin{prop}\label{prop: no fire -> no ptl no dbt}
$B_t \neq \emptyset \iff \phi_t(\seg_t) + d_t(\seg_t) > 0$, for all $t\in\N$.
\end{prop}

In addition, $\phi$ and $d$ are designed to guarantees that the total sum $\phi_t+d_t$ satisfies the following evolution inequality, established in Section~\ref{subsection: growth, eighth plane}

\begin{prop}\label{prop: Phi grows 1}
Let $s \le t$. If $B_{s-1} \ne \emptyset$ then $\Delta\phi_s(\seg_t) + |\Delta F_s(\seg_t)| + \Delta d_s(\seg_t) \ge 1$.
\end{prop}

\begin{proof}[Proof of Theorem~\ref{thm: sub-linear win} for $\Zei$]
Fogarty~\cite{fogarty2003catching} established the fact that $\str_{\Zei}=(1,\infty)$. Thus, to prove the theorem, it remains to show that $1\notin \str(\Zei, Ct^{6/7})$ for some sufficiently large $C$, and some initial fire $B_0$. Following our discussion in Section~\ref{subsection: sparsity}, this reduces to the existence of a winning Spreader strategy against $\str =1$ with $|B_t^{3H_t}| \le Ct^{6/7}$. The fact that our strategy satisfies the latter follows directly from Proposition~\ref{prop: B(t) is virtually-viable}, applied with $h_t := 2^{\lfloor \log_2((t+2^{14})^{1/7})\rfloor}$. To see that \eqref{eq: doubling condition} is satisfied, denote $t_k$ for the $k$-th doubling time-step and observe that for all $k\ge 1$  we have\[t_{k}-t_{k-1}=2^{7(k+2)}-2^{7(k+1)}\ge 100\cdot 2^{2k}\ge 10q^2 2^{2k}=2H_{t_k}.\]

Next, let us show that the strategy $B = (B_t)_{t\in\N}$ is winning against $\str=1$ for $B_0=\{(0,0)\}$. To so do we prove by induction that $B_t\ne\emptyset$ for every $t\in\N$ against any Container strategy $\{\FF_t\}_{t\in\N}$.
For $t = 0$ this is clear. Next, assume that this holds up to time-step $t-1$.
By Proposition~\ref{prop: Phi grows 1} we have $\Delta\phi_s(\seg_t) + |\Delta F_s(\seg_t)| + \Delta d_s(\seg_t) \ge 1$ for every $1\le s\le t$. Summing this up over $s\le t$ we obtain 
\[\phi_t(\seg_t) + |F_t(\seg_t)| +  d_t(\seg_t) \ge t+\phi_0(\seg_t) + |F_0(\seg_t)| + d_0(\seg_t)  = t+1.\]
Since $\str = 1$ we have $|F_t(\seg_t)| \le t$, by Claim~\ref{clm: f < rho t} and the fact that $\FF_0 = \emptyset$. Hence $\phi_t(\seg_t) + d_t(\seg_t) \ge 1$. By Proposition~\ref{prop: no fire -> no ptl no dbt} this implies that $B_t\ne\emptyset$, as required.
\end{proof}

\subsection{Counting blocked vertices and simulated fire: evolution of \texorpdfstring{$F(S)$}{F(S)} and \texorpdfstring{$\phi(S)$}{phi(S)}}\label{subsection: analysis}

In this section we formally define the functions $F_t(S)$ and $\phi_t(S)$, mentioned in the previous section. %

\textbf{Look-ahead.} We begin by defining the \emph{look-ahead} of a segment, denoted by $\ell_t(S)$. This represents the vertical distance between the front and the simulated fire in the segment $S$ at time-step $t$. %
The box $S\times [t,t+\ell_t(S)]$ serves as the region where deleted vertices are added to $F_t(S)$.

 The simulated fire row of a segment $S$ never moves backwards. The look-ahead $\ell_t(S)$ of a segment $S\in \seg_t$ is defined to be $h_{t}$ for simulative segments, and once a segment transitions from simulative to spreading states, the simulated fire ceases to advance until the front  catches up with it. As a result, the look-ahead distance is reduced gradually, one unit at a time. Formally, for $t\in\N$ and $S\in\seg_{t}$ we define,
\begin{align}
    \ell_{t-1}(S) &:= \max\{\ell_{t-1}(S')\ :\ S'\in \anc_{t-1}(S)\},\notag\\
    \ell_t(S) &:= \begin{cases}
    h_{t} & \text{if $\spr_t(S) = 0$,}\\
    (\ell_{t-1}(S)-1)_+ & \text{if $\spr_t(S) = 1$},
\end{cases}\label{eq: look-ahead-def}\\
    L_t(S) &:= S\times [t,t+\ell_t(S)].\notag
\end{align}
The look-ahead is initialised as $\ell_0 \equiv 0$ for all segments. We refer to $L_t(S)$ as the \emph{look-ahead region} of $S$, saying that $S$ is  \emph{simple} whenever $\ell_t(S) = 0$.
We generalise the definitions to $S\in \seg_t$ and $s<t$ by setting
\begin{equation}\label{eq: look-ahead past generalisation}
  \ell_{s}(S) := \max_{S'\in \anc_{s}(S')}\ell_{s}(S).  
\end{equation}

Using \eqref{eq: doubling condition} this guarantees the following.
\begin{obs}\label{obs: lookahead dies out}
    Let $s,t\in\N$ and $S\in\seg_t$ such $s\le t-2h_s$ and $\spr_{t'}(S)=1$ for all $t'\in [s,t]$, then $\ell_{t}(S)=0$.
\end{obs}

 On such segments the simulated fire consists simply of the occupied vertices. See Figure~\ref{fig:segment step} for a depiction of one Spreader step in a single segment under three different conditions -- $\ell_t(S) = 0$, $0<\ell_t(S)<h_t$, and $\ell_t(S) = h_t$.

\begin{figure}[ht]
    \captionsetup{margin=1cm}
    \subfloat[\label{fig:segment step, simulated}]{
    \begin{tikzpicture}[scale=\sca]
        \draw[step=1cm,darkgray,thin] (0,0) grid (7,8);
        \draw [black, very thick, dashed,pattern=my north east lines,pattern color=lightgrayX] (0,0) rectangle (7,7);
        \draw[black, very thick,fill=none] (0,0) rectangle (7,1); 
        
        \node at (2.5,0.5) {\Fire}; 
        \node[fill opacity=0.5] at (3.5,1.5) {\Fire};
        
        \node at (0.5,5.5) {\faShield};
        \node at (2.5,2.5) {\faShield};
        \node at (3.5,2.5) {\faShield};
        \node at (5.5,7.5) {\faShield};
        \node at (0.5,7.5) {\faShield};
 
        \draw [|-|, line width=0.3mm] (-0.5,0.5) -- (-0.5,6.5);
        \draw [|-|, line width=0.3mm] (-1,1.5) -- (-1,7.5);

    \end{tikzpicture}}
    \hfill
    \subfloat[\label{fig:segment step, transitional}]{
    \begin{tikzpicture}[scale=\sca]
        \draw[step=1cm,darkgray,thin] (0,0) grid (7,4);
        \draw [black, very thick, dashed,pattern=my north east lines,pattern color=lightgrayX] (0,0) rectangle (7,4);
        \draw[black, very thick,fill=none] (0,0) rectangle (7,1); 
        
        \node at (0.5,0.5) {\Fire};
        \node at (2.5,0.5) {\Fire};
        \node at (3.5,0.5) {\Fire};
        \node at (4.5,0.5) {\Fire};
        \node[fill opacity=0.5] at (1.5,1.5) {\Fire};
        \node[fill opacity=0.5] at (2.5,1.5) {\Fire};
        \node[fill opacity=0.5] at (5.5,1.5) {\Fire};
        
        \node at (0.5,1.5) {\faShield};
        \node at (3.5,1.5) {\faShield};
        \node at (4.5,1.5) {\faShield};
        \node at (2.5,2.5) {\faShield};
        \node at (3.5,2.5) {\faShield};
        \node at (2.5,3.5) {\faShield};
 
        \draw [|-|, line width=0.3mm] (-0.5,0.5) -- (-0.5,3.5);
        \draw [|-|, line width=0.3mm] (-1,1.5) -- (-1,3.5);
    \end{tikzpicture}}
    \hfill
    \subfloat[\label{fig:segment step, simple}]{
    \begin{tikzpicture}[scale=\sca]
        \draw[step=1cm,darkgray,thin] (0,0) grid (7,2);
        \draw [black, very thick, dashed,pattern=my north east lines,pattern color=lightgrayX] (0,0) rectangle (7,1);
        \draw[black, very thick,fill=none] (0,0) rectangle (7,1); 
        
        \node at (0.5,0.5) {\Fire};
        \node at (1.5,0.5) {\Fire};
        \node at (5.5,0.5) {\Fire};
        \node[fill opacity=0.5] at (1.5,1.5) {\Fire};
        \node[fill opacity=0.5] at (2.5,1.5) {\Fire};
        \node[fill opacity=0.5] at (5.5,1.5) {\Fire};
        \node[fill opacity=0.5] at (6.5,1.5) {\Fire};
        
        \node at (0.5,1.5) {\faShield};
 
        \draw [|-|, line width=0.3mm] (-0.5,0.3) -- (-0.5,0.7);
        \draw [|-|, line width=0.3mm] (-1,1.3) -- (-1,1.7);
    \end{tikzpicture}}
\caption{Depicted in each sub-figure, are a segment and the union of its look-ahead regions in in two consecutive time-steps $t-1$ and $t$. Look-ahead distances are marked by vertical lines and $S\times L_{t-1}(S)$ by a filling pattern.
In~\eqref{fig:segment step, simulated} the segment is simulative, in~\eqref{fig:segment step, transitional} it is spreading with positive look-ahead, while in~\eqref{fig:segment step, simple} it is simple. }
    \label{fig:segment step}
\end{figure} 

\textbf{Deleted vertices.}
For each segment $S\in\seg_{0}$ set $F_{0}(S)=\emptyset$. We inductively define $F_t(S)$ for times $t\in\N$ and segments $S\in\seg_{t-1} \cup \seg_{t}$ as follows.
In a non-simple segment we add to $F(S)$ all deleted vertices inside its look-ahead region, and in a simple segment we add only the deleted vertices of the front at time-step $t$ neighbouring occupied vertices (this will be important only for the analysis of the strategy in $\ZZ$). Namely,
\begin{align}
F_{t-1}(S)&:= \bigcup_{S'\in \anc_{t-1}(S)}F_{t-1}(S'),\notag\\
\Delta F_t(S) &:= (\FF_{t}\setminus F_{t-1}) \cap \begin{cases}
    L_{t}(S) & \text{if $\ell_t(S) > 0$}\\
    L_{t}(S) \cap \big(B_{t-1}+\{(-1,1),(0,1), (1,1)\}\big)  & \text{if $\ell_t(S) = 0$}.
\end{cases}\label{eq: firefighters def}
\end{align}
Observe that Claim~\ref{clm: f < rho t} follows directly from the facts that $F_t(S) \subset \FF_t$ for any $S\in\seg_t$, and that $\{F_t(S)\}_{S\in\seg_t}$ are pair-wise disjoint, since $F_t(S) \subset S\times\Z$.

Non-disruptive segments have the following property.
\begin{obs}\label{obs: exhaustive count in non-disruptive segs}
If $S\in \seg_t\setminus\dis_{t}$ then 
$L_{t-1}(S)\cap \Delta F_t(S)=\emptyset$.
\end{obs}
\begin{proof}
As $S\in \seg_t\setminus\dis_{t}$, and $\ell_{t-1}(S)\le h_t$, we have $L_{t-1}(S) \cap \Delta \FF_t = \emptyset$. In addition, by \eqref{eq: firefighters def}, at every time-step $t'<t-1$, we have
$\FF_{t'}(L_{t'}(S))\setminus {F_{t'}} \subset S\times \{t'\}$, which is disjoint from $L_{t-1}(S)$. Hence, 
if $\ell_{t-1}(S)\neq 0$, we have
$F_t \supset \FF_{t-1}(L_{t-1}(S)),$
from which the observation easily follows. On the other hand, $\ell_{t-1}(S)= 0$, then $L_t(S)\cap L_{t-1}(S)=\emptyset$ and, as $\Delta F_t(S)\subset L_t(S)$ the observation follows.
\end{proof}

From this we deduce the following.
\begin{obs}\label{obs: delta f in slab}
    Let $S\in\seg_t$ such that $\spr_t(S)=0$. Then $\Delta F_t(S) = \FF_t(S\times[t+\ell_{t-1}(S),t+\ell_t(S)]$.
\end{obs}
\begin{proof}
    $\spr_t(S)=0$ implies $S\notin\dis_t$. The observation follows by~\eqref{eq: firefighters def} together with Observation~\ref{obs: exhaustive count in non-disruptive segs}.
\end{proof}

We also me the following observation concerning simple segments.
\begin{obs}\label{obs: new ff after simple 2}
Let $t\in\N,S\in\seg_{t}$. If $\ell_{t-1}(S)=0$ then $L_t(S)\cap F_{t-1}(S) = \emptyset$.
\end{obs}
\begin{proof}
Note that $F_{t-1}(S) \subset S\times (-\infty, t-1+\ell_{t-1}(S)) = S\times (-\infty, t-1)$.
As $L_t(S) \subset S\times [t,\infty)$, the observation follows.
\end{proof}

\textbf{Simulated fire.} 
Given a segment $S\in \seg_{t}$, we define 
\begin{equation}\label{eq: r def}\rng_t(S) := \{y\in G_t : \exists (x_0,\dots,y)\in P_t^{\ell_t(S)}(S), x_0\in B_t(S)\},\end{equation}
the number of elements in $S\times\{t+\ell_t(S)\}$ which are the endpoints of a $t$-path emanating from $B_t(S)$. These represent the set of simulated burning vertices in $S$.
In order to handle doubling times, we generalise the definition for time-step $t-1$ by setting $\rng_{t-1}(S) := \sum_{S'\in \anc_{t-1}(S)}\rng_{t-1}(S')$.
For every $t\in\N$ and $S\in\seg_{t}$ we define $\phi_{t}(S)$, which will serve as a useful lower bound on the size of $\rng_t(S)$.
In a simulative segment $S$, it is reduced by its value by $1$ for every newly counted deleted vertex and is otherwise fixed. In spreading segments it measures the size of $\rng_t(S)$. 
Formally, we define $\phi_t(S)$ inductively by
\begin{align}
    \phi_{t-1}(S) &:=\sum_{S'\in \anc_{t-1}(S)}\phi_{t-1}(S'),\notag \\
    \phi_{t}(S) &:=
    \begin{cases}
        (\phi_{t-1}(S) - |\Delta F_{t}(S)|)_+ & \text{if $\spr_t(S) = 0$},\\
        \rng_t(S) & \text{if $\spr_t(S) = 1$}.
    \end{cases}\label{eq: phi def}
\end{align}
Observe that this fully defines $\phi$, as $\spr_t(S)=1$ for all $S\in \seg_0$. 

Note that $\phi_t(S) \le h_t$ for all $t$ and that in case that $\ell_t(S) = 0$ we have
\begin{equation}\label{eq:simple phi}
\phi_{t}(S)=|B_t(S)|.
\end{equation}

The fact that $\phi_t(S)$ bounds $\rng_t(S)$ from below, is expressed in the following lemma, whose somewhat technical proof we postpone to Section~\ref{subsection: phi < r}.

\begin{lem}\label{lem: phi < r}
For all $S\in\seg_{t}$ we have $\phi_t(S) \le \rng_t(S)$.
\end{lem}

\subsection{Potential}\label{subsection: mu}

Following \cite{feldheim20133}, for any segment $S\in\seg_{t}$ we define its \emph{potential} by $\ppo_t(S) := \phi_t(S) + |F_t(S)|$. This function played a key role in previous works on the Firefighter Problem, and its evolution was the main instrument in establishing victory for the fire. In particular, on each finite segment it was non-decreasing.
In our setting, this is indeed the case for simple segments, as the following variation on \cite{fogarty2003catching}*{Theorem~1} indicates. Note that \ref{prop: mu not ex nihilo simple: itm fogarty2} is only meant to enhance the result so that it will be viable for handling $\Zhp$ in Section~\ref{section: directed half plane}.

\begin{prop}\label{prop: mu not ex nihilo simple}
Let $t\in\N$ and $S\in\seg_{t}$ be such that $\ell_{t-1}(S) = \ell_t(S) = 0$ then 
\begin{enumerate}[label=\textnormal{(\alph*)}]
    \item \label{prop: mu not ex nihilo simple: itm simple}
    $\Delta \ppo_t(S) \ge 0$,
    \item \label{prop: mu not ex nihilo simple: itm fogarty}
    If there exists $x\in B_{t-1}$ such that $x+(1,0) \notin B_{t-1}(S)$, then $\Delta \ppo_t(S) \ge 1$.
    \item \label{prop: mu not ex nihilo simple: itm fogarty2}
    Moreover, if there exists $x\in B_{t-1}$ such that $x+(1,0) \notin B_{t-1}(S)$, then \[\Delta \ppo_t(S) \ge 1+\ind\{\min B_{t-1} + (-1,1)\in F_t(S)\cup B_t(S)\}.\] 
\end{enumerate}
\end{prop}

\begin{proof}
Let $a$ be an indicator to the fact that the condition of Items~\ref{prop: mu not ex nihilo simple: itm fogarty} and \ref{prop: mu not ex nihilo simple: itm fogarty2} holds.
By~\eqref{eq: Bt def contains}, any neighbour in $S\times\{t\}$ of $B_{t-1}$ is either occupied or deleted at time-step $t$. There are always at least $|B_{t-1}(S)|$ such neighbours which are vertically above a vertex in $B_{t-1}$, and in case that $a=1$, at least one additional such neighbour exists in $B_{t-1} + \{(0,1),(1,1)\}$. By Observation~\ref{obs: new ff after simple 2} we know that every such deleted vertex is counted towards $|\Delta F_t(S)|$. We thus have $|B_{t-1}(S)| + a \le |B_t(S)| + |\Delta F_t(S)|$. As $\ell_{t-1}(S)=\ell_{t}(S)=0$, for $t' \in \{t-1,t\}$ we have $\phi_{t'}(S) = |B_{t'}(S)|$ by~\eqref{eq:simple phi}. We conclude that
\[\Delta \ppo_t(S) = \Delta \phi_t(S) + |\Delta F_t(S)| = \Delta |B_t(S)| + |\Delta F_t(S)| \ge a,\]
as required. Observing that $\min B_{t-1} + (-1,1) \notin B_{t-1} + \{(0,1),(1,1)\}$, item~\ref{prop: mu not ex nihilo simple: itm fogarty2} follows.
\end{proof}


The same property holds for all segments as long as no disruption occurs, as stated in the following proposition. 

\begin{prop}\label{prop: mu not ex nihilo}
If $S\in\seg_t\setminus\dis_t$ then $\Delta \ppo_t(S) \ge 0$. 
\end{prop}

\begin{proof}
The fact that $\Delta \ppo_t(S) \ge 0$ holds when $\spr_t(S)=0$ is immediate from~\eqref{eq: phi def}.
When $\spr_t(S) = 1$, we have $\phi_{t}(S) = \rng_{t}(S)$, while, by 
Lemma~\ref{lem: phi < r}, $\phi_{t-1}(S) \le \rng_{t-1}(S)$. 
Observing that 
\[\Delta\ppo_t(S)=\Delta\phi_{t}(S)+|\Delta F_t(S)|\ge \Delta\rng_t(S)+|\Delta F_t(S)|,\]
we are left with showing
\begin{equation}\label{eq: mu not ex nihilo: itm r+f}
 \Delta \rng_t(S) + |\Delta F_t(S)| \ge 0.
 \end{equation}


In order to tackle potential doubling at time-step $t$ we observe that, by our assumption $S \notin \dis_t$,
\begin{align*}
\rng_t(S) &\ge \sum_{S'\in \anc_{t-1}(S)}\rng_t^{\ell_t(S)}(S'),\\
\Delta F_t(S)&=\sum_{S'\in \anc_{t-1}(S)}\big|\FF_t(S' \times [t+\ell_{t-1}(S'),t+\ell_t(S)]) \big|.
\end{align*}
Hence, to obtain \eqref{eq: mu not ex nihilo: itm r+f}, it would suffice to show for all $S'\in \anc_{t-1}(S)$ 
\begin{equation*} \rng_t^{\ell_t(S)}(S')-\rng_{t-1}^{\ell_{t-1}(S')}(S')  + \big|\FF_t(S' \times [t+\ell_{t-1}(S'),t+\ell_t(S)]) \big| \ge 0.
 \end{equation*}
Indeed, by our assumptions $S\notin \dis_t$, $\spr_t(S)=1$ we have $\rng_t^{\ell_{t-1}(S')-1}(S')=\rng_{t-1}^{\ell_{t-1}(S')}(S')$, so that this is a direct consequence of Claim~\ref{clm: faithful upwards reduction} with $[a,b]=S'$,$\ell_1=\ell_{t-1}(S')-1$, 
$\ell_2=\ell_{t}(S)$.
\end{proof}

Disruptions, however, may cause a drastic reduction in $\phi_t(S)$, with only a small increase to $|F_t(S)|$, resulting in a reduction in $\ppo_t(S)$. Subsequent sections will introduce the debt $d_t(S)$ to handle these reductions in an amortised fashion.

\subsection{Debt}\label{subsection: debt}

Here we define $d_t(S)$, the \emph{debt}, 
used to keep track of reductions in $\ppo$ and see that these are compensated for. 
Hence we set $d_0 \equiv 0$, and let $\Delta d_t$ increase at every time-step by $\big(-\Delta \ppo_t(S)\big)_+$. To reduce the debt, we identify $\lambda_t(S)$, the nearest segment to the left of $S$ with positive $\Delta \ppo_t(S)$. The increase in $\ppo_t$ of $\lambda_t(S)$ will compensate for its decrease in $S$. For use in Section~\ref{section: directed half plane}, it will be convenient not to allow this compensation to come from the leftmost occupied segment. 
Formally, 
\begin{equation}\label{eq: lambda def} 
\lambda_t(S) := \max\{Q\in\seg_{t} : Q \le S, \Delta \ppo_t(Q) > 0, Q > S_t(\min B_t)\}.
\end{equation}
Thus $\lambda_t(S)=-\infty$ whenever the maximum is taken over an empty set.

We inductively define $d_t$  via
\begin{equation}\label{eq: debt def}
\begin{aligned}
    d_{t-1}(S) &:= \sum_{S'\in \anc_{t-1}(S)}d_{t-1}(S'),\\
  \tempd_t(S) &:= d_{t-1}(S) + \big(-\Delta \ppo_t(S)\big)_+ ,\\
  d_t(S) &:= \begin{cases}
      \left(\tilde d_t(S) - \Delta\ppo_t(\lambda_{t}(S))\right)_+ & \text{if } \lambda_t(S)> -\infty, \tilde d_t(Q) = 0 \text{ for every } \lambda_t(S) \le Q < S,\text{ and } \tilde d_t(S)>0,\\
      \tilde d_t(S) & \text{otherwise}.
  \end{cases}
\end{aligned}
\end{equation}

Observe that $d_t(S)\ge0$ for all $t$. Let us make formal the fact that the increase of the debt is governed by $(-\Delta \ppo_t(S))_+$, by making the following straightforward observation.
\begin{obs}\label{obs: d > 0 implies del d < - del mu}
For all $t\in\N$ and $S\in\seg_t$ we have $\Delta d_t(S) \le (-\Delta \ppo_t(S))_+$.
\end{obs}

Next, we provide the following criteria for identifying segments in which the debt does not increase.
\begin{cor}\label{cor: debt prop}
Let $t\in\N$ and $S\in \seg_{t}$. Any of the following are sufficient conditions for $\Delta d_t(S)\le 0$: 
\begin{enumerate}[label=\textnormal{(\alph*)}]
    \item\label{cor: debt prop, mu inc -> d dec} $\Delta \ppo_t(S) \ge 0,$
    \item\label{cor: debt prop, no new debt post-trns} $S$ is simple at times $t-1$ and $t$, i.e. $\ell_{t-1}(S) =\ell_{t}(S)= 0$.
    \item\label{cor: debt prop, no new debt no dis} $S \notin \dis_t.$
\end{enumerate}
\end{cor}

\begin{proof}
\ref{cor: debt prop, mu inc -> d dec} follows immediately from Observation~\ref{obs: d > 0 implies del d < - del mu}. \ref{cor: debt prop, no new debt post-trns} and \ref{cor: debt prop, no new debt no dis} follow by combining~\ref{cor: debt prop, mu inc -> d dec} with Propositions~\ref{prop: mu not ex nihilo simple} and~\ref{prop: mu not ex nihilo} respectively.
\end{proof}

Finally, we show that $\phi_t(S)$ together with the debt cannot exceed the width of a segment.
\begin{prop}\label{prop: d + phi <= h}
Let $t\in\N$ and $S\in\seg_t$. Then $\phi_t(S) + d_t(S) \le h_t$.
\end{prop}

\begin{proof}
We start by proving the following property for any $t\in\N$ and $S\in\seg_{t}$.
\begin{equation}\label{eq: d>0 -> del d + del phi le 0}
    d_t(S) > 0 \implies \Delta d_t(S) + \Delta \phi_t(S) \le 0.    
\end{equation}
In case that $\Delta \ppo_t(S) \le 0$, \eqref{eq: d>0 -> del d + del phi le 0} is immediate by plugging $\Delta \phi_t(S) \le \Delta \ppo_t(S)$ into  Observation~\ref{obs: d > 0 implies del d < - del mu}.
Assuming $\Delta \ppo_t(S) > 0$ we have $\lambda_t(S) = S$, by~\eqref{eq: lambda def}. 

We now show that in this case $d_t(S)=\tilde d_t(S) - \Delta\ppo_t(\lambda_{t}(S))$, namely that the conditions for the first case of \eqref{eq: debt def} hold.
As $\lambda_t(S) = S$, the fact that $\lambda_t(S)\neq -\infty$ is straightforward, and the condition that $\tilde d_t(Q) = 0$ for every $\lambda_t(S) \le Q < S$ holds vacuously. Finally, observe that
\[\tilde d_t(S) = d_{t-1}(S)+(-\Delta \ppo_t(S))_+   = d_{t-1}(S) =d_t(S) - \Delta d_t(S) > -\Delta d_t(S) \ge  0,\]
where the first inequality uses 
our assumption
$d_t(S)>0$, and the last inequality uses Corollary~\ref{cor: debt prop}\ref{cor: debt prop, mu inc -> d dec}. Therefore,
\[\Delta d_t(S) = \tilde d_t(S) - \Delta\ppo_t(\lambda_{t}(S)) - d_{t-1}(S)= -\Delta\ppo_t(S) \le -\Delta \phi_t(S),\]
from which \eqref{eq: d>0 -> del d + del phi le 0} follows.

We now prove the proposition 
 by induction on $t$. The case $d_t(S) = 0$, which serves also as the basis of the induction, as it holds at time-step $t=0$, is straightforward, as $\phi_t(S) \le h_t$ by~\eqref{eq: phi def}. Otherwise, $d_t(S) > 0$, so that \eqref{eq: d>0 -> del d + del phi le 0} together with the induction hypothesis imply
\[d_t(S) + \phi_t(S) = d_{t-1}(S) + \phi_{t-1}(S) + \Delta d_t(S) + \Delta \phi_t(S) \le\hspace{-6pt} \sum_{S'\in \anc_{t-1}(S)}\hspace{-10pt}(\phi_{t-1}(S') + d_{t-1}(S'))  \le \sum_{S'\in \anc_{t-1}(S)}\hspace{-10pt}h_{t-1}= h_{t}.\qedhere\]
\end{proof}



\subsection{Bounding the debt}\label{subsection: debt bdd}

The purpose of this section is to 
establish Proposition~\ref{prop: no fire -> no ptl no dbt}, showing that all debts are paid before the occupied set dies out, namely $B_t = \emptyset$ implies $\phi_t(\seg_t) = 0$ and $d_t(\seg_t) = 0$.

The main proposition of this section (established in Section~\ref{section: bound debt duration} below) states that the debt accumulated in a given segment is bound to diminish and eventually nullify with a given time-frame. 

\begin{prop}\label{prop: debt duration bdd}
For all $t > 0$ and $S\in \seg_t$ we have $\max\{s\le t\ :\ d_{s-1}(S) = 0\} \ge t-2\str h_t^2 - 2h_t$.
\end{prop}

Segments near the edges of the occupied set must have been simple for a while. During this time their debt couldn't have increased, and hence, by Proposition~\ref{prop: debt duration bdd} they
are free of debt. This is captured by the following corollary.

\begin{cor}\label{cor: small fire -> simple and debt free}
Let $t\in \N$ and $S\in\seg_t$ such that $\min\big(|B_t([S,\infty))|,|B_t((-\infty,S])|\big) \le 2h_t$. Then $\ell_{t'}(S) = 0$ and $d_{t'}(S) = 0$ for $t'\in\{t-1,t\}$.
\end{cor}
\begin{proof}
Let $t'\in [t - 1 -\str h_t^2 - h_t,t]$.
Under the assumptions of the proposition, we obtain, by Proposition~\ref{prop: simulated path},
$\spr_{t'}(S) = 1$. This implies that 
$\ell_{t'}(S) = 0$, by \eqref{eq: look-ahead-def}. From this we deduce, using Corollary~\ref{cor: debt prop}\ref{cor: debt prop, no new debt post-trns}, that $d_t(S)\le d_{t'}(S)$. The proposition now follows from Proposition~\ref{prop: debt duration bdd}.
\end{proof}

Using this we easily establish Proposition~\ref{prop: no fire -> no ptl no dbt}.
\begin{proof}[Proof of Proposition~\ref{prop: no fire -> no ptl no dbt}]
Assume that $B_t = \emptyset$ and let $S\in\seg_t$. By Lemma~\ref{lem: phi < r} we have $\phi_t(S) = 0$.
In addition, $b_t([S,\infty))= 0 \le 2h_t$, and hence, by Corollary~\ref{cor: small fire -> simple and debt free}, $d_t(S) = 0$. Putting these together we obtain that $\phi_t(\seg_t) + d_t(\seg_t) = 0$, as required.
\end{proof}

\subsubsection{Upper bound on the duration of the debt }\label{section: bound debt duration}

This section is dedicate to the proof of Proposition~\ref{prop: debt duration bdd}.
The main idea here is that after a debt is created, the front must contain cavities in the occupied set. Each cavity gradually closes, incurring an increased growth of the occupied set that could be used to pay the debt. Let us illustrate how this unfolds.
Each indebted segment $S$ is contained in such a cavity, whose leftmost edge is within the segment $\lambda_t(S)$, in which $\ppo_t$ increases. As this increase is not not expected in a fully occupied front, it can be used to reduce the debt of segments to its right, eventually paying also for $S$ and eliminating its debt. 
 
\begin{proof}[Proof of Proposition~\ref{prop: debt duration bdd}]


Let $t\in\N$ and $S\in\seg_t$.
Only spreading segments can increase $\ppo_t$ and pay debts. Hence it would be convenient to 
prove the following stronger statement, by induction on $t$.

\begin{enumerate}[label=\textnormal{(\alph*)}]    
    \item\label{prop: debt duration bdd strong, sim no debt} If $d_t\!\left(\seg_t\big(I_t(S) \cap (-\infty,S + h_t]\big)\right) > 0$ then
    $\spr_{t}(S) = 1$.
    \item\label{prop: debt duration bdd strong, bdd}
    $\max\{s\le t\ :\ d_{s-1}(S) = 0\} \ge t-2\str h_t^2 - 2h_t$.
\end{enumerate}

The basis of the induction is straightforward, as $d_0(S) = 0$ for all $S\in\seg_0$. Assume that the proposition holds for all $t' \le t-1$. 
Given a segment $S'\in \seg_t$, denote $t_{S'} := \max\{s\le t\ :\ d_{s-1}(S') = 0\}$ and $d_{t_{S'}}(S')>0$, so that by Corollary~\ref{cor: debt prop}\ref{cor: debt prop, no new debt no dis} we have $S' \in \dis_{t_{S'}}$.
By~\ref{prop: debt duration bdd strong, bdd} applied to time-step $t-1$, we have
\begin{equation}\label{eq: tQ > t - C}
  t_{S'} \ge \max\{s\le t-1\ :\ d_{s-1}(S') = 0\} \ge t-1 - 2\str h_{t-1}^2 - 2h_{t-1} > t-\tfrac1{5q}H_t>t-\tfrac4{5q}H_{t_{S'}},  
\end{equation}
where the second-to-last inequality uses the fact that $H_t = 11\str^2 h_t^2$, and our assumption $h_0 \ge 100$, and the last inequality uses~\eqref{eq: doubling condition}.

Next, let $S\in\seg_t$.
Towards proving~\ref{prop: debt duration bdd strong, sim no debt}, i.e. that $S$ is spreading, 
denote $I_S := I_t(S) \cap (-\infty,S + h_t]$ and assume that there exists some $Q\in \seg_t(I_S)$ such that $d_t(Q) > 0$. 
Denote $\vec J := I_{t_Q}(Q) \cap [Q,\infty)$. At time-step $t_Q$ the countdown of all segments $S'\subset\vec J$ was reset to $H_{t_Q}$ so that for all $t'\in[t_Q,t_Q+H_{t_Q}-1]$, we have $\spr_{t'}(S')=1$. In addition, applying \eqref{eq: tQ > t - C} with $S'=Q$, we obtain $t<t_Q+H_{t_Q}$. Thus, in order to show~\ref{prop: debt duration bdd strong, sim no debt}, we are left with showing $S\subset \vec J$. As $Q\in \seg_t(I_S)$, we have $\dispath{t}{\H_t}([\max Q, \min S)) \le \H_t$ by~\eqref{eq: alert-interval-def}, and hence
\[\dispath{t}{\H_{t}}(\vec J \cap (-\infty,S]) \le \dispath{t}{\H_t}([Q, S]) \le \dispath{t}{\H_t}([\max Q, \min S)) + 2h_t \le \H_{t} + 2h_t.\]
It would therefore suffice to show $b_t^{\H_t}(\vec {J})>H_t+2h_t$. 
When $I_Q$ is infinite this is straightforward; otherwise, by~\eqref{eq: alert-interval-def}, we have $\dispath{t_Q}{\H_{t_Q}}(J) \ge 2\H_{t_Q}$. Hence 
\[\dispath{t}{\H_{t}}(\vec J) \ge 
\dispath{t}{2\H_{t_Q} - (t-t_Q)}(\vec J) \ge \dispath{t_Q}{2\H_{t_Q}}(\vec J)  
- \floor{\str t}+\floor{\str t_Q}
\ge 5\H_{t_Q}
- \floor{\str t}+\floor{\str t_Q} > \H_{t} + 2h_t,\]
where the first inequality uses $t<t_Q+H_{t_Q}$ and Observation~\ref{obs: pathcol mono}\ref{obs: pathcol mono, length}, the second is by Observation~\ref{obs: pathcol mono time}, the third uses~\eqref{eq: alert-interval-def}, and the fourth is obtained by plugging in~\eqref{eq: tQ > t - C} and $h_{t} \le 2h_{t_Q}$.
Hence $S\subset \vec J$ and so \ref{prop: debt duration bdd strong, sim no debt} holds at time-step $t$.

Next, towards proving~\ref{prop: debt duration bdd strong, bdd}, denote $t_S := \max\{s\le t\ :\ d_{s-1}(S) = 0\}$ and $t_0 := \max\{s \le t_S : \spr_{s}(S) = 0\}$. If $t_S=t$, then~\ref{prop: debt duration bdd strong, bdd} is straightforward. Otherwise, observe that, as before, $S \in \dis_{t_S}$, so that, by \eqref{eq: nsim-def} and \eqref{eq: countdown-def}, $\spr_{s}(S) = 1$ for all $s\in [t_S,t_S+H_{t_S}-1]$ . Hence, by Observation~\ref{obs: lookahead dies out}, $\ell_s(S)=0$ for all $s\in [t_0+2h_{t_0},\dots,t_S + \H_{t_S}-1]$, so that $\Delta d_s(S)\le 0$, by Corollary~\ref{cor: debt prop}\ref{cor: debt prop, no new debt post-trns}.
In particular,
\begin{equation}\label{eq: debt red times}
    t_0 \le t_S \le t_0  + 2h_{t_0}.
\end{equation}
From~\eqref{eq: tQ > t - C} applied with $S'=S$, we obtain, using~\eqref{eq: debt red times}, that
\begin{equation}\label{eq: t-t0 < const}
    t-t_0  \le t-t_S + 2h_{t_0} \le \tfrac1{5q}H_t + 2h_{t_0}\le\tfrac4{5q}H_{t_0} + 2h_{t_0}.
\end{equation}
Using~\eqref{eq: doubling condition}, this implies that $h_{t} \le 2h_{t_0}$, namely that $|[t_0,t]\cap \tms_2|\le 1$.

Next, writing $S'=\min \,\seg_{t_0}\!\left(I_{t_0}(S)\right)$, we define
\[\cev J := \big(I_{t_0}(S) \cap (-\infty,S]\big) \setminus S'.\]
Let $S_0 := \min \anc_{t_0}(S)$. Observe that, since $h_{t} \le 2h_{t_0}$, we have $\cev J \subseteq I_{t_0}(S_0) \cap (-\infty, S_0 + h_{t_0}]$, so that
$\seg_{t'}(\cev J) \subset I_{t_0}(S)$ (here $(-\infty, S_0]$ would have been insufficient, as it is possible that $\max \anc_{t_0}(S)=S_0 + h_{t_0}$). As $\spr_{t_0}(S) = 0$, we  have, by definition, $\spr_{t_0}(S_0) = 0$, so that, from~\ref{prop: debt duration bdd strong, sim no debt} we obtain $d_{t_0}(\cev J) = 0$.
At most $2\str h_t$ segments could be disrupted at times $[t_0+1,\dots,t_0+2h_{t_0}]$, each disrupted segment has at most $h_t$ debt (by Proposition~\ref{prop: d + phi <= h}), and debt can only be created via a disruption (by Corollary~\ref{cor: debt prop}\ref{cor: debt prop, no new debt no dis}), which implies $d_{t_0 + 2h_{t_0}}(\cev J) \le 2\str h_t^2$.

We are left with showing that $\Delta d_{t'}(\cev J) \le -1$ for all $t_0+2h_{t_0} \le t' \le t$, as since 
$d_{t}(\cev J) > 0$, 
this would imply $t-t_0 \le 2\str h_t^2 + 2h_{t_0}$, which, using~\eqref{eq: debt red times}, would imply \ref{prop: debt duration bdd strong, bdd} via
\[t-t_S  \le t-t_0 \le 2\str h_t^2 + 2h_{t_0} \le 2\str h_t^2 + 2h_t.\] 

Similarly to before, by Observation~\ref{obs: lookahead dies out} together with~\eqref{eq: t-t0 < const}, for all $Q\in \cev J$ and all ${t'}\in [t_0+2h_{t_0},\dots,t]$, we have $\ell_{t'}(Q)=0$,
so that $\Delta d_{t'}(Q)\le 0$, by Corollary~\ref{cor: debt prop}\ref{cor: debt prop, no new debt post-trns}. Thus, to prove that $\Delta d_{t'}(\cev J) \le -1$, it would suffice to show the existence of $Q=Q(t') \subset \cev J$ such that $\Delta d_{t'}(Q) \le -1$.

Let $t_0+2h_{t_0} \le t' \le t$. We have
\begin{equation}\label{eq: two occupied segments in J}
    b_{t'}(\cev J \setminus\{S\}) \ge b_{t'}(\cev J) - 2h_{t_0} \ge \dispath{t_0}{H_{t_0}}(\cev J) -\floor{qt'}+\floor{qt_0} -2h_{t_0} \ge \H_{t_0} - 2h_{t_0} -\floor{qt'}+\floor{qt_0} > h_{t},
\end{equation}
where the first inequality is by containment, the second is by Observation~\ref{obs: pathcol mono time}, the third uses \eqref{eq: alert-interval-def} and the forth is obtained by plugging in~\eqref{eq: t-t0 < const}. 
Denote
\[S_1 := S_{t'}(\max B_{t'-1}(\cev J \setminus \{S\}) + \{(1,0)\}) \quad \text{and} \quad
D := \min \{Q\in \seg_{t'}(\cev J) : Q \ge S_1, d_{t'-1}(Q) > 0\},\]
observing that $S_1$ is well-defined by~\eqref{eq: two occupied segments in J} and $D$ is well-defined as $S\ge S_1$ and $d_{t'-1}(S) > 0$.

We now show the following statements from which $\Delta d_{t'}(D)=-1$ follows by~\eqref{eq: lambda def} and \eqref{eq: debt def}.
\begin{itemize}
    \item $\lambda_{t'}(D)=S_1$.
    \item $\tilde d_{t'}\big(\seg_{t'}([S_1,D))\big) = 0$.
\end{itemize}

To see the first item, we verify the conditions of \eqref{eq: lambda def}. 
Note that, by definition, $S_1 \le D$ and $S_1 > S_{t'}(\min B_{t'})$, as, by~\eqref{eq: two occupied segments in J} there are at least two occupied segments in $\cev J \setminus \{S\}$. To see that $\Delta \ppo_t(S_1) \ge 1$ we wish to apply Proposition~\ref{prop: mu not ex nihilo simple}\ref{prop: mu not ex nihilo simple: itm fogarty}. The only non-trivial part is showing the existence of $x\in B_{t'-1}$ such that $x + (1,0)\notin B_{t'-1}(S_1)$. Whenever $S_1\neq S$, this holds for $x=\max B_{t'-1}(\cev J \setminus \{S\})$ by the definition of $S_1$. The remaining case is $S_1=S=D$, where we have $|B_{t'-1}(D)| = \phi_{t'-1}(D) \le h_{t'-1} - d_{t'-1}(D) < h_{t'-1} = |D|$ by Proposition~\ref{prop: d + phi <= h} and the fact $\ell_t(D) = 0$. Hence a suitable $x\in S_1$ exists.

To see the second item, Let $Q\in \seg_{t'}([S_1,D))$. Using $\ell_{t'-1}(Q) = \ell_{t'}(Q)=0$ together with Proposition~\ref{prop: mu not ex nihilo simple}\ref{prop: mu not ex nihilo simple: itm simple}, we obtain $\tilde d_{t'}(Q) = d_{t'-1}(Q)$. The item then follows by the definition of $D$.
\end{proof}

\subsection{Potential growth:  proof of Proposition~\ref{prop: Phi grows 1}}\label{subsection: growth, eighth plane}

This section consists of the proof of Proposition~\ref{prop: Phi grows 1}, showing that $\Delta\Phi_s(\seg_t) + \Delta d_s(\seg_t) \ge 1$ for every $s\le t$, whenever $B_{s-1}\ne \emptyset$.
Recalling that $\Phi_s(S) = \sum_{S'\in \anc_s(S)} \Phi_s(S')$ and $d_s(S)=\sum_{S'\in \anc_s(S)}d(S')$, it suffices to prove the proposition only for the case $s = t$.

Fix $t\in \N$ and denote 
\[\lowphiseg := \{S\in\seg_{t}: \Delta \Phi_t(S) + \Delta d_t(S) < 0\}\quad \text{and}\quad  S_1: = S_t(\max B_{t-1} + \{(1,0)\}).\]
We establish the proposition in two steps. First we show that for every $S\in\lowphiseg$ we have $\Delta \Phi_t(S) + \Delta d_t(S) + \Delta \Phi_t(\lambda_t(S)) + \Delta d_t(\lambda_t(S)) \ge 0$. and that the map $S\mapsto \lambda_t(S)$ is one-to-one in $\lowphiseg$. Then we show that $\Delta \Phi_t(S_1) + \Delta d_t(S_1) \ge 1$, and that $S_1 \ne \lambda_t(S)$ for any $S\in\lowphiseg$. From these three steps the proposition readily follows.

Let $S\in \lowphiseg$, denote $S'=\lambda_t(S)$, and observe that necessarily $\Delta \ppo_t(S) + \Delta d_t(S) \neq \Delta \ppo_t(S) + (-\Delta \ppo_t(S))_+  \ge 0$. Hence, $\Delta d_t(S) \ne (-\Delta \ppo_t(S))_+$ so that by~\eqref{eq: debt def}, $d_t(S) = \big(\tilde d_t(S) - \Delta \Phi_t(S')\big)_+$, and, in particular,  $S'\neq -\infty$.
Thus, using $\tilde d_t(S) = d_{t-1}(S) + (-\Delta\Phi_t(S))_+$, we obtain
\[\Delta d_t(S) \ge \tilde d_t(S) - \Delta \Phi_t(S') - d_{t-1}(S) = (-\Delta \Phi_t(S))_+ - \Delta \Phi_t(S'),\]
so that
\begin{equation}\label{eq: lower bound on Delta Phi + Delta d}
   \Delta \Phi_t(S) + \Delta d_t(S) \ge \Delta \Phi_t(S) + (-\Delta \Phi_t(S))_+ - \Delta \Phi_t(S') = (\Delta \Phi_t(S))_+ - \Delta \Phi_t(S'), 
\end{equation}
using the fact that $x + (-x)_+ = x_+$ for any $x\in\R$. Note that as
$S\in\lowphiseg$, the left hand size must be negative, so that $S' \neq S$. Again by~\eqref{eq: debt def} we have $\tilde d_t(Q) = 0$ for every $S'\le Q < S$. This implies that $Q\notin \lowphiseg$, and hence $S\mapsto\lambda_t(S)$ is one-to-one in $\lowphiseg$.
In particular, $\tilde d_t(S') = 0$, which implies $\Delta d_t(S') \ge 0$. 
Using this together with~\eqref{eq: lower bound on Delta Phi + Delta d}, we obtain
\begin{align*}
    \Delta \Phi_t(S) + \Delta d_t(S) + \Delta \Phi_t(S') + \Delta d_t(S') &\ge (\Delta \ppo_t(S))_+ - \Delta \ppo_t(S') + \Delta \ppo_t(S') + \Delta d_t(S')
    \\&= (\Delta \ppo_t(S))_+ + \Delta d_t(S') \ge \Delta d_t(S') \ge 0,
\end{align*}
concluding the first step in our proof plan.

Finally, consider $S_1$.
Observe that, by Proposition~\ref{cor: small fire -> simple and debt free}, $d_{t-1}(S_1)=0$ so that $\Delta \Phi_t(S_1) + \Delta d_t(S_1)\ge \Delta \ppo_t(S_1)$, which, by Proposition~\ref{prop: mu not ex nihilo simple}\ref{prop: mu not ex nihilo simple: itm fogarty}, implies $\Delta \Phi_t(S_1) + \Delta d_t(S_1) \ge 1$.
By the fact that all segments to the right of $S_1$ are non-occupied and thus free of debt by Proposition~\ref{cor: small fire -> simple and debt free}, we obtain $S_1\neq \lambda_t(S)$ for any $S\in\seg_t$.
\qed

\subsection{Controlling the range: proof of Lemma~\ref{lem: phi < r}}\label{subsection: phi < r}

This section is dedicated to the proof of Lemma~\ref{lem: phi < r}, stating that $\phi_t(S) \le \rng_t(S)$ for every $t\in\N$ and $S\in\seg_t$. We start by proving the following claim, showing that only simple segments consolidate.

\begin{clm}\label{clm: consolidation implies simple}
    Let $t\in\N$ and $S\in\seg_t$ such that $\spr_{t-1}(S)=1$ and $\spr_t(S) = 0$. Then $\ell_{t-1}(S) = 0$.
\end{clm}
\begin{proof}
    First, observe that $\ct_{t-1}(S) = 1$ by~\eqref{eq: nsim-def} and~\eqref{eq: countdown-def}. It is thus sufficient to show that $\ell_{t-1}(S) < \ct_{t-1}(S)$. 
    Denote $t_0 := \max\{s <  t : \spr_{s-1}(S) = 0\}$ for the time-step after the last time $S$ was simulative before $t$, and $t_2 := \max\{s\le t : s\in\tms_2\}$ for the last doubling time-step up to time-step $t$ (which could be $-\infty$ if no such time-step exist). We show by induction $\ell_{s}(S) < \ct_s(S)$ for every $s\in [max(t_0,t_2), t-1]$.
    
    The basis of the induction consists of two cases.
    For $s = t_0$, by the definition of $t_0$, one of the first two cases of~\eqref{eq: countdown-def} must hold and hence $\ct_{t_0}(S) \ge h_{t_0}$. Hence by \eqref{eq: look-ahead-def},
    \[\ell_{t_0}(S) = (\ell_{t_0-1}(S)-1)_+ = h_{t_0-1}-1 \le h_{t_0}-1 < \ct_{t_0}(S).\]
    For $s=t_2>t_0$, we have $\ct_{t_2-1}(S) > 0$. Hence, by Observation~\ref{obs: countdown large after doubling}, we have $\ct_{t_2-1} > h_{t_2-1}$, so that, by \eqref{eq: look-ahead-def}, 
    \[\ell_{t_2}(S) = (\ell_{t_2-1}(S) - 1)_+ \le h_{t_2-1} - 1 < \ct_{t_2-1} -1 \le \ct_{t_2}(S).\]

    At any later time-step $s\in [t_0+1,t-1]$ we know that $S$ is spreading and no doubling occurs, and hence either $\ell_{s}(S) = 0$, or $\Delta \ell_s(S) = \Delta\ct_s(S) = -1$, and the claim follows by induction.
\end{proof}

Next, we introduce the following notations, used only for the proof of the lemma. 
Denote 
\[(p_t([a,b]),t):=\min {B_t([a,b])}\]
for the leftmost occupied vertex in an interval $[a,b]$,
so that given $S\in\seg_t$ such that $\chi_t(S) = 0$  we have $B_t(S)\in\{\{(p_t(S),t)\},\emptyset\}$ by~\eqref{eq: Bt def contains}. For such a segment we write
\[w_t(S) := \begin{cases}\big|[\,p_t(S),\max S\,]\big| & B_t(S)=\{(p_t(S),t)\}\\0&B_t(S)=\emptyset\end{cases}\qquad\text{and}\qquad w_{t-1}(S) := \sum_{S' \in \anc_{t-1}(S)} w_{t-1}(S'),\]
Finally we write, 
\[\f_t(S) := \big|\FF_t \cap ([p_t(S),\max S] \times L_t(S))\big|.\]
In addition, for the purpose of this section we use extended notation $L^{\ell}_t(S)=S\times [t,t+\ell]$ so that $L_t(S)=L^{\ell_t(S)}_t(S)$.

From Claim~\ref{clm: accecibility implies faithful} we quickly draw the following.
\begin{clm}\label{clm: w - f < r}
    Let $S\in\seg_t$ with $\spr_t(S) = 0$ and $r_t(S) > 0$. Then $w_t(S)-\f_t(S) \le r_t(S)$.
\end{clm}
\begin{proof}
Recall our definition, $r_t^\ell(A;I)=|\{x\in G_t\ :\ \exists (x_0,\dots,x)\in P_t^\ell(I), x_0\in A|$.
   As $r_t(S) > 0$ we have
   $\rng_t^{\ell_t(S)}(\{(p_t(S),t)\};S) > 0$. Hence by applying Claim~\ref{clm: accecibility implies faithful} for $\ell = \ell_t(S)$ we get $w_t(S) - \f_t(S) \le r_t(S)$.
\end{proof}

With these claims in hand, we turn to the proof of the lemma itself.

\begin{proof}[Proof of Lemma~\ref{lem: phi < r}]
    Let $S\in\seg_t$.
    Observe that the lemma is straightforward when $\spr_t(S) = 1$ (as in that case $\phi_t(S) = \rng_t(S)$ by \eqref{eq: phi def}), and when $\phi_t(S) = 0$ (as $\rng_t(S) \ge 0$ by definition).
    Hence we assume $\spr_t(S) = 0$ and $\phi_t(S) > 0$. 

    To establish the lemma, we inductively prove the following:
    \begin{enumerate}[label=\textnormal{(\alph*)}]
        \item\label{clm: geom ineq phi 2, r>0} $r_t(S) > 0$,
        \item\label{clm: geom ineq phi 2, phi < w - f} $\phi_t(S) \le w_t(S) - \f_t(S)$,
        \item\label{clm: phi<r}
        $\phi_t(S)\le r_t(S)$.
    \end{enumerate}
    Observe that~\ref{clm: geom ineq phi 2, r>0} and \ref{clm: geom ineq phi 2, phi < w - f}, together with Claim~\ref{clm: w - f < r} imply~\ref{clm: phi<r} from which, in turn, the lemma follows.

    We study two cases according to the value of $\spr_{t-1}(S)$, the first of which is proved directly and serves also as the basis of the induction.

    \textbf{Case 1 \gray{-- Consolidation}}: $\spr_{t-1}(S) = 1$. Observe that $\ell_{t-1}(S) = 0$ by Claim~\ref{clm: consolidation implies simple}. Hence
    \[0 < \phi_{t}(S) = \phi_{t-1}(S) - |\Delta F_t(S)| = |B_{t-1}(S)| - |\Delta F_t(S)| = |B_{t-1}(S)| - |\FF_{t} \cap L_t(S)|,\]
    where the first equality is by~\eqref{eq: phi def}, the second -- by~\eqref{eq:simple phi}, and the last equality -- by Observation~\ref{obs: new ff after simple 2}.

    Clearly $\rng_{t-1}^0(B_{t-1}(S);S) = |B_{t-1}(S)|$, and as $F_t\cap B_{t-1}=\emptyset$,
    we also have $\rng_t^0(B_{t-1}(S);S) = |B_{t-1}(S)|$. Applying Claim~\ref{clm: faithful upwards reduction} with $\ell_1 = 0$, $\ell_2 = \ell_t(S)+1$, applied to $B_{t-1}(S) \subset S\times\{t-1\}$ at time-step $t$, together with the fact that $|\FF_{t} \cap L_t(S)| < |B_{t-1}(S)|$,
    we deduce that $\rng_t^{\ell_t(S)+1}(B_{t-1}(S);S) > 0$.
    Hence, there must exist a vertex of $S\times\{t\}$ from which starts a $t$-path ending in $S\times\{t+\ell_t(S)\}$.
    By~\eqref{eq: Bt def contains}, the leftmost such vertex will form $B_t$, so that indeed $B_t\ne\emptyset$ and $r_t(S) > 0$, proving \ref{clm: geom ineq phi 2, r>0}.
    
    Towards proving~\ref{clm: geom ineq phi 2, phi < w - f}, we partition $\Delta F_t(S)$ via
    \[|\Delta F_t(S)| = \f_t(S) + \left|\FF_t\left(L_t^{\ell_t(S)}\big([\min S, p_t(S)\big)\right)\right|.\]
    Hence, substituting $\phi_t(S) = \phi_{t-1}(S) - |\Delta F_t(S)|$ via \eqref{eq: phi def}, we must show \begin{equation}\label{eq: lemma proof, consolidation}
        |B_{t-1}(S)| = \phi_{t-1}(S) \le w_t(S) + |\Delta F_t(S)| - \f_t(S) = w_t(S) +  \left|\FF_t\left(L_t^{\ell_t(S)}\big([\min S, p_t(S)\big)\right)\right|.
    \end{equation}
    To see this, observe $|B_{t-1}([p_t(S),\max S])| \le |[p_t(S),\max S]|=w_t(S)$. 
    On the other hand, by definition we have $\rng_t^0(B_{t-1}([\min S, p_t(S)));S) = \big|B_{t-1}([\min S, p_t(S))\big|$, while by~\eqref{eq: Bt def contains} we have $\rng_t^{\ell_t(S)}([\min S, p_t(S));S) = 0$. From this we obtain  
    $|B_{t-1}([\min S, p_t(S)))| \le |\FF_t(L_t^{\ell_t(S)}([\min S, p_t(S)))|$, by
   applying Claim~\ref{clm: faithful upwards reduction} to $B_{t-1}([\min S, p_t(S)))$ at time-step $t$ with $\ell_1 = 0$ and $\ell_2 = \ell_t(S)+1$. All in all
    \[|B_{t-1}(S)| = |B_{t-1}([\min S, p_t(S)))| + |B_{t-1}([p_t(S),\max S])| \le w_t(S) +  \left|\FF_t\left(L_t^{\ell_t(S)}\big([\min S, p_t(S)\big)\right)\right|.\]
 
    \textbf{Case 2 \gray{-- Simulative step}}: $\spr_{t-1}(S) = 0$. We observe 
    \[0 < \phi_t(S) = \phi_{t-1}(S) - |\Delta F_t(S)| \le r_{t-1}(S) - |\Delta F_t(S)|,\]
    where the first equality is by~\eqref{eq: phi def} and the second inequality follows from the induction hypothesis at time-step $t-1$.
    By Observation~\ref{obs: delta f in slab} we thus have $|\FF_t(S\times[t +\ell_{t-1}(S),t+\ell_t(S)])| < \rng_{t-1}(S)=\rng_{t-1}^{\ell_{t-1}(S)}(S)$.
    Applying Claim~\ref{clm: faithful upwards reduction} with $A = S\times\{t-1\}$, $\ell_1 = \ell_{t-1}(S)$, and $\ell_2 = \ell_t(S)+1$, at time-step $t$ we deduce that $\rng_{t-1}^{\ell_t(S)+1}(S) > 0$, so that, as before, $r_t(S) > 0$ and $B_t(S) \ne \emptyset$, proving \ref{clm: geom ineq phi 2, r>0}.

    Next, we establish~\ref{clm: geom ineq phi 2, phi < w - f}. By the induction hypothesis at time-step $t-1$, we have $\phi_{t-1}(S) \le w_{t-1}(S) - \f_{t-1}(S)$, so that $\phi_t\le w_{t-1}(S) - \f_{t-1}(S) - |\Delta F_t(S)| $.
    Hence, it would suffice to show
    \[w_{t-1}(S) - \f_{t-1}(S) - |\Delta F_t(S)|  \le w_t(S) - \f_t(S),\]
    or, equivalently, \begin{equation}\label{eq: delta tilde f < delta w + delta f}
        \Delta \f_t(S) \le |\Delta F_t(S)|+\Delta w_t(S).  
    \end{equation}

Next, we bound $\Delta\f_t(S)$ through
    \begin{align}
    \begin{split}
\Delta\f_t(S)  &= \Big|\FF_t\big([p_t(S),\max S]\times [t+h_{t-1},t+h_t]\big)\Big|-
    \Big|\FF_t\big([p_{t-1}(S),p_{t}(S))\}\times [t,t+h_{t-1}]\big)\Big|\notag\\&\hspace{12.5pt}\phantom{\Big|\FF_t\big([p_t(S),\max S]\times [t+h_{t-1},t+h_t]\big)\Big|}-\Big|\FF_t\big([p_{t-1}(S),\max S])\times \{t-1\}\big)\Big|\notag
    \end{split}
    \notag \\    &\le    \Big|\FF_t\big([p_t(S),\max S]\times [t+h_{t-1},t+h_t]\big)\Big|-    \Big|\FF_t\big([p_{t-1}(S),p_{t}(S))\}\times [t,t+h_{t-1}]\big)\Big|\notag\\&\le \Delta F_t(S) - \Big|\FF_t\big([p_{t-1}(S),p_{t}(S))\}\times [t,t+h_t]\big)\Big|,
    \label{eq: bnd delta f tilde}
    \end{align}
    where the last inequality follows by Observation~\ref{obs: delta f in slab}.
    When $p_{t}(S)=p_{t-1}(S)$ this directly implies \eqref{eq: delta tilde f < delta w + delta f} as in this case $\Delta w_t(S)=0$. When $p_{t}(S)=p_{t-1}(S)+1$, which implies $\Delta w_t(S)\ge 1$, we obtain \eqref{eq: delta tilde f < delta w + delta f} 
    by observing $p_{t}(S)>p_{t-1}$ implies $\FF_t\big([p_{t-1}(S),p_{t}(S))\}\times [t,t+h_t]\big)\neq \emptyset$.
    
    We are thus left with the case $p_{t}(s)>p_{t-1}(S)+1$. For this to be the case,  $t\in\tms_2$ must be a doubling time, and, denoting $S_1 < S_2$ for the segments in $\seg_{t-1}$ satisfying $\anc_{t-1}(S) = \{S_1,S_2\}$, we must have $p_t(S) \in p_{t-1}(S_2) + \{0,1\}$ and $B_{t-1}(S_1) \ne \emptyset$.
    Observe that in this case $\f_t(S) = \f_t(S_2)$ and $w_t(S) = w_{t}(S_2)$, and 
    \begin{equation}\label{f(S) > f(S_1) + f(S_2)}
        \f_{t-1}(S) \ge \f_{t-1}(S_1) + \f_{t-1}(S_2).
    \end{equation}
    

    Hence, for $i\in\{1,2\}$, writing 
    $\Delta F_t(S_i)=\Big|\FF_t\big(S_i\times [t+h_{t-1},t+h_{t}]\big)\Big|$, so that $\Delta F_t(S) = \Delta F_t(S_1) + \Delta F_t(S_2)$, we obtain
    \begin{align*}\Delta\f_t(S) + \f_{t-1}(S_1)&\le\f_t(S) - \f_{t-1}(S_2)=\Delta \f_t(S_2)\\
    &\le \Delta F_t(S_2) - \Big|\FF_t\big([p_{t-1}(S_2),p_{t}(S_2))\}\times [t,t+h_t]\big)\Big|\\&\le \Delta F_t(S_2) + \Delta w_t(S_2)=
    \Delta F_t(S_2) + w_t(S)-w_{t-1}(S_2)\\&=|\Delta F_t(S)|-|\Delta F_t(S_1)| + \Delta w_t(S)+w_{t-1}(S_1),
    \end{align*}
    where the first inequality is by~\eqref{f(S) > f(S_1) + f(S_2)},
    the second follows by applying~\eqref{eq: bnd delta f tilde} to $S_2$,
    and the third uses the arguments that follows~\eqref{eq: bnd delta f tilde}, namely that 
    $p_t(S)=p_{t-1}(S_2)+1$ implies 
    $\FF_t\big(\{p_{t-1}(S_2)\}\times [t,t+h_t]\big)\neq \emptyset$.
    
    Hence, to obtain~\eqref{eq: delta tilde f < delta w + delta f} it would suffice to show  
    \begin{equation} w_{t-1}(S_1)\le |\Delta F_t(S_1)|+\f_{t-1}(S_1).  \label{eq: w t-1 bound}
    \end{equation} 
    If $B_{t-1}(S_1)=\emptyset$, this is straightforward as $w_{t-1}(S_1)= 0$. Otherwise, apply Claim~\ref{clm: faithful upwards reduction} with $A=\{(p_{t-1}(S_1),t-1)\}\subset S_1\times\{t-1\}$, $\ell_1 = h_{t-1}$, and $\ell_2 = h_t+1$ at time-step $t$, observing that $\rng_t^{h_t+1}(S_1\times\{t-1\};S) = 0$ by~\eqref{eq: Bt def contains} (together with $p_t(S) > p_{t-1}(S_1) + 1$), to obtain 
    \begin{equation}\label{eq: r(S_1) <= F(S_1)}
        \rng_{t-1}(S_1) \le |\Delta F_t(S_1)|   
    \end{equation}
    By our assumption that $B_{t-1}(S_1) \ne \emptyset$, and as $\spr_{t-1}(S_1) = 0$,
    we conclude that $\rng_{t-1}(S_1) > 0$. Hence, Claim~\ref{clm: w - f < r} applied to $S_1$ at time-step $t-1$, together with~\eqref{eq: r(S_1) <= F(S_1)}, yields~\eqref{eq: w t-1 bound}.

\end{proof}

{
}

\section{The Directed Half-Plane}\label{section: directed half plane}

This section is dedicated to the study of the Containment Game on the \emph{directed half-plane} $\Zhp=(V,E)$, the sub-graph of $\ZZ$ given by
\[V = \{(x,y)\in \Z^2 : y \ge 0\},\,\,\, E = \{\left((x,y), (x+i,y+1)\right) : |x| \le y, i\in\{-1,0,1\}\}.\]
This will serve as a major stepping stone in handling the full $\ZZ$.

Recall Spreader's strategy for this graph given in \eqref{eq: Bt def directed half plane}.
Firstly we observe the following 
analogue of Proposition~\ref{prop: Phi grows 1}, which it readily implies when combined with Proposition~\ref{prop: mu not ex nihilo simple}\ref{prop: mu not ex nihilo simple: itm fogarty}.

\begin{prop}\label{prop: Phi grows by 2, quarter}
Let $s \le t$. If $B_{s-1} \ne \emptyset$ then $\Delta\phi_s(\seg_t) + |\Delta F_s(\seg_t)| + \Delta d_s(\seg_t) \ge 2$.
\end{prop}

We also slightly augment our analytic results, in order to establish the following analogues of 
Claim~\ref{obs: dead ends Zei}
and
Proposition~\ref{prop: B(t) is virtually-viable}, which we establish in Section~\ref{subsection: devirtualisation}

\begin{prop}\label{prop: devirt quarter pl} There exists an explicit Spreader strategy $B' =(B'_t)_{t\in\N}$ satisfying the following for every $t\in\N$. 
\begin{enumerate}[label=\textnormal{(\alph*)}]
    \item\label{prop: devirt quarter pl, viable} $|B'_t| \le O(h_{t}^6) + \frac{2t}{h_{t}} + |B_0|$.
    \item\label{prop: devirt quarter pl, fire-preserving} $B'_t \ne \emptyset \iff B_t \ne \emptyset$.
\end{enumerate}
\end{prop}

Finally, as the directed half-plane will serve us as a stepping stone towards treating $\ZZ$ and our strategy there will involve playing in parallel up to four copies of $\Zhp$ games, we must make sure that these don't step on each others' toes. To this end we define the notion of play area.

We define the \emph{play area} of a game $\GG$ at time-step $t$ by
\begin{equation}\label{eq: play area}\area^{\GG}_t := \bigcup_{s=0}^t B_{s} \cup F_{s}.\end{equation}
Namely, all the vertices that were ever occupied together with all deleted vertices counted by $F_s$ up to time-step $t$. Here we often omit the superscript $\GG$ as only a single game is considered. The full notation will be used only in Section~\ref{section: plane}.
For the analysis of the play area, we introduce the following notion. Letting $D\subset \Z^2$, we define the \emph{infinite trapezoid generated by $D$}, denoted by $\trap(D)$, to be the convex hull in $\Z^2$ of
$D+\{k(1,1),k(-1,1)  :\ k\in \N\}$.
The following proposition, established in Section~\ref{subsection: bounding ga}, bounds $\area^\GG_t$.
\begin{prop}\label{prop: ga prop}
The following hold for all $t\in \N$.
\begin{enumerate}[label=\textnormal{(\alph*)}]
     \item\label{prop: ga prop, item: dead front} If $B_t = \emptyset$
     then $\area_t \subset \Z\times [0,t]$,
    \item\label{prop: ga prop, item: future ga} For every $t' > t$ we have $\area_{t'} \setminus \area_{t} \subseteq \trap(B_t)$,
    \item\label{prop: ga prop, past rectangle}  $\area_t\!\subseteq\! \left[x_{\min}\!-\!1, x_{\max} \!+\!1\right]\!\times\!\Z$, where $x_{\min}$ and $x_{\max}$ are the columns of  $\min \bigcup_{s=0}^{t-1} B_s$ and $\max\bigcup_{s=0}^{t-1} B_s$, respectively. 
\end{enumerate}
\end{prop}
{
}

\subsection{Avoiding dead-ends -- proof of Proposition~\ref{prop: devirt quarter pl}}\label{subsection: devirtualisation}

When constructing the modified strategy $B'$, we aim to ensure that $B'_t$ always contains $\min B_t$, the leftmost occupied vertex, by including any vertex that could take this role within a bounded number of turns. We proceed to define this formally.

As in the proof of Proposition~\ref{prop: B(t) is virtually-viable} in Section~\ref{subsection: sparsity}, we define $\origin_s(y)$ for $s\le t$, as the
leftmost starting vertex of a $s$-path ending in a given $y\in B_t$. Denote $D_t^{\ell}$ for the set of vertices $x\in B_t$ for which there exists a Container strategy under which $x=\origin_t(\min B_{t+\ell})$ when the strategy $B$ is played against it.

Define $B' = (B'_t)_{t\in\N}$ inductively as follows. 
\begin{align}
    B'_0 &= B_0,\notag\\
    B'_t &= B_t^{\cR H_t} \cup \left((B'_{t-1}+\{(-1,1),(0,1),(1,1)\}) \cap \bigcup_{\ell=0}^{\cR H_t}D_t^\ell\right).\label{eq: def Ct}
\end{align}

\begin{proof}[Proof of Proposition~\ref{prop: devirt quarter pl}]
By \eqref{eq: def Ct} $B'_t$ is clearly a viable strategy.
Let $t\in\N$. Firstly we show that $\min B_t \in B'_t$ which implies $B'_t\ne \emptyset \iff B_t \ne \emptyset$ (recalling  that $D^{\cR H_t}_t \subseteq B_t$). This we show using induction. Since $B'_0 = B_0$, we clearly have $\min B_0 \in B'_0$.
Assume this holds up to time-step $t-1$, and let $y := \min B_{t}$. By~\eqref{eq: def Ct} and the definition of $D_t^\ell$, it suffices to show that $\origin_s(y) \in B'_s$ for some $s\in [t-3H_t,t]$.
Indeed, let $P = (x_s,\dots,x_t = y)$ be an $s$-path of maximum length leading to $y$, where $x_i \in B_i$ for $i=s,\dots,t$. If $|P| \ge 3H_t$, then $\origin_{t-3H_t}(y) = x_{t-3H_t} \in B_{t-3H_t}^{3H_t} \subset B'_{t-3H_t}$, and we are done. Otherwise, either $x_s\in B_0$, or by~\eqref{eq: Bt def directed half plane}, we must have $\origin_{s}(y)+(1,-1) = \min B_{s-1}$, so that
$\origin_s(y) = x_s = \min B_s$. Hence, by the induction hypothesis, $x_s \in B'_s$, and again, we are done.


Finally, we establish the bound on $|B'_t|$.
Let $\ell\in[0,\cR H_t]$ and denote the points of $D^\ell_t$ by $x_1<\dots<x_{N}$.
We denote by $P_i$ a $t$-path of length $\ell$ starting from $x_i$ and ending in the minimal $y_i$, such that there exists a Container strategy under which $y_i=\min B_{t+\ell}$. Observe that, necessarily $y_i\le y_{i+1}$ for all $i$.

By Observation~\ref{obs: far paths no intersection}, at most $\ell-1$ distinct two-sided $(t,\ell)$-paths can intersect at a point. Hence we must have $y_i < y_N$ for every $i < N-\ell$. Since $x_N\in D_t^\ell$, Container, which deletes up to $q+1$ vertices every turn, must have a strategy which blocks $P_i$, for every $i\in 1,\dots,N-1$. Observing that $P_i \cap \FF_{t+\ell}=P_i \cap (\FF_{t+\ell}\setminus \FF_{t})$, we have 
\[|\{i: P_{i} \cap \FF_{t+\ell}\neq \emptyset \}| \le  \ell^2(\str+1).\]
We deduce that $N - \ell \le \ell^2(\str+1)$, so that $|D_t^\ell| = N \le \ell^2(\str+1) + \ell = O(\ell^2).$
Taking a union bound we obtain,
\[\left|\bigcup_{\ell = 0}^{\cR H_t} D_t^\ell\right| = O(H_t^3) = O(h_t^6).\]

Applying Proposition~\ref{prop: B(t) is virtually-viable} we obtain
\[|B_t'| \le \left|B_t^{3H_t}\right| + \left|\bigcup_{\ell = 0}^{\cR H_t} D_t^\ell\right| = O(h_t^6) + \frac{t}{h_t}.\qedhere\]
\end{proof}

\subsection{Bounding the play area -- proof of Proposition~\ref{prop: ga prop}}\label{subsection: bounding ga}
Let $t\in\N$.
To prove~\ref{prop: ga prop, item: dead front}, assume that $B_t = \emptyset$. We claim that $\ell_{t'}(S) = 0$ for every $t'\in[t-h_t,t]\cap \N$ and $S\in\seg_{t}$.
Indeed, as $b_t=0$, Proposition~\ref{prop: simulated path} implies that $\spr_{s}(S) = 1$ for every $s\in [t-2h_t,t]\cap \N$, and so $\ell_{t'}(S) = 0$ for every $t'\in [t-h_t,t]\cap \N$ by~\eqref{eq: look-ahead-def}.
By definition, $F_{s}(S)\subset S\times (-\infty, s+\ell_{s}(S)]$. Since $\ell_s(S)\le h_s$ for all $s$, item~\ref{prop: ga prop, item: dead front} follows.

To prove~\ref{prop: ga prop, item: future ga}, note that for every $n\in\N$ for which $B_{t+n}\neq \emptyset$, the interval connecting $\min B_t+(-n,n)$ and $\max B_t+(n,n)$ must contain $B_{t+n}$.
Next, let $n\ge 1, S\in\seg_{t+n}$ and $v\in \Delta F_{t+n}(S)$. Clearly $v \in S \times \{t+n,...,t+n+h_{t+n}\}$. Hence if $S$ is neither the leftmost occupied segment nor the rightmost, we obtain $v \in \trap(B_t)$. Otherwise, $S$ must be simple (by Proposition~\ref{cor: small fire -> simple and debt free}), so that $v\in B_{t+n-1}+\{(-1,1),(0,1),(1,1)\})$ by~\eqref{eq: firefighters def}, and hence $v\in \trap(B_t)$.

To prove~\ref{prop: ga prop, past rectangle}, denote $x_{\min}^s, x_{\max}^s$ for the first coordinate of $\min B_s$ and $\max B_s$, respectively.
Note that a deleted vertex counted towards $F$ at time-step $s$ outside the rectangle $[x^{s-1}_{\min}, x^{s-1}_{\max} ]\times \Z$
must have been counted by either the leftmost occupied segment or the rightmost, which are always simple (again, by Proposition~\ref{cor: small fire -> simple and debt free}). Thus, by~\eqref{eq: firefighters def}, it is of distance at most $1$ from the rectangle. As this holds for any $s \le t$, item~\ref{prop: ga prop, past rectangle} follows.
\qed

\section{The Plane}\label{section: plane}

\subsection{Winning in the plane}\label{subsection: plane}
In this section we construct a winning Spreader strategy in $\ZZ$ from the $\Zhp$ strategy described in Section~\ref{section: directed half plane}. Given the $\Zhp$ results, the proof of Theorem~\ref{thm: sub-linear win} will closely follow the firefighter analogue of~\cite{feldheim20133}*{Theorem~1}, showing that at least three fronts contain occupied vertices at all times.

For simplicity, we assume the game starts at time-step $t=1$ with $B_1 = \{-1,0,1\}^2\setminus\{(0,0)\}$ and $\FF_1=\emptyset$ (See Figure~\ref{fig:plane, initial game}), and that on every subsequent turn Container gets to delete at most $3$ vertices. A more detailed computation would show that Spreader can win also with the occupied set initialised at a single occupied vertex at the origin, but this is not needed for our result and incurs technical complications.

\begin{figure}[!ht]
    \centering
    \subfloat[\label{fig:plane, initial game}]{
    \begin{tikzpicture}[scale=\sca]
        \draw[step=1cm,grayX,thin] (0,0) grid (7,7);
        
        \fill [pattern=my north east lines,pattern color=lightgrayX] (1,6) rectangle (7,7);
        \fill [pattern=my north east lines,pattern color=lightgrayX] (2,5) rectangle (6,6);
        \fill [pattern=my north east lines,pattern color=lightgrayX] (3,4) rectangle (5,5);
        
        \draw[black, very thick,fill=none] (3,4) rectangle (5,5); 
        \draw[black, very thick,fill=none] (4,2) rectangle (5,4); 
        \draw[black, very thick,fill=none] (2,2) rectangle (4,3); 
        \draw[black, very thick,fill=none] (2,3) rectangle (3,5); 

        \node at (3.5,4.5) {\Fire}; 
        \node at (4.5,4.5) {\Fire};
        \node at (4.5,2.5) {\Fire}; 
        \node at (4.5,3.5) {\Fire};
        \node at (2.5,2.5) {\Fire}; 
        \node at (3.5,2.5) {\Fire};
        \node at (2.5,3.5) {\Fire}; 
        \node at (2.5,4.5) {\Fire};
        
        \draw [dashed, line width=0.3mm] (5,5) -- (6,5);
        \draw [dashed, line width=0.3mm] (6,5) -- (6,6);
        \draw [dashed, line width=0.3mm] (6,6) -- (7,6);
        \draw [dashed, line width=0.3mm] (7,6) -- (7,7);
        
        \draw [dashed, line width=0.3mm] (2,5) -- (2,6);
        \draw [dashed, line width=0.3mm] (2,6) -- (1,6);
        \draw [dashed, line width=0.3mm] (1,6) -- (1,7);
        \draw [dashed, line width=0.3mm] (1,7) -- (0,7);
        
        \draw [dashed, line width=0.3mm] (2,2) -- (1,2);
        \draw [dashed, line width=0.3mm] (1,2) -- (1,1);
        \draw [dashed, line width=0.3mm] (1,1) -- (0,1);
        \draw [dashed, line width=0.3mm] (0,1) -- (0,0);

        \draw [dashed, line width=0.3mm] (5,2) -- (5,1);
        \draw [dashed, line width=0.3mm] (5,1) -- (6,1);
        \draw [dashed, line width=0.3mm] (6,1) -- (6,0);
        \draw [dashed, line width=0.3mm] (6,0) -- (7,0);

    \end{tikzpicture}}
    \hspace{8em}
    \subfloat[\label{fig:plane, spilling}]{
    \begin{tikzpicture}[scale=\sca]
    
        \draw[step=1cm,grayX,thin] (0,0) grid (7,7);
        
        \fill [pattern=my north east lines,pattern color=lightgrayX] (0,6) rectangle (7,7);
        \fill [pattern=my north east lines,pattern color=lightgrayX] (0,5) rectangle (6,6);
        \fill [pattern=my north east lines,pattern color=lightgrayX] (1,4) rectangle (5,5);
        \fill [pattern=my north east lines,pattern color=lightgrayX] (2,3) rectangle (4,4);
        \draw[darkgrayX, very thick,fill=none] (2,3) rectangle (4,4); 
        
        \node at (0.5,2.5) {\faShield};
        \node at (1.5,2.5) {\faShield};
        \node at (2.5,2.5) {\faShield};
        
        \node at (3.5,2.5) {\Fire};
        \node at (3.5,1.5) {\Fire};
        \node at (3.5,0.5) {\Fire};
        
        \draw[-stealth, thick, black, fill=none] (3.3,2.7) -- (3.3, 3.2);
        \draw[-stealth, thick, black, fill=none] (3.3,2.7) -- (2.8, 3.2);

        \node[fill opacity=0.5] at (2.5,3.5) {\Fire};
        \node[fill opacity=0.5] at (3.5,3.5) {\Fire};
        
        \draw [dashed, line width=0.3mm] (4,4) -- (5,4);
        \draw [dashed, line width=0.3mm] (5,4) -- (5,5);
        \draw [dashed, line width=0.3mm] (5,5) -- (6,5);
        \draw [dashed, line width=0.3mm] (6,5) -- (6,6);
        \draw [dashed, line width=0.3mm] (6,6) -- (7,6);
        \draw [dashed, line width=0.3mm] (7,6) -- (7,7);
        
        \draw [dashed, line width=0.3mm] (2,4) -- (1,4);
        \draw [dashed, line width=0.3mm] (1,4) -- (1,5);
        \draw [dashed, line width=0.3mm] (1,5) -- (0,5);

        

    \end{tikzpicture}}
    
    \caption{\eqref{fig:plane, initial game} illustrates the initial game in the plane. The four fronts are outlined by rectangles, and the four disjoint infinite trapezoids are outlined by dashed lines. The infinite trapezoid of front $0$ is highlighted by a filling pattern.
    In~\eqref{fig:plane, spilling} front $0$ is re-ignited by front $1$. The initial occupied set of the re-initialised game at front $0$ is depicted in gray, and the infinite trapezoid is highlighted by a filling pattern.
    }
    \label{fig:plane}
\end{figure}
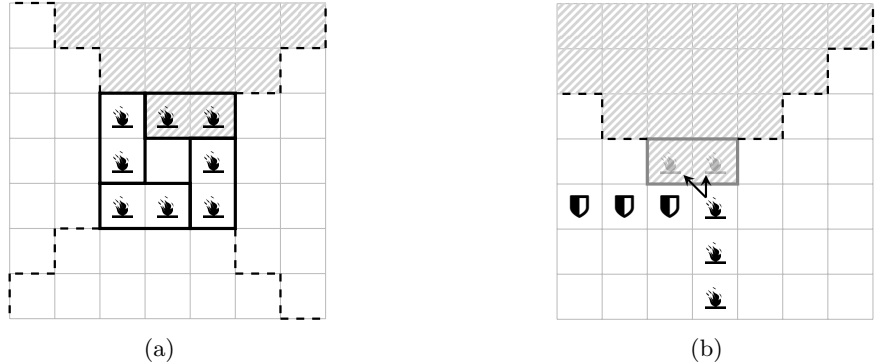

\subsubsection{The strategy}\label{subsubsection: plane strategy}
The strategy of Spreader in the plane consists of two components. The first is a simultaneous implementation of a rotated and translated directed half-plane strategy in up to four disjoint play areas, each corresponding to an occupied front. The second component is the re-ignition of an extinguished front by an adjacent occupied front. We proceed to write this formally.

Writing $\II:= \{0,1,2,3\}$, we denote the four cardinal directions $\{\theta^i\}_{i\in \II}$ by
\[\theta^0 := (0,1),\quad \theta^1 := (1,0),\quad \theta^2 := (0,-1),\quad \theta^3 := (-1,0),\]
in a clockwise fashion. Superscripts in $\II$ are always taken modulo 4; in particular we write $|i-j| = 1$ if $i$ and $j$ are consecutive modulo $4$. We also denote the four secondary directions $\{\theta^{i,i+1}\}_{i\in\II}$
by
\[\theta^{0,1} := (1,1),\quad \theta^{1,2} := (1,-1),\quad \theta^{2,3} := (-1,-1),\quad \theta^{3,0} := (-1,1).\]
The infinite line of distance $d$ from the origin, perpendicular to the $i$-th direction, is denoted by
\[L^i(d) := d\theta^i + \Z\theta^{i+1}.\]
We also introduce a notion of directed \emph{occupation radius} (denoted by $\rad^i_{t}$) and \emph{front-line}  (denoted by $L^i_t$), keeping track of the distance of the front from the origin in every cardinal direction. These are defined as follows.
\begin{align}
    \rad^i_{t} &:= \min\left\{r\ :\ L^i(r) \cap \bigcup_{t'=0}^{t-1}B_{t'}=\emptyset\right\} \label{eq: fire-rad-def},\\
    L^i_t &:= L^i(\rad^i_t).\notag
\end{align}

Throughout we denote by $\GG^i_t$ the half-plane game played in direction $i$ at time-step $t$ employed by the plane strategy. We write $B^i_t := B_t(\seg^{\GG^i_t}_{t})$ and use a similar convention for $\phi,d,\Delta f$ and $b$ to denote the corresponding sets and quantities used for the analysis of the game $\GG^i_t$. Observing that, by Proposition~\ref{prop: no fire -> no ptl no dbt}, on a termination time-step $t$, $B^i_t=\emptyset$ and all of the remaining quantities are $0$, we set these as default values for all times $s$ in which no game is played on the $i$-th front.
Once a front game is \emph{initialised}, it plays until $B^i_t=\emptyset$ at which step it \emph{terminates}.
We use the equality $\GG^i_t=\GG^i_s$ to indicate that the game played at time-step $t$ didn't terminate before time-step $s$. For each front $i\in\II$, the fact that $\GG^i_t$ is \emph{active} (rather than terminated) is indicated by $a^i_t := \ind\{B^i_t \neq \emptyset\}$.
We also set $B_t := \bigcup_{i\in\II} B^i_t$.

The $\ZZ$ game starts in time-step $1$, where the half-plane games are initialised with $B^{\GG^i}_1:=\{\theta^i, \theta^{i,i+1}\}\setminus \FF_1$ (See Figure~\ref{fig:plane, initial game}). In every turn $t > 1$, the strategy plays as follows.
\begin{enumerate}
    \item If $a^i_{t-1} = 1$, a step is played in the $i$-th front game, according to the $\Zhp$ strategy on the corresponding shifted and rotated piece of $\ZZ$.
    \item If $a^i_{t-1} = 0$ then, 
    \begin{itemize}
        \item in case that $B_{t-1} \cap L^i_{t-1} = \emptyset$, nothing happens in this front.
        \item Otherwise, when $B_{t-1} \cap L^i_{t-1} \neq \emptyset$, a new $\Zhp$ game $\GG$ is initialised, rotated towards direction $i$ and shifted by $\rho^i_{t}$, with starting conditions $B_0 := O\setminus \FF_t$,  $\FF_0 := \FF_t$  and $F_0 := O\cap \FF_t$, where
    \begin{equation}\label{eq: init}O := L^i_{t} \cap \left((B^{i+1}_{t-1} + \{\theta^i, \theta^{i-1,i}\}) \cup (B^{i-1}_{t-1} + \{\theta^i, \theta^{i,i+1}\})  \right) .
    \end{equation}
    \end{itemize}
\end{enumerate}
Namely, a terminated game can be re-initialised when the front $L^i_{t-1}$ contains a vertex occupied by Spreader as part of the game played on an adjacent front. In this case $\Delta\rho^i_t=1$, and the initial play area consists of all of the neighbours of this occupied vertex on $L^i_{t-1}+\theta_i$ which are not in the play area of the adjacent front's game (see Figure~\ref{fig:plane, spilling}). Note that this new game may terminate immediately upon creation, but at a given front, a game cannot be initialised on the same turn that another game terminates.

Observe that $\rho^i_1=1$, $\phi^i_1=|B^i_1|=2$, $F^i_1=\emptyset$  and $|F^t|\le 3(t-1)$ for all $i\in \II$.


\subsubsection{Analysis of the strategy}

We start by making the following claim, concerning rotated infinite trapezoids. Given a set $S \subset \Z^2$ we denote the infinite trapezoid of $S$ rotated to direction $i$ by $\trap_i(S)$. This is given by the convex hull of $S+\{k\theta^{i,i+1},k\theta^{i,i-1}  :\ k\in \N\}$. We omit the subscript $i$ when it is clear from context. 
\begin{clm}\label{clm: disjoint trap}
Let $a_1 < a_2 < a_3, b_1 < b_2$, and define $C:= [a_1,a_2] \times \{b_2\}, D := \{a_3\} \times [b_1,b_2]$. Then $\trap_0(C) \cap \trap_1(D) = \emptyset$.
\end{clm}
\begin{proof}
Observe that for all $(x,y)\in \trap_0(C)$ we have $y-x \ge b_2 -a_2$ while 
for all $(x,y)\in \trap_0(D)$ we have $y-x \le b_2-a_3$.
\end{proof}

The analysis of the plane strategy is based on two key propositions. The first states that the play areas of all $\Zhp$ games employed by Spreader are disjoint.

\begin{prop}\label{prop: pl: ga disjoint}
Let $t_0,t_1\in\N$, $i_0,i_1\in\II$. If $\GG^{i_0}_{t_0} \ne {\GG^{i_1}_{t_1}}$ then $\area^{\GG^{i_0}_{t_0}}_{t_0} \cap \area^{\GG^{i_1}_{t_1}}_{t_1} = \emptyset$.
\end{prop}

\begin{proof}
Without loss of generality assume that $i_0=0$ and $t_0 \le t_1$. Firstly, we consider the case $i_1=0$, under which $t_0 \ne t_1$. 
Let $t_0 < s \le t_1$ be the turn in which the game $\GG^0_{t_1}$ was initialised, so that $B^0_{s-1} = \emptyset$. As $\GG^0_{t_0}$ must have terminated before turn $s$, we obtain, by  Proposition~\ref{prop: ga prop}\ref{prop: ga prop, item: dead front}, that $\area^0_{t_0} \subseteq (-\infty, \rho^0_{s-1}]$.
In contrast, $\area^0_{t_1} \subset \trap(B^0_s)$ by Proposition~\ref{prop: ga prop}\ref{prop: ga prop, item: future ga}, and $\trap(B^0_s) \subset \Z \times [\rho^0_{s}, \infty) = [\rho^0_{s-1}+1, \infty)$. Thus, the two play areas are disjoint.

We are left with the case $i_0\ne i_1$. Clearly, areas of non-adjacent fronts cannot intersect, so we assume that $i_1=1$. The case $i_1=-1$ follows by similar arguments. Observe that at turn $1$, when the four starting games are initialised, their initial occupied sets generate disjoint trapezoids, and hence they have disjoint play areas.
Next, let $s \le t_1$ be the turn in which $\GG^1_{t_1}$ was initialised, so that $B^1_{s-1} = 0$ and $B_{s-1} \cap L^1_{s-1} \ne \emptyset$.
Recall that, by \eqref{eq: init}, $B^1_s = O \setminus \FF_s$, where
\[O := L^1_{s} \cap \left(\big(B^{2}_{s-1} + \{\theta^1, \theta^{0,1}\}\big) \cup \big(B^{0}_{s-1} + \{\theta^1, \theta^{1,2}\}\big)  \right).\]
By Proposition~\ref{prop: ga prop}\ref{prop: ga prop, item: future ga} we have $\trap(B^1_{t_1})\subset \trap(B^1_{s})\subset\trap(O)$. To complete the proof it would thus suffice to show that $\area^0_{t_0} \cap \trap(O) = \emptyset$.

Indeed, by Proposition~\ref{prop: ga prop}\ref{prop: ga prop, past rectangle} and~\eqref{eq: fire-rad-def} we have
$\area^0_{s-1}\subset [-\rho^3_{s-1}, \rho^1_{s-1}]\times \Z$. This latter set is disjoint from $\trap(O)$, as $O \subset \{\rad^1_{s}\} \times \Z=\{\rad^1_{s-1}+1\} \times \Z$. Hence, if $t_0 \le s-1$, then $\area^0_{t_0} \subset \area^0_{s-1}$ and we are done.
Otherwise, if $t_0>s-1$, we have $\area^0_{t_0}\setminus \area^0_{s-1} \subset \trap(B^0_{s-1})$, by Proposition~\ref{prop: ga prop}\ref{prop: ga prop, item: future ga}. However, $\trap(B^0_{s-1})$ is disjoint from $\trap(O)$ by Claim~\ref{clm: disjoint trap}. The proposition follows.
\end{proof}

The second proposition states that the amortised potential of an advancing front grows by at least $2$.
\begin{prop}\label{prop: pl: Phi grows}
$\Delta \phi^i_t + \Delta |F^i_t| + \Delta d_t^i \ge 2\Delta \rho^i_t$.
\end{prop}

\begin{proof}
If $\Delta \rho^i_t = 0$, then, by~\eqref{eq: fire-rad-def}, we have $B^i_{t-1} = \emptyset$. By Proposition~\ref{prop: no fire -> no ptl no dbt}, this implies that $\phi^i_{t-1}(\seg^i_{t}) = 0$ and $d^i_{t-1}(\seg^i_{t}) = 0$. As $\Delta |F_t^i(\seg^i_t)|\ge 0$, we obtain $\Delta \phi_t^i(\seg^i_t) + \Delta |F_t^i(\seg^i_t)| + \Delta d_t^i(\seg^i_t) \ge 0$.

We therefore assume $\Delta \rho^i_t = 1$. We consider separately the cases $a^i_{t-1} =1$ and $a^i_{t-1} =0$. If $a^i_{t-1} =1$ then $\GG^i_{t-1}$ is active and takes a step according to its $\Zhp$ strategy, in which case, the proposition follows directly from Proposition~\ref{prop: Phi grows by 2, quarter}.
If $a^i_{t-1} = 0$, we must have $B_{t-1}\cap L^i_{t-1} \ne \emptyset$. 
According to the strategy, a new game starts on the trapezoid generated by
 \[O = L^i_{t} \cap \left((B^{i+1}_{t-1} + \{\theta^i, \theta^{i-1,i}\}) \cup (B^{i-1}_{t-1} + \{\theta^i, \theta^{i,i+1}\})  \right),\]
where $B^i_t \cup F^i_t=O$. Observe that, in this case $|O|\in \{2,4\}$. Since $a^i_{t-1} = 0$, we have $d^i_{t-1}=\phi^i_{t-1}=0$,
and, in addition $\Delta |F^i_t| = |F^i_t|$, by Proposition~\ref{prop: pl: ga disjoint}. Since all segments in $\seg^i_t$ are initialised to be simple, and $a^i_{t-1} = 0$, we have $\Delta \phi^i_t = |B^i_t|$. Putting all of these together, we have $\Delta \phi_t^i + \Delta |F_t^i| + \Delta d_t^i \ge |B^i_t| + |F^i_t| = |O| \ge 2$.
\end{proof}

\subsection{Proof of Theorem~\ref{thm: sub-linear win} on \texorpdfstring{$\ZZ$}{ZZ}}\label{subsection: pf of thm 1}

In this section we conclude the proof of Theorem~\ref{thm: sub-linear win} by applying Propositions~\ref{prop: pl: ga disjoint} and~\ref{prop: pl: Phi grows} to imitate the proof of~\cite{feldheim20133}*{Theorem~1}. The key observation, which follows directly from~\eqref{eq: fire-rad-def}, is that $B^i_t$ is always contained in the axis-aligned interval
\[\horint^i_t := \rho^i_t \cdot \theta^i + [-\rho^{i-1}_t, \rho^{i+1}_t] \cdot \theta^{i+1}.\]

Throughout the proof, putting several superscripts in a function serves to represent summation over their values, e.g. $\rho_t^{13} := \rho^1_t + \rho^3_t$. Omitting a superscript serves to represent summation over all possible superscripts, e.g. $\rho_t := \sum_{i\in\II} \rho^i_t$.
We start by proving several auxiliary claims.

\begin{clm}\label{clm: pl: phi+d <= fnt len}
$\phi^i_t + d^i_t \le |\horint^i_t|$.
\end{clm}
\begin{proof}
Given a set $S\subset \Z$, we write $D_t^i(S)$ for the set $S\times \{t\}$, rotated to direction $\theta^i$.
Let $S\in \seg_t(\GG^i_t)$ such that $D_t^i(S) \cap \horint^i_t \ne \emptyset$.
If $D_t^i(S)\subset \horint^i_t$, then, by Proposition~\ref{prop: d + phi <= h}, we have $\phi^i_t(S) + d^i_t(S) \le h_t = |S| = D^i(S)$. Otherwise, by Corollary~\ref{cor: small fire -> simple and debt free} and~\eqref{eq:simple phi}, $\phi^i_t(S) = |B^i_t(S)| \le |D^i(S)\cap \horint^i_t|$, and $d^i_t(S)=0$. The claim follows.
\end{proof}

By Proposition~\ref{prop: pl: Phi grows} we have $\Delta \phi^i_t + \Delta |F^i_t| + \Delta d_t^i \ge 2\Delta \rad^i_t$ for every $t\in\N$. Summing this over time-steps $1,\dots,t$ we obtain
\begin{equation}\label{eq: pl: ptl at fnt}
    \phi_t^i + d_t^i \ge \phi_1^i  - |F_t^i| +2\rho^i_t - 2\rho^i_1= 2\rho^i_t - |F_t^i|.
\end{equation}

Recalling the definition of $\horint^i_t$, observe that $|\horint^i_t| = 1+\rad^{i-1}_t + \rad^{i+1}_t$ and that $\rad_t$ is the semi-perimeter of the rectangle bounding the occupied set.
We prove two claims under the assumption $\rad_t > |F_t|$, which will be established inductively in the course of the proof.

\begin{clm}\label{clm: pl: phi+d >= rad}
If $\rad_t \ge  |F_t|+3$ then $\phi_t + d_t \ge \rad_t+3$.
\end{clm}
\begin{proof}
Summing (\ref{eq: pl: ptl at fnt}) over all $i\in\II$, and observing that $\phi_1= 2\rad_1$, yields
\[\phi_t + d_t \ge 2\rho_t - |F_t|\ge \rad_t+3\qedhere\]
\end{proof}

\begin{clm}\label{clm: pl: opposing fnts}
If $\rho_t \ge |F_t| + 3$  then $|B^{i,i+2}_t| > 0$ for every $i\in\II$.
\end{clm}
\begin{proof}
Without loss of generality we assume that $i=0$. Note that $\rad^{02}_t + \rad^{13}_t = \rad_t$. We thus consider two cases.

\textbf{Case 1.} $\rad^{13}_t \ge \frac{1}{2}\rad_t$. By Claim~\ref{clm: pl: phi+d <= fnt len} we have
\[\phi^{13}_t+d^{13}_t \le |\horint^1_t| + |\horint^3_t| = 2\rad^{02}_t+2
= 2\rad_t - 2\rad^{13}_t + 2 \le \rad_t + 2.\]
By Claim~\ref{clm: pl: phi+d >= rad} we obtain
\[\phi^{02}_t + d^{02}_t = \phi_t + d_t - \phi^{13}_t - d^{13}_t \ge \rad_t +3 - \rad_t-2 > 0,\]
from which we conclude, by Proposition~\ref{prop: no fire -> no ptl no dbt}, that $|B^{02}_t| > 0$.

\textbf{Case 2.} $\rad^{02}_t \ge \frac{1}{2}\rad_t$. We have
\[\phi^{02}_t + d^{02}_t \ge
 2\rad^{02}_t - |F^{02}_t|\ge
\rad_t-|F_t| > 0,\]
where the first inequality is obtained by applying (\ref{eq: pl: ptl at fnt}) for $i = 0,2$, the second -- by plugging in $|F_t^{02}| \le |F_t|$ and $\rad^{02}_t \ge \frac{1}{2}\rad_t$, the third -- by the assumption that $\rho_t \ge |F_t| + 3$. As before, by applying Proposition~\ref{prop: no fire -> no ptl no dbt} we conclude that $|B^{02}_t| > 0$.
\end{proof}

To control the size of the occupied set, we apply our method for avoiding dead-ends to the plane. We define $B'$, the modified strategy in the plane, as the strategy that follows the modified strategy of Proposition~\ref{prop: devirt quarter pl} in every game $\GG^i_t$ separately.

Recalling our notation $\str(G,g)$ for the range of values $q$ such that $(G,q,g)$ is won by Container, we now establish the following proposition, from which Theorem~\ref{thm: sub-linear win} will be easily obtained.

\begin{prop}\label{prop: B' winning viable plane}
$3\notin \str(\ZZ, O(h_{t}^6) + \frac{8t}{h_{t}})$.
\end{prop}

\begin{proof}
Let $t\in\N, i\in\II$. By Proposition~\ref{prop: devirt quarter pl}\ref{prop: devirt quarter pl, viable}, following the modified strategy in $\GG^i_t$ which started at some $t_0\in\N$, we are guaranteed to have 
\[|B^i_t| \le O(h_{t-t_0}^6) + \frac{2(t-t_0)}{h_{t-t_0}} + |B^i_{t_0}|\le O(h_{t}^6) + \frac{2t}{h_{t}} + |B^i_{t_0}|,\]
where $B^i_{t_0}$ is the initial occupied set of $\GG^i_t$. Using the fact that $|B^i_{t_0}| \le 4$, and that at most four games are played in parallel at every time-step, we conclude that indeed $|B'_t| \le O(h_{t}^6) + \frac{8t}{h_{t}}$.

By Proposition~\ref{prop: devirt quarter pl}\ref{prop: devirt quarter pl, fire-preserving}, $B'_t \ne \emptyset \iff B_t \ne \emptyset$, and hence to complete the proof it suffices to show that $B_t\ne\emptyset$.
To this end, we show by induction on $t$ that we have $\Delta \rho_t \ge 3$, so that the occupied set is never contained. In time-step $t=1$ four deleted vertices are needed to block a front, so that indeed $\Delta \rho_2 \ge 3$ and the basis of the induction is satisfied. Next, assume that this holds for all $t' \le t$ so that $\rad_{t}-\rad_1 =\sum_{\tau=2}^{t} \Delta \rho_\tau \ge 3t-3\ge |\FF_t|$.
By Proposition~\ref{prop: pl: ga disjoint}, we never count deleted vertices more than once towards $|F_t|$, so that $|F_t| \le |\FF_t|$, and hence $\rad_t \ge |F_t| + 3$.

Next, let $i\in\II$. By Claim~\ref{clm: pl: phi+d <= fnt len}, we have 
\[\phi^{i,i+1} + d^{i,i+1} \le |\horint^i_t| + |\horint^{i+1}_t| = \rad^{i-1}_t + \rad^{i+1}_t + \rad^{i}_t + \rad^{i+2}_t + 2 = \rad_t + 2.\]
However, by Claim~\ref{clm: pl: phi+d >= rad} we know that $\phi_t + d_t \ge\rad_t + 3$, from which we deduce that $\phi^{i+2,i+3} + d^{i+2,i+3} > 0$. By Proposition~\ref{prop: no fire -> no ptl no dbt} this implies that $|B^{i+2,i+3}_t| > 0$. On the other hand, by Claim~\ref{clm: pl: opposing fnts} we have $|B^{i,i+2}_t > 0|$. As $i\in\II$ was arbitrary, we conclude that no pair of fronts, whether adjacent or opposite can be blocked at time-step $t$. All in all, at least three fronts must contain occupied vertices, so that $\Delta \rho_{t+1} \ge 3$, as required.
\end{proof}

\begin{proof}[Proof of Theorem~\ref{thm: sub-linear win} for $\ZZ$]
As we did throughout this section, we assume that the game starts at time-step $1$ with $B_t=\{-1,0,1\}^2\setminus\{(0,0)\}$, and $\str\le 3$. Indeed, by~\cite{feldheim20133} we have $\str(\ZZ, \infty) = (3,\infty)$ so that,  to establish the theorem, by monotonicity over the spreading function, it would suffice to show  $\str(\ZZ, C_2 t^{6/7}) \subset (3,\infty)$. 

Setting $h_t := 2^{\lfloor \log_2((t+2^{14})^{1/7})\rfloor} = \Theta(t^{1/7})$, we have that $h$ satisfies \eqref{eq: doubling condition} (as in the proof of the theorem for $\Zei$ in Section~\ref{subsecion: strategy analysis}). We conclude from Proposition~\ref{prop: B' winning viable plane} that $3 \notin \str(\ZZ, C_2 t^{6/7})$ for all sufficiently large $C_2$. From monotonicity over the strength of Container, this implies $\str(\ZZ, C_2 t^{6/7}) \subset (3,\infty)$, as required.
\end{proof}

\section{Upper bounds}\label{section: upper bounds}
This section is dedicated to the proof of Theorem~\ref{thm: root-fire loses}. Recall that Theorem~\ref{thm: root-fire loses} is proven under the stronger-Spreader variation of the game, where Spreader is allowed to retain previously occupied vertices, i.e. $B_t$ is added to $B_{t-1}$, instead of replacing it (as per Remark~\ref{rmk: spreading instead of replacing}).

We first reduce the $\ZZ$ case of the theorem to that of $\Zei$, and then prove the latter. 

\subsection{A reduction to the eighth plane}\label{subsection: a reduction to Zei}

\begin{figure}[!ht]
\ffigbox
{%
  \begin{subfloatrow}
    \ffigbox[\FBwidth+0.5cm]
    {%
        \begin{tikzpicture}[scale=\sca/1.5]
        \draw[step=1cm,grayX,thin] (-1,-2) grid (2,1);
        \fill[gray] (0.5,-0.5) circle (10pt);
        \draw[black] (0,-1) rectangle (1,0);
        \foreach \x in {-1,1}
        {
        \draw[black, ultra thick] (\x+\xmargin,0+\xmargin) -- (\x+1-\xmargin,1-\xmargin);
        \draw[black, ultra thick] (\x+\xmargin,1-\xmargin) -- (\x+1-\xmargin,0+\xmargin);
        }
        \foreach \x in {0}
        {
        \draw[black] (\x+\xmargin,0+\xmargin) -- (\x+1-\xmargin,1-\xmargin);
        \draw[black] (\x+\xmargin,1-\xmargin) -- (\x+1-\xmargin,0+\xmargin);
        }
    \end{tikzpicture}
    }
    {%
      \subcaption{Step 1}%
      \label{fig: reduction, step 1}
    }
    \ffigbox[\FBwidth+0.5cm]
    {%
     \begin{tikzpicture}[scale=\smallsca/1.5]
    
    \draw[step=1cm,grayX,thin] (-22,-25) grid (25,1); 

        \fill[lightgrayX] (0,-21) rectangle (1,0); 
        \fill[gray] (0.5, -21.5) circle (10pt); 

        \fill[lightgrayX] (-20,-1) rectangle (21,0); 
        \fill[gray] (-20.5,-0.5) circle (10pt); 
        
        \fill[gray] (21.5,-0.5) circle (10pt); 
        \draw [black, densely dotted] (21.5, -0.5) -- (24, -6); 
        \draw [black, densely dotted] (21.5, -0.5) -- (24, 0); 
        \draw [black] (24, -6) rectangle (25,0); 
        
        \foreach \i in {0,1,...,20} 
        {
            \fill[lightgrayX] (\i, -\i) rectangle (\i+1, -\i-1);
        }
        \fill[gray] (21.5,-21.5) circle (10pt); 
        \draw [black, densely dotted] (21.5,-21.5) -- (24, -16); 
        \draw [black, densely dotted] (21.5,-21.5) -- (24, -25); 
        \draw [black] (24, -25) rectangle (25,-16); 
        
        \foreach \i in {0,1,...,20} 
        {
            \fill[lightgrayX] (-\i+1, -\i) rectangle (-\i, -\i-1);
        }
        \fill[gray] (-20.5,-21.5) circle (10pt); 
        
        \def\dbot{-16}
        \fill[gray] (-20.5,\dbot-0.5) circle (10pt); 
        \fill[lightgrayX] (\dbot+1, \dbot-1) rectangle (-20, \dbot);
        
        \fill[lightgrayX] (\dbot+1, \dbot-1) rectangle (\dbot, -21); 
        \fill[gray] (\dbot+0.5,-21.5) circle (10pt); 

        
        
        \foreach \x in {-22 ,..., -1} 
        {
            \draw[black, very thick] (\x+\xmargin,0+\xmargin) -- (\x+1-\xmargin,1-\xmargin);
            \draw[black,  very thick] (\x+\xmargin,1-\xmargin) -- (\x+1-\xmargin,0+\xmargin);
        }
        \foreach \x in {1 ,..., 22} 
        {
            \draw[black, very thick] (\x+\xmargin,0+\xmargin) -- (\x+1-\xmargin,1-\xmargin);
            \draw[black,  very thick] (\x+\xmargin,1-\xmargin) -- (\x+1-\xmargin,0+\xmargin);
        }
        \foreach \x in {0} 
        {
            \draw[black] (\x+\xmargin,0+\xmargin) -- (\x+1-\xmargin,1-\xmargin);
            \draw[black] (\x+\xmargin,1-\xmargin) -- (\x+1-\xmargin,0+\xmargin);
        }
        
        
        \foreach \y in {-2, -3, -4, -5, -6, -8, -9, -10, -11, -12, -14, -15 ,-16, -17, -18, -20, -21, -22, -23, -24, -25}
        {
            \draw[black] (25-1+\xmargin,\y+\xmargin) -- (25-\xmargin,\y+1-\xmargin);
            \draw[black] (25-1+\xmargin,\y+1-\xmargin) -- (25-\xmargin,\y+\xmargin);
        }
        
        \draw[black] (0,-1) rectangle (1,0); 
        
    \end{tikzpicture}
    }
    {%
      \subcaption{Step 2}%
      \label{fig: reduction, step 2}
    }
    \ffigbox[\FBwidth+0.5cm]
    {%
      \begin{tikzpicture}[scale=\smallsca/1.5]
    
    \draw[step=1cm,grayX,thin] (-26,-29) grid (25,1); 

    \fill[lightgrayX] (0,-25) rectangle (1,0); 
    \fill[gray] (0.5, -25.5) circle (10pt); 

    \fill[lightgrayX] (-24,-1) rectangle (22,0); 
    \fill[gray] (-24.5,-0.5) circle (10pt); 
    
    \fill[lightgrayX] (22,-1) rectangle (23,-2); 
    \fill[lightgrayX] (23,-2) rectangle (24,-3);
    \fill[lightgrayX] (23,-3) rectangle (24,-4);
    \fill[gray] (23.5,-4.5) circle (10pt); 
    
    \foreach \i in {0,1,...,24} 
    {
        \fill[lightgrayX] (-\i+1, -\i) rectangle (-\i, -\i-1);
    }
    \fill[gray] (-24.5,-25.5) circle (10pt); 
    
    \def\dbot{-16}
    \fill[gray] (-24.5,\dbot-0.5) circle (10pt); 
    \fill[lightgrayX] (\dbot+1, \dbot-1) rectangle (-24, \dbot);
    
    \fill[lightgrayX] (\dbot+1, \dbot-1) rectangle (\dbot, -25); 
    \fill[gray] (\dbot+0.5,-25.5) circle (10pt); 

    \foreach \i in {0,1,...,23} 
    {
        \fill[lightgrayX] (\i, -\i) rectangle (\i+1, -\i-1);
    }
    \fill[lightgrayX] (23, -24) rectangle (24, -25);
    
    \fill[gray] (23.5,-25.5) circle (10pt); 
    
    \foreach \x in {-26 ,..., -1} 
    {
        \draw[black, very thick] (\x+\xmargin,0+\xmargin) -- (\x+1-\xmargin,1-\xmargin);
        \draw[black,  very thick] (\x+\xmargin,1-\xmargin) -- (\x+1-\xmargin,0+\xmargin);
    }
    \foreach \x in {1 ,..., 24} 
    {
        \draw[black, very thick] (\x+\xmargin,0+\xmargin) -- (\x+1-\xmargin,1-\xmargin);
        \draw[black,  very thick] (\x+\xmargin,1-\xmargin) -- (\x+1-\xmargin,0+\xmargin);
    }
    \foreach \x in {0} 
    {
        \draw[black] (\x+\xmargin,0+\xmargin) -- (\x+1-\xmargin,1-\xmargin);
        \draw[black] (\x+\xmargin,1-\xmargin) -- (\x+1-\xmargin,0+\xmargin);
    }

    \foreach \y in {-26,-27} 
    {
        \draw[black, very thick] (25-1+\xmargin,\y+\xmargin) -- (25-\xmargin,\y+1-\xmargin);
        \draw[black, very thick] (25-1+\xmargin,\y+1-\xmargin) -- (25-\xmargin,\y+\xmargin);
    }
    
    \foreach \y in {-1, ..., -25} 
    {
        \draw[black] (25-1+\xmargin,\y+\xmargin) -- (25-\xmargin,\y+1-\xmargin);
        \draw[black] (25-1+\xmargin,\y+1-\xmargin) -- (25-\xmargin,\y+\xmargin);
    }
    
    \draw[black] (0,-1) rectangle (1,0); 
    \end{tikzpicture}
    }
    {
      \subcaption{Step 2.5}%
      \label{fig: reduction, step 2.5}
    }

  \end{subfloatrow}
  
  \begin{subfloatrow}
    \ffigbox[\FBwidth+0.5cm]
    {%
        \begin{tikzpicture}[scale=\sca/2]
        \fill [verylightgrayX] (-4.5,0) rectangle (10,-10.5);

        
    \fill [gray] (10.2,0.2) rectangle (10.4,-4.2); 
    \fill [black] (10.2,-4.2) rectangle (10.4,-10.2); 
    \fill [black] (-4.6,0.2) rectangle (10.4,0.4); 

        \draw[lightgrayX, thick] (6,0) -- (10,0); 
        \draw[lightgrayX, thick] (6,0) -- (10,-4); 
        \draw[lightgrayX, thick] (10,-4) -- (10,-10); 
        \fill[gray] (10,-10) circle (3pt); 
        
        \draw[lightgrayX, thick] (6,0) -- (-4,0); 
        \draw[lightgrayX, thick] (3.5,-2.5) -- (-4,-2.5); 
        \draw[lightgrayX, thick] (1,-5) -- (-4,-5);  
        \draw[lightgrayX, thick] (-1.5,-7.5) -- (-4,-7.5); 
        \draw[lightgrayX, thick] (6,0) -- (-4,-10); 
        
        \fill[gray] (-4,0) circle (3pt);
        \fill[gray] (-4,-2.5) circle (3pt);
        \fill[gray] (-4,-5) circle (3pt);
        \fill[gray] (-4,-7.5) circle (3pt);
        \fill[gray] (-4,-10) circle (3pt);
        
        \draw[lightgrayX, thick] (6,0) -- (6,-10); 
        \draw[lightgrayX, thick] (3.5,-2.5) -- (3.5,-10); 
        \draw[lightgrayX, thick] (1,-5) -- (1,-10);  
        \draw[lightgrayX, thick] (-1.5,-7.5) -- (-1.5,-10); 
        
        \draw[lightgrayX, thick] (10,-5) -- (8,-7);
        \draw[lightgrayX, thick] (8,-7) -- (8,-10);

        \foreach \x in {8,6,3.5,1,-1.5}
        {
         \fill[gray] (\x,-10) circle (3pt);
        }

        \draw [black, densely dotted] (-4,0) -- (-4.4, 0.1);
        \draw [black, densely dotted] (-4,0) -- (-4.4, -0.7);
        
        \draw [black, densely dotted] (-4,-2.5) -- (-4.4, -2.5+0.7);
        \draw [black, densely dotted] (-4,-2.5) -- (-4.4, -2.5-0.7);
        
        \draw [black, densely dotted] (-4,-5) -- (-4.4, -5+0.7);
        \draw [black, densely dotted] (-4,-5) -- (-4.4, -5-0.7);
        
        \draw [black, densely dotted] (-4,-7.5) -- (-4.4, -7.5+0.7);
        \draw [black, densely dotted] (-4,-7.5) -- (-4.4, -7.5-0.7);
        
        \draw [black, densely dotted] (-4,-10) -- (-4.4, -10+0.7);
        \draw [black, densely dotted] (-4,-10) -- (-4.4, -10.5);

        
        \foreach \y in {-10,...,-1}
        {
        \fill [gray] (-4.6,\y - 0.5) rectangle (-4.4,\y + 0.3);
        }
        \fill [gray] (-4.6,-0.5) rectangle (-4.4,0.2);
        
        \foreach \y in {-2.5, -5, -7.5}
        {
        \draw [black] (-4.6, \y+0.7) rectangle (-4.4, \y-0.7);
        }
        \draw [black] (-4.6, -0.7) rectangle (-4.4, 0.2);
        \draw [black] (-4.6, -10+0.7) rectangle (-4.4, -10.5);
      \end{tikzpicture}%
    }
    {%
      \subcaption{Step 3}%
      \label{fig: reduction, step 3}
    }
    \ffigbox[\FBwidth+0.5cm]
    {%
      \begin{tikzpicture}[scale=\sca/2]
        \fill [verylightgrayX] (-4.5,0) rectangle (10,-11.5-0.5);
    
        \fill [gray] (-4.6,-10 - 0.5) rectangle (-4.4,0.2); 
        \fill [gray] (10.2,0.2) rectangle (10.4,-4.2); 
        \fill [black] (10.2,-4.2) rectangle (10.4,-11.5-0.2); 
        \fill [black] (-4.6,0.2) rectangle (10.4,0.4); 
        
        \draw[lightgrayX, thick] (6,0) -- (10,0); 
        \draw[lightgrayX, thick] (6,0) -- (10,-4); 
        \draw[lightgrayX, thick] (10,-4) -- (10,-11.5); 
        \fill[gray] (10,-11.5) circle (3pt); 

        \draw[lightgrayX, thick] (6,0) -- (6,-11.5); 
        \draw[lightgrayX, thick] (3.5,-2.5) -- (3.5,-11.5); 
        \draw[lightgrayX, thick] (1,-5) -- (1,-11.5);  
        \draw[lightgrayX, thick] (-1.5,-7.5) -- (-1.5,-11.5); 
        
        \draw[lightgrayX, thick] (10,-5) -- (8,-7);
        \draw[lightgrayX, thick] (8,-7) -- (8,-11.5);
        \fill[gray] (8,-11.5) circle (3pt);
        
        
        \draw[lightgrayX, thick] (6,0) -- (-4,-10); 
        \foreach \x in {6,3.5,1,-1.5}
        {
         \fill[gray] (\x,-11.5) circle (3pt);
         \draw[lightgrayX, thick] (\x,\x-6) -- (-4,\x-6);
         \draw[lightgrayX, thick] (-4,\x-6) -- (-4.2,\x-6-0.2);
         \draw[lightgrayX, thick] (-4.2,\x-6-0.2) -- (-4.2, \x-6-1.5);
         \fill[gray] (-4.2, \x-6-1.5) circle (3pt);
        }
        
        \foreach \x in {-4}
        {
         \draw[lightgrayX, thick] (\x,\x-6) -- (-4,\x-6);
         \draw[lightgrayX, thick] (-4,\x-6) -- (-4.2,\x-6-0.2);
         \draw[lightgrayX, thick] (-4.2,\x-6-0.2) -- (-4.2, \x-6-1.5);
         \fill[gray] (-4.2, \x-6-1.5) circle (3pt);
        }

        \fill [black] (-4.6, -10.5) rectangle (-4.4, -11.7);
        
    \end{tikzpicture}
    }
    {
      \subcaption{Step 3.5}%
      \label{fig: reduction, step 3.5}
    }
     \ffigbox[\FBwidth+0.5cm]
    {%
      \begin{tikzpicture}[baseline, scale=\sca/4]
      \def\srad{-24.6}
    \fill [verylightgrayX] (-4.5,0) rectangle (10,-25);

    \fill [gray] (-4.6,-10 - 0.5) rectangle (-4.4,0.2); 
    
    \fill [gray] (10.2,0.2) rectangle (10.4,-4.2); 
    \fill [black] (10.2,-4.2) rectangle (10.4,\srad); 
    \fill [black] (-4.6,0.2) rectangle (10.4,0.4); 

    \draw[gray, very thick] (-4.6,\srad) -- (10.25,\srad); 
    
    \draw[lightgrayX, thick] (6,0) -- (10,0); 
    \draw[lightgrayX, thick] (6,0) -- (10,-4); 
    \draw[lightgrayX, thick] (10,-4) -- (10,-24); 
    \fill[gray] (10,-24) circle (3pt); 
    
    \draw[lightgrayX, thick] (6,0) -- (6,-24); 
    \draw[lightgrayX, thick] (3.5,-2.5) -- (3.5,-24); 
    \draw[lightgrayX, thick] (1,-5) -- (1,-24);  
    \draw[lightgrayX, thick] (-1.5,-7.5) -- (-1.5,-24); 
    \draw[lightgrayX, thick] (-4.2,-10.2) -- (-4.2,-24); 
    
    \draw[lightgrayX, thick] (10,-5) -- (8,-7);
    \draw[lightgrayX, thick] (8,-7) -- (8,-24);
    
    \draw[lightgrayX, thick] (-1.5,-12) -- (-2.8, -13.3);
    \draw[lightgrayX, thick] (-2.8,-13.3) -- (-2.8,-24);
    
    \draw[lightgrayX, thick] (1,-14) -- (-0.25, -15.25);
    \draw[lightgrayX, thick] (-0.25,-15.25) -- (-0.25,-24);
    
    \draw[lightgrayX, thick] (3.5,-16) -- (2.25, -17.25);
    \draw[lightgrayX, thick] (2.25,-17.25) -- (2.25,-24);
    
    \draw[lightgrayX, thick] (6,-18) -- (4.75, -19.25);
    \draw[lightgrayX, thick] (4.75,-19.25) -- (4.75,-24);
    
    \draw[lightgrayX, thick] (8,-18) -- (7, -19);
    \draw[lightgrayX, thick] (7,-19) -- (7,-24);
    
    \draw[lightgrayX, thick] (8,-18) -- (7, -19);
    \draw[lightgrayX, thick] (7,-19) -- (7,-24);
    
    \draw[lightgrayX, thick] (8,-20) -- (9, -21);
    \draw[lightgrayX, thick] (9,-21) -- (9,-24);

    \foreach \x in {-4.2,-2.8,-1.5,-0.25,1, 2.25,3.5, 4.75,6,7,8,9}
    {
    \fill[gray] (\x,-24) circle (3pt);
    }
    
    
    \draw[lightgrayX, thick] (6,0) -- (-4,-10); 
    \foreach \x in {6,3.5,1,-1.5}
    {
     \draw[lightgrayX, thick] (\x,\x-6) -- (-4,\x-6);
     \draw[lightgrayX, thick] (-4,\x-6) -- (-4.2,\x-6-0.2);
     \draw[lightgrayX, thick] (-4.2,\x-6-0.2) -- (-4.2, \x-6-1.5);
    }
    
    \foreach \x in {-4}
    {
     \draw[lightgrayX, thick] (\x,\x-6) -- (-4,\x-6);
     \draw[lightgrayX, thick] (-4,\x-6) -- (-4.2,\x-6-0.2);
     \draw[lightgrayX, thick] (-4.2,\x-6-0.2) -- (-4.2, \x-6-1.5);
    }

    \fill [black] (-4.6, -10.5) rectangle (-4.4, \srad); 
        
    \end{tikzpicture}
    }
    {%
      \subcaption{Step 4}%
      \label{fig: reduction, step 4}
    }
    
  \end{subfloatrow}
  
}
{
  \caption
  {A depiction of the winning Container strategy against Spreader with a single initial occupied vertex in $\ZZ$, as described in Section~\ref{subsection: a reduction to Zei}. The current occupied set is depicted by gray circles, while the history of the occupied set's movement -- by lighter gray squares. In~\eqref{fig: reduction, step 2} and~\eqref{fig: reduction, step 2.5} the vertices deleted to maintain an existing front are marked by thick X's, and those deleted by the implementation of the eighth plane strategy are marked by thin X's. In~\eqref{fig: reduction, step 3},~\eqref{fig: reduction, step 3.5}, and~\eqref{fig: reduction, step 4}, the same distinction is made via black and gray rectangles. In~\eqref{fig: reduction, step 2} and~\eqref{fig: reduction, step 3} dashed lines and black-outlined rectangles depict the regions on the front reachable by Spreader prior to the time-step in which Container completely blocks the front.
  }
  \label{figure: reduction}
}
\end{figure}
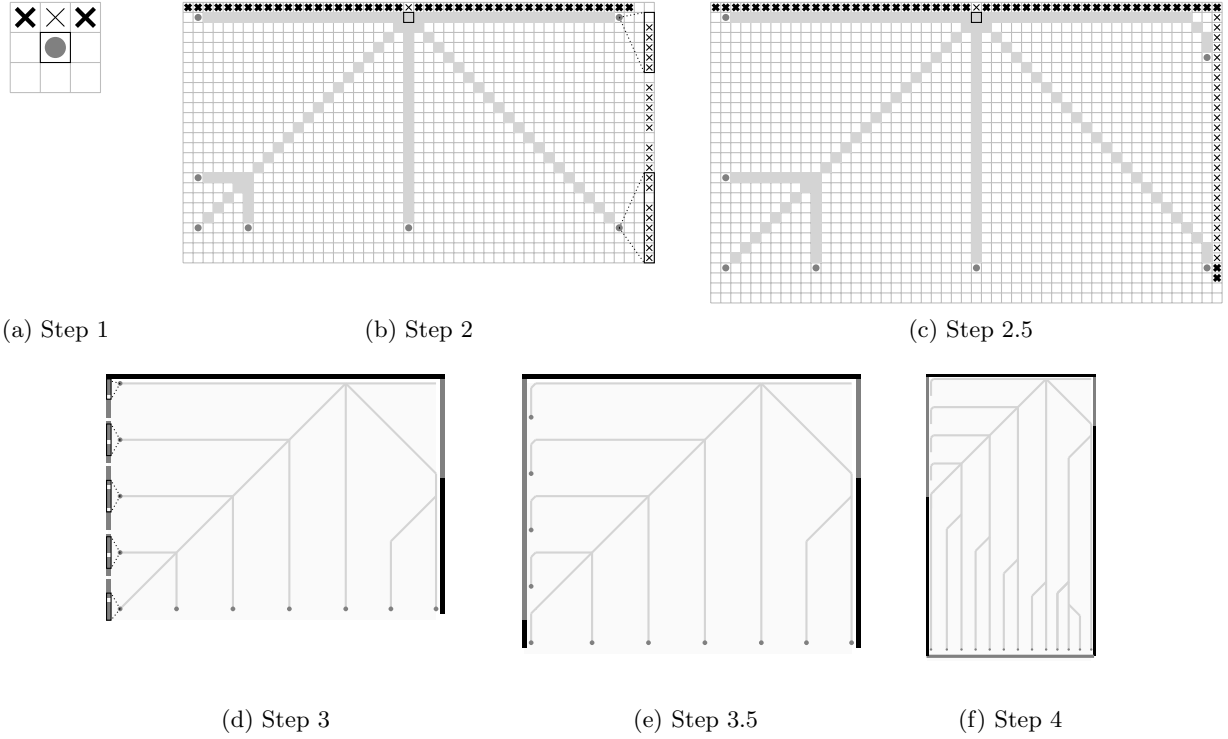

Here we reduce the $\ZZ$ case of Theorem~\ref{thm: root-fire loses} to the following proposition, a slightly stronger version of the theorem for the $\Zei$ case, which we establish in next section.
\begin{prop}\label{prop: ff win sqrt zei}
Container wins the game $(\Zei, c\sqrt{t}, 1)$, with an initial set contained in $\{(x,y)\ :\ y<r_0\} \subset \Zei$, for any $c<\frac{1}{6}$, 
by deleting all the  vertices of a single row $\{(x,y)\ :\ y=r\}\subset \Zei$, for $r$ depending only on $r_0$.
\end{prop}

\begin{proof}[Reduction of Theorem~\ref{thm: root-fire loses} to Proposition~\ref{prop: ff win sqrt zei}]
\noindent The winning Container strategy in $(\ZZ, g,3)$ is obtained by handling each front separately, containing the occupied set in an axis-parallel rectangle. This description is accompanied by Figure~\ref{figure: reduction}. The computational details are rather tedious and not too enlightening, while the arguments for the existence of sufficiently large constants are rather simple, and follow the general strategy of \cite{feldheim20133}.
Hence we leave out some details.

Let $r_0 > 0$ such that the initial occupied set is contained in $[-r_0,r_0]\times [-r_0,r_0]$. Let $r_1, r_2, r_3$ and $r_4$ be suitably chosen large constants, implicitly defined in the course of the proof.

\textbf{Step 1.} At time-steps $[0,r_1-1]$ Container deletes the horizontal segment $[-r_1,r_1]\times \{r_1\}$, restricting the occupied set to $\{(x,y) : y  < r_1\}$. Here $r_1$ is selected so that this could be achieved for an initial set of radius $r_0$.

\textbf{The remainder of the strategy.} The purpose of the strategy is to delete the boundary of the rectangle $[-r_3,r_2]\times [-r_4,r_1]$. The principle is simple: we take the fronts one by one, starting from the east, then the west and finally the south. Each front is taken by applying Proposition~\ref{prop: ff win sqrt zei} to a suitably shifted and rotated copy of $\Zei$ which guarantees the existence of a suitable $r_i$, using $1$ deletion per time-step. The remaining $2$ deletions are used to maintain the existing deleted boundaries, ensuring that Container is restricted to $[-r_3,r_2]\times [-r_4,r_1]$ at all times. 

After \textbf{Step 2,} consisting time-steps $[r_1, 2r_1 + r_2-1]$, the deleted set is  
$\{r_2\} \times [-2r_1-r_2, r_1] \cup 
[-r_2-2r_1,r_2]\times \{r_1\}$,
where the board to which the proposition is applied is a copy of $\Zei$, rotated by $90^\circ$ clockwise and shifted by $(-r_1,-r_1)$, and $r_2$ is given by $r-r_1$ where $r$ is that provided by Proposition~\ref{prop: ff win sqrt zei}.

After \textbf{Step 3,} consisting of time-steps $[2r_1+r_2 , 3r_1 + r_2 + r_3]$, the deletes set consists of 
$[-r_3,r_2]\times \{r_1\}\cup
\{-r_3\} \times [-3r_1-r_2-r_3,r_1]\cup 
\{r_2\times [-3r_1-r_2-r_3,r_1]\}$, where Proposition~\ref{prop: ff win sqrt zei} is applied to the western boundary (and thus determined $r_3$). 

Finally, \textbf{Step 4} does not really require the proposition, as after using $2$ deletions to prolong the eastern and western fronts, only $r_2+r_3$ more deletions are needed to complete the bounding rectangle, so that $r_4=4r_1+2r_2+2r_3+1$ is certainly sufficient.
\end{proof}

The remainder of the section is dedicated to proving Proposition~\ref{prop: ff win sqrt zei}.



\subsection{Container win in the eighth plane -- proof of Proposition~\ref{prop: ff win sqrt zei}}\label{subsection: container win in eighth plane}

This section is dedicated to the proof of Proposition~\ref{prop: ff win sqrt zei}. The proof is accompanied by Figure~\ref{fig:ff strategy Zei}.

\begin{figure}[!ht]
    \newcounter{mysum}
    \centering
    \begin{tikzpicture}[scale=\sca/1.5]


     \fill [lightgrayX] (0,6) rectangle (6,7);
    \fill [lightgrayX] (8,6) rectangle (19,7);
    \fill [lightgrayX] (21,6) rectangle (32,7);
    
    \foreach \x in {0,7,14,21,28,35}
    {
        \foreach \i in {0,...,5}
        {
            \setcounter{mysum}{\x+\i};
            \draw[black] (\value{mysum}+\xmargin,6+\xmargin) -- (\value{mysum}+1-\xmargin,7-\xmargin);
            \draw[black] (\value{mysum}+\xmargin,7-\xmargin) -- (\value{mysum}+1-\xmargin,6+\xmargin);
        }
        
    }
    
    \fill[gray] (0.5, 0.5) circle (10pt);
    \fill[gray] (0.5, 1.5) circle (10pt);
    \fill[gray] (0.5, 2.5) circle (10pt);
    \fill[gray] (0.5, 3.5) circle (10pt);
    
    \fill[gray] (13.5, 0.5) circle (10pt);
    \fill[gray] (13.5, 1.5) circle (10pt);
    
    \fill[gray] (26.5, 0.5) circle (10pt);
    \fill[gray] (26.5, 1.5) circle (10pt);
    \fill[gray] (26.5, 2.5) circle (10pt);

    \foreach \y in {0,...,6}
    {
    \draw[step=1cm,grayX,thin] (0,\y) grid (35+\y,\y+1); 
    }
    
    \draw[black, line width=0.2mm,densely dotted] (0.5, 3.5) -- (0,6);
    \draw[black, line width=0.2mm, densely dotted] (0.5, 3.5) -- (6,6);
    \draw[black] (0,6) rectangle (6,7);
    
    \draw[black, line width=0.2mm,densely dotted] (13.5, 1.5) -- (8,6);
    \draw[black, line width=0.2mm, densely dotted] (13.5, 1.5) -- (19,6);
    \draw[black] (8,6) rectangle (19,7);
    
    \draw[black, line width=0.2mm,densely dotted] (26.5, 2.5) -- (21,6);
    \draw[black, line width=0.2mm, densely dotted] (26.5, 2.5) -- (32,6);
    \draw[black] (21,6) rectangle (32,7);
    
    \end{tikzpicture}
    \caption{Illustration of the Container strategy in the eighth plane at time-step $H-(h+r_0)$ for $h=2$. Circles mark the occupied vertices and $\times$ mark deleted vertices; the set $D\subset L_H$ is shaded. The contribution of various elements of $B_{H-(h+r_0)}$ to $D$ is marked by dotted lines and black rectangles in $L_H$. Note that there are five vacant vertices in $L_H$ (i.e. $|L_H \setminus \FF_t| = 5$), while there are only two such vertices within $D$, these will be occupied by Container before time-step $H-r_0$. The occupation of $L_H$ will follow in the remaining $r_0$ turns.}
    \label{fig:ff strategy Zei}
\end{figure}
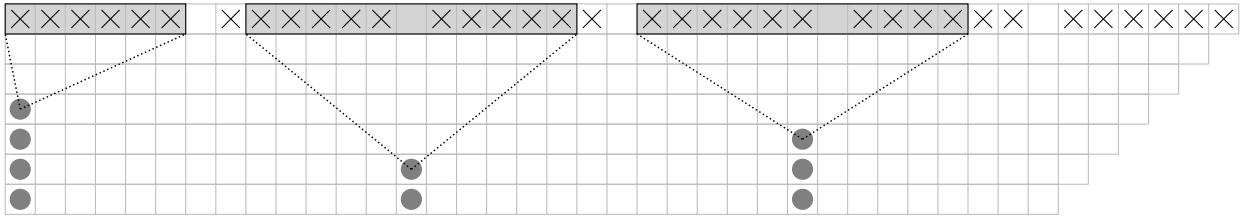

Throughout this section, let $c < \frac{1}{6}$ as per Proposition~\ref{prop: ff win sqrt zei},
and let $r_0 > 0$ be such that the initial occupied set is contained in $[0,r_0]\times [0,r_0]$. Hence, throughout the game, for all $t>0$
\begin{equation}\label{eq: space-time relation}
B_t\subset \Z\times [0,t+r_0].
\end{equation}

Let $h,H\in \N$ be a pair of constants to be specified later.
We shall delete the horizontal segment $L_H:=[0,H]\times \{H\}$, restricting the occupied set to the lower triangle bounded by it. 
Our strategy follow three steps. 

\textbf{Phase 1 -- Leaving out a regular sieve.}
Under the restriction that $H$ is divisible by $h+r_0$, denote by
\[X := \left\{\left(i\cdot \frac{H}{h+r_0}, H\right) :  0\le i\le h+r_0-1\right\},\] 
a set of vertices on $L_H$ regularly spaced at $\frac{H}{h+r_0}$ intervals.
During the first phase of our strategy, spanning across times $[0,H-(h+r_0)-1]$, Container deletes the vertices of $L_H\setminus X$ in an arbitrary order. 

\textbf{Phase 2 -- Deleting the most imminent bottleneck through which Spreader may cross $L_H$.}
During the second phase of our strategy, spanning across time-steps
$[H-(h+r_0),H-2r_0-1]$, Container first deletes, in an arbitrary order, the vertices $D:= \{u\in L_H : d(u, B_{H-(h+r_0)}) \le h+r_0\}$, namely vertices of the sieve which are of distance less than $h+r_0$ from an occupied vertex. In the remaining time-steps, arbitrary vertices of $X$ are also deleted, as much as the deletion power permits.
To see that $D$ could be deleted on time, we require the following claim, whose proof is postponed to section~\ref{subsection: bounding danger zone} below.
\begin{clm}\label{clm: ub: danger zone}
$D$ is contained in $2c\sqrt{H}$ horizontal segments of total length at most $3(h+r_0)\cdot c\sqrt{H}$.
\end{clm}
 
Since every two vertices in $X$ are at distance at least $\frac{H}{h+r_0}$ from each other, for every $d\in\N$, a horizontal segment of length $d$ can intersect at most $\left\lceil\frac{d}{H/(h+r_0)}\right\rceil \le d\cdot\frac{h+r_0}{H} + 1$ vertices in $X$.
Thus the union of $n$ segments of total length at most $d$ intersects $X$ in at most $d\cdot\frac{h+r_0}{H} + n$ vertices in $X$.

Using Claim~\ref{clm: ub: danger zone} we need consider at most $n=2c\sqrt H$ segments of total length at most $3(h+r_0)\cdot c\sqrt{H}$, so that
\[|D\cap X| \le  3(h+r_0)\cdot c\sqrt{H}\cdot \frac{h+r_0}{H}+2c\sqrt{H} = c\sqrt{H}\left(2+\frac{3(h+r_0)^2}{H}\right).\]
To verify that this is indeed less than or equal to $h-r_0$, observe that, 
writing $\Bar h:= \frac{h+r_0}{\sqrt{H}}$, which  will be chosen to satisfy $\Bar h \le \sqrt H$, it would suffice to show that
\[3c\sqrt{H}(1+ \Bar h^2) \le \Bar h\sqrt{H}-2r_0.\]
Isolating $c$, it would suffice to show that we may chose $H$ and $\Bar h$ that satisfy that \vspace{-2pt}
\begin{equation}\label{eq: c is small enough}
c \le \frac{\Bar h}{3(1+\Bar h^2)} - \frac{2r_0}{3\sqrt{H}(1+\Bar h^2)}.
\end{equation}
\vspace{-2pt}
Indeed, using our assumption $c < \max_{\Bar h\geq 0}\left\{\frac{\Bar h}{3(1+\Bar h^2)}\right\} = \frac{1}{6}$ and the fact that we may pick arbitrarily large $H$, a choice of such $\Bar h$ and $H$ satisfying all previous requirements is possible. 

\textbf{Phase 3 -- Deleting the remaining vertices of $L_H$.}
Finally, during the third phase of our strategy, spanning across time-steps $[H-2r_0,H]$, we delete the remaining vertices of $L_H$ in an arbitrary order, taking advantage of the fact that, as these are not in $D$, they are too far from occupied vertices to be claimed by Spreader within this time-frame. By time-step $H$, Container will have deleted all $H+1$ vertices of $L_H$.

\subsubsection{Bounding \texorpdfstring{$D$}{D} -- proof of Claim~\ref{clm: ub: danger zone}}\label{subsection: bounding danger zone}
Let $t\in\N$, and, given $B' \subset B_{t}$, denote $D(B') := \{u\in L_H : d(u,B') \le h+r_0\}$, and observe that 
$D=D(B_{H-(h+r_0)})$, while
$D(B_{H-2h-3r_0-1})=\emptyset$, since $d(B_{H-2h-3r_0-1},L_H)>h+r_0$. Thus, we define $\tilde B:=B_{H-h-r_0} \setminus B_{H-2h-3r_0-1}$ which satisfies $D=D(\tilde B)$. Using $|\Delta B_t| = c\sqrt{t}$, we have
\begin{equation}\label{eq: tilde B bounded}
  |\tilde B| = \sum_{k=H-2h-3r_0}^{H-h-r_0}|\Delta B_{k}| \le (h+2r_0+1) |\Delta B_{H-(h+r_0)}| \le \frac{3}{2}(h+r_0) c\sqrt{H}.
\end{equation}

Denote by $z\rightsquigarrow x$ the transitive closure of the relation
\[(\exists s\le H-(h+r_0)\,:\,z\in B_{s-1},x\in \Delta B_s, d(z,x)=1),\]
so that $z\rightsquigarrow x$ if the occupation of $x$ could have been instigated by the occupation of $z$. Further denote the leftmost $(H-2h-3r_0-1)$-ancestor of a vertex $x\in \tilde B$ with respect to this relation by $z(x):=\min \{ z\in B_{H-2h-3r_0-1} : z\rightsquigarrow x\}$.

We partition $\tilde B$ into equivalence classes of vertices $x\in\tilde B$ sharing the same $z(x)$.
To this end, denote
\[Z:=\{z\in B_{H-2h-3r_0-1}  \,:\,\exists x\in\tilde B\text{ s.t. } z=z(x)\} \quad \text{and} \quad \Sigma(z):=\{x\in \tilde B\,:\, z(x)=z\}.\]

Because $\displaystyle\tilde B=\bigcupdot_{z\in Z}\Sigma(z)$, and as for every $u\in\tilde B$ the set $D(\{u\})$ is either empty, or a horizontal segment of length $2(h+r_0)+1$, we obtain that
$D$ is contained in at most $|D| / (2(h+r_0)+1)$ horizontal segments. Using \eqref{eq: tilde B bounded}, we are left with showing
$|D| \le 2|\tilde B|$, which we will obtain by showing that for every $z\in Z$ we have
\begin{equation}\label{eq: D(sigma) le 2*sigma}
  |D(\Sigma(z))| \le 2|\Sigma(z)|.
\end{equation}

To prove~\eqref{eq: D(sigma) le 2*sigma}, let $z\in Z$ and denote by $v_0 = (x_0,y_0)$ and $v_1=(x_1,y_1)$, the leftmost and rightmost elements of $\Sigma(z)$ whose distance from $L_H$ is at most $h+r_0$, respectively.
Write $P_{0}$ and $P_{1}$ for the associated paths from $z$ to $v_0$ and to $v_1$ respectively.
Observing that 
\[|D(\Sigma(z))| \le  \big|[x_0-(h+r_0), x_1+(h+r_0)]\times \{H\}\big| \le x_1-x_0+2(h+r_0)+1,\]
and $P_0,P_1\subset \Sigma(z)$, it would suffice to show that \vspace{-2pt}
\begin{equation}\label{eq: many ancestors}
  2|P_0\cup P_1|\ge x_1-x_0+2(h+r_0) +1.
\end{equation}
\vspace{-2pt}
To see this, assume without loss of generality that $|P_0| \le |P_1|$. Observe that since $d(v_0, L_H) \le h+r_0$ and $d(z, L_H) \ge d(B_{H-2h-3r_0-1}, L_H) \ge 2(h+r_0)+1$ (by~\eqref{eq: space-time relation}),  we have $|P_0| > h+r_0$. 
Next, observe that \vspace{-2pt}
\[x_1 - x_0 \le 2d(v_1, P_0 \cap P_1)\le 2|P_1 \setminus P_0|,\]
\vspace{-2pt}
where the first inequality uses $|P_0| \le |P_1|$ and the triangle inequality, and the second inequality uses discrete continuity of $P_1$. As $|P_0\cup P_1|\ge |P_0|+|P_1\setminus P_0|$, putting all of these together, \eqref{eq: many ancestors} follows. \qed
\section*{Acknowledgment}

The authors wish to thank Rani Hod for his insightful remarks concerning the presentation, and for many useful discussions.

\bibliography{Bib}
\bibliographystyle{abbrv}

\pagebreak
\section{Table of Notations}

\begin{center}
\begin{tabular}{ |p{2.5cm}|p{12.5cm}| p{0.8cm}|} 
 \hline
 \textbf{Notation} & \textbf{Description} & Loc.\\ 
 \hline
\multicolumn{3}{|c|}{\textbf{Section~\ref{section: introduction}}}\\
 \hline
    $\ZZ$ & \textbf{Plane.} $\Z^2$ with strong connectivity. &\S\ref{section: introduction} \\
    $\Zei$ & \textbf{Eighth plane.} Upper half of the positive quadrant of $\ZZ$.&\S\ref{section: introduction} \\
    $g(t)$ & \textbf{Spreading function.} Function controlling the size of the occupied set. &\S\ref{section: introduction}\\
    $q(G,g)$ & $q(G,g):=\{q\ :\ \text{$(G,q,g)$ is Container win for every finite $B_0$}\}$. &\S\ref{section: introduction}\\
    $\FF_t$ &  Set of deleted vertices up to time-step $t$. &\S\ref{section: introduction}\\
    $G_t$ &  Induced graph after the deletion of the vertices of $\FF_t$. &\S\ref{section: introduction}\\
    $B_t$ &  Occupied set at time-step $t$. &\S\ref{section: introduction}\\
     \hline
\multicolumn{3}{|c|}{\textbf{Section~\ref{section: eighth plane}}}\\
 \hline
    $h_t$ and $H_t$ &  $h_t$ is the size of the segments at time-step $t$. $H_t := 11\str h_t^2$. & \S\ref{subsection: Zei strategy} \\
    $\tms_i$ &  $\tms_i:=\{t\in\N : h_t/h_{t-1} = i \}$.& \S \ref{subsection: Zei strategy} \\    
    $\seg_t$ &  \textbf{Segments.} $\seg_t:= \{\{0,\dots,h_t-1\} + h_t k : k\in\Z\}$.
    & \S \ref{subsection: Zei strategy}\\
    $\anc_s(S)$ &  $\anc_s(S):=\{S' \in \seg_s : S'\subset S\}$.
    & \S \ref{subsection: Zei strategy}\\
    $B_t(S)$ &  $B_t(S)$ is the set of occupied vertices in $S\times\{t\}$. &\eqref{eq: Bt def eighth plane}\\
    $t$-path & An upwards path in $G_t$. & \S\ref{subsection: Zei strategy}\\
    
    $B_t^\ell$ &  Set of vertices from which a $t$-path of length $\ell$ starts. & \S\ref{subsection: fire simulation}\\
    $\dispath{t}{\ell}(A)$ & Size of largest collection of disjoint $t$-paths of length $\ell$ starting in $B_t(A)$. 
    & \S\ref{subsection: fire simulation}\\
    $\dis_t$ & \textbf{Disrupted Segments.} $\dis_t:= \Big\{S\in\seg_t : (\FF_t\setminus\FF_{t-1}) \cap \left(S\times [t,t+h_{t}-1]\right) \ne \emptyset\Big\}$. &\S\ref{subsection: chi} \\
    $I^\ell_s(S)$ and $I_t(S)$  & $I^\ell_s(S)$ is the minimal interval around $S$ with many disjoint $s$-paths. $I_t(S) := I_t^{\H_t}(S)$. &\eqref{eq: alert-interval-def} \\
    $\spr_t(S)$ &  Indicator for the spreading status of $S$ at time-step $t$. $\spr_t(S):= \ind \{\ct_t(S) > 0\}$.&\eqref{eq: nsim-def} \\
    $\ct_t(S)$  & \textbf{Consolidation countdown.} Number of turns until $S$ becomes simulative. &\eqref{eq: countdown-def}\\
    $\ell_t(S)$ & \textbf{Look ahead.} Vertical distance between the front and the simulated fire in $S$.&\eqref{eq: look-ahead-def}\\
    $L_t(S)$ & \textbf{Look ahead region.} $L_t(S) := S\times [t,t+\ell_t(S)]$. &\eqref{eq: look-ahead-def}\\
    $F_t(S)$ and $f_t(S)$ &  $F_t(S)$ is the set of deleted vertices associated with $S$ up to time-step $t$. $f_t(S) := |F_t(S)|$.&\eqref{eq: firefighters def}\\
    $\rng_t(S)$ & \textbf{Range.} Number of endpoints of $t$-paths emanating from $B_t(S)$  in $S\times \{t+\ell_t(S)\}$. &\eqref{eq: r def}\\
    $\phi_t(S)$ &  The size of the simulated fire in $S$ at time-step $t$. &\eqref{eq: phi def} \\
    $\lambda_t(S)$ & Nearest segment left of $S$ with a positive change in $\ppo_t$. &\eqref{eq: lambda def} \\
    $d_t(S)$ & \textbf{Debt.} The debt $S$ holds towards the desired evolution of $\phi$ at time-step $t$. &\eqref{eq: debt def} \\
    $\Phi_t(S)$ & \textbf{Potential.} $\Phi_t(S):= \phi_t(S) + |F_t(S)|$.&\S\ref{subsection: mu} \\
    \hline
\multicolumn{2}{|c|}{\textbf{Section~\ref{section: directed half plane}}}\\
\hline
     $\Zhp$ & \textbf{Directed half-plane.} Upper half-plane with edges directed upwards and diagonally. &\S\ref{subsection: Zei strategy}\\
     $\area^{\GG}_t$ & \textbf{Play area.} $\bigcup_{t'=0}^t B_{t'} \cup F_{t'}$. &\eqref{eq: play area}\\
     $\trap(D)$ & \textbf{Infinite Trapezoid.} Convex hull in $\Z^2$ of $D+\{k(1,1),k(-1,1)  :\ k\in \N\}$.& \S\ref{section: directed half plane}.0\\
\hline
\multicolumn{2}{|c|}{\textbf{Section~\ref{section: plane}}}\\
\hline
    $\theta^i$ and $\theta^{i,i+1}$ & Cardinal and secondary directions in $\Z^2$. &\S\ref{subsection: plane}\\
    $L^i_t$ and $L^i(d)$ & \textbf{Front-line.} $L^i_t := L^i(\rad^i_t)$, where $L^i(d):=d\theta^i + \Z\theta^{i+1}$. &\eqref{eq: fire-rad-def} \\
    $\rad^i_{t}$ & \textbf{Occupation radius.} $\rad^i_{t}:=\min\left\{r\ :\ L^i(r) \cap \bigcup_{t'=0}^{t-1}B_{t'}=\emptyset\right\}$&\eqref{eq: fire-rad-def} \\
    $\GG^i_t$ & Half-plane game played in direction $i$ at time-step $t$ employed by the plane strategy.& \S\ref{subsection: plane}\\
     \hline
\end{tabular}
\end{center}

\end{document}